\documentclass[10pt,reqno]{amsart}
\include{formatAndDefs}
\usepackage{amsfonts,amssymb,amsmath}

\parskip        \baselineskip
\usepackage{aliascnt}
\usepackage[mathscr]{eucal}
\usepackage{graphicx}
\usepackage{latexsym}
\usepackage{psfrag}
\usepackage{epsfig}
\usepackage[utf8]{inputenc}
\usepackage[english]{babel}
\usepackage{tikz}
\usepackage[foot]{amsaddr}
\usetikzlibrary{decorations.pathreplacing}

\newcommand{\suchthat}{\;\ifnum\currentgrouptype=16 \middle\fi|\;}

\numberwithin{equation}{section}
\newtheorem{thm}{Theorem}
\numberwithin{thm}{section}

\newaliascnt{lemma}{thm}
\newtheorem{lem}[lemma]{Lemma}
\aliascntresetthe{lemma}

\newaliascnt{proposition}{thm}
\newtheorem{prop}[proposition]{Proposition}
\aliascntresetthe{proposition}

\newaliascnt{corollary}{thm}
\newtheorem{corollary}[corollary]{Corollary}
\aliascntresetthe{corollary}

\newaliascnt{definition}{thm}
\newtheorem{mydef}[definition]{Definition}
\aliascntresetthe{definition}

\newaliascnt{remark}{thm}
\newtheorem{remark}[remark]{Remark}
\aliascntresetthe{remark}

\usepackage[colorlinks=true]{hyperref}
\def\R{{\mathbb{R} }}

\begin{document}

\title{Stabilization of the non-homogeneous Navier-Stokes equations in a 2d channel}

\date{\today}

\author{Sourav Mitra}
\thanks{{Acknowledgments}: The author wishes to thank the ANR project ANR-15-CE40-0010 IFSMACS as well as the Indo-French Centre for Applied Mathematics (IFCAM) for the funding provided during this work.}
\address{Sourav Mitra, Institut de Math\'ematiques de Toulouse; UMR5219; Universit\'e de Toulouse; CNRS; UPS IMT, F-31062 Toulouse Cedex 9, France.}
\email{Sourav.Mitra@math.univ-toulouse.fr}

\begin{abstract}
	In this article we study the local stabilization of the non-homogeneous Navier- Stokes equations in a 2d channel around Poiseuille flow. We design a feedback control of the velocity which acts on the inflow boundary of the domain such that both the fluid velocity and density are stabilized around  Poiseuille flow  provided the initial density is given by a constant added with a perturbation, such that the perturbation is supported away from the lateral boundary of the channel. Moreover the feedback control operator we construct has finite dimensional range.
\end{abstract}
\maketitle
\section{Introduction}
\subsection{Settings of the problem}
We are interested in stabilizing the density dependent Navier-Stokes equations around some stationary state $(\rho_{s},v_{s})$ (where $(\rho_{s},v_{s},p_{s})$ is a stationary solution) in a two dimensional channel $\Omega$. For that we will use an appropriate boundary control $u_{c}$ acting on the velocity in the inflow part of the boundary $\partial{\Omega}$.\\
Let $d$ be a positive constant. Throughout this article we will use the following notations (see Figure 1.)
\begin{equation}\label{domain}
\begin{array}{l}
\Omega =(0,d)\times (0,1),\quad \Gamma=\partial \Omega,\quad 
Q_{T}=\Omega \times (0,T),\quad \Sigma_{T}=\Gamma \times (0,T)\quad \mathrm{for}\quad 0<T\leqslant \infty. 
\end{array}
\end{equation} 
The unit outward normal to the boundary $\Gamma$ is denoted by $n.$ The velocity, density and pressure of the fluid are denoted respectively by ${v},$ ${\rho}$ and $p.$ The viscosity $\nu>0$ of the fluid is a positive constant. We consider the following control system
\begin{equation}\label{1-3}
\left\{
\begin{array}{ll}
\displaystyle
\frac{\partial \rho}{\partial t} + \mbox{div}(\rho {v})=0&\quad \mbox{in}\quad Q_{\infty},
\vspace{1.mm}\\
\displaystyle
\rho=\rho_s &\quad \mbox{on}\quad \{(x,t)\in\Sigma_{\infty}\suchthat (v(x,t)\cdot n(x))<0\},
\vspace{1.mm}\\
\displaystyle
\rho(x,0)=\rho_{s}+\rho_0&\quad\mbox{in}\quad\Omega,
\vspace{1.mm}\\
\displaystyle
\rho \left(\frac{\partial {v}}{\partial t} +({v}\cdot\nabla){v} \right) - \nu \Delta {v} + \nabla p = 0&\quad \mbox{in}\quad Q_{\infty},
\vspace{1.mm}\\
\displaystyle
\mbox{div}({v}) = 0&\quad \mbox{in}\quad Q_{\infty},
\vspace{1.mm}\\
\displaystyle
{v}=v_{s} + u_{c}\chi_{\Gamma_{c}} &\quad \mbox{on}\quad \Sigma_{\infty},
\vspace{1.mm}\\
\displaystyle
{v}(x,0)={v}_s+{v}_{0}&\quad\mbox{in}\quad\Omega,
\end{array} \right.
\end{equation}
where $u_{c}\chi_{\Gamma_{c}}$ is a control function for the velocity $v$ with $\chi_{\Gamma_{c}}$ denoting the characteristics function of a set $\Gamma_{c}$ which is compactly supported on $\Gamma.$ The set $\Gamma_{c}$ will be precisely defined shortly afterwards. The equation \eqref{1-3}$_{1}$ is the mass balance equation and \eqref{1-3}$_{4}$ is the momentum balance equation. The triplet $(\rho_{s},v_{s},p_{s})$ is the Poiseuille profile defined as follows 
\begin{equation}\label{vs}
\begin{array}{l}
\rho_{s}(x_{1},x_{2})=1,\quad{v}_s(x_{1},x_{2})= \
\begin{bmatrix}
x_2(1-x_2)
\\
0
\end{bmatrix},\quad p_{s}=-2\nu x_{1},\quad\mbox{in}\quad\Omega.
\end{array}
\end{equation}
Observe that $(\rho_{s},v_{s},p_{s})$ (given by \eqref{vs}) is a stationary solution of the Navier-Stokes equations \eqref{1-3}. We remark that in the definition \eqref{vs} of the Poiseuille profile we can choose $\rho_{s}$ to be any positive constant in place of one up to modifying $p_{s}$ accordingly. Also in the definition \eqref{vs} one can consider $v_{s}=(\alpha x_{2}(1-x_{2}),0),$ for a positive constant $\alpha>0.$ 
The strategy and results of our analysis apply for any constant $\rho_{s}>0$ and $\alpha>0.$\\
 The aim of this article is to determine feedback boundary control ${u}_{c}$ (the control of the velocity) such that the solution $(\rho, v)$ of the controlled system
is exponentially stable around the stationary solution $(\rho_{s},v_{s})$ provided the perturbation $(\rho_{0},v_{0})$ of the steady state $(\rho_{s},v_{s})$ is sufficiently small (in some suitable norm).\\ 
In view of the stationary profile \eqref{vs}, it is natural to control the inflow part of the boundary, $i.e.$ we will consider the control function $u_{c}$ supported on
\begin{equation}\label{inflow}
\begin{array}{l}
\Gamma_{in}=\{x\in \Gamma \mid (v_{s}\cdot n)(x) < 0\}=\{0\}\times (0,1).
\end{array}
\end{equation}
In fact we do slightly more and control on some open subset $\Gamma_{c}$ of $\Gamma_{in}.$ We consider $\Gamma_{c}$ of the following form
\begin{equation}\label{dgc}
\begin{array}{l}
\Gamma_{c}=\{0\}\times (L,1-L)\subset\Gamma_{in},
\end{array}
\end{equation}
 for some fixed $0<L<\frac{1}{2}.$
 \begin{remark}
 We consider the control zone of the form \eqref{dgc} to simplify the notations. In fact our analysis allows to consider any subset $\{0\}\times (A,B)$ ($0<A<B<1$) of $\Gamma_{in}$ as the control zone.
 \end{remark}
 To state our results precisely, we introduce some appropriate functional spaces. 
 \subsection{Functional framework for the Naviers-Stokes equation }\label{funcframe}
 Let $ H^s (\Omega;\mathbb{R}^N)\,\mbox{and}\,L^2(\Omega;\mathbb{R}^N)$
 denote the vector valued Sobolev spaces. If it is clear from the context, we may simply denote these spaces by $H^{s}(\Omega)$ and $L^{2}(\Omega)$ both for scalar and vector valued functions.
 The same notational conventions will be used for the trace spaces. 
 We now introduce different spaces of divergence free functions and some suitable spaces of boundary data:
 \begin{equation}\nonumber
 \begin{array}{l}
 \displaystyle
 {V}^{s}(\Omega) = \{ {y} \in {H}^s(\Omega;\mathbb{R}^{2})\suchthat \mbox{div} {y} =0 \quad \mbox{in} \quad \Omega \}\quad \mbox{for} \quad s\geqslant 0,
 \vspace{1.mm}\\
 {V}^{s}_{n} (\Omega) = \{{y} \in {H}^{s} (\Omega;\mathbb{R}^{2})\suchthat \mbox{div}{y}=0\quad \mbox{in} \quad \Omega, \quad {y}\cdot {n} =0\quad \mbox{on}\quad \Gamma\}\quad \mbox{for} \quad s\geqslant 0,
 \vspace{1.mm}\\
 {V}^{s}_{0}(\Omega)=\{{y} \in {H}^{s}(\Omega;\mathbb{R}^{2}) \suchthat  \mbox{div}{y}=0\quad \mbox{in} \quad \Omega,\quad {y}=0 \quad \mbox{on} \quad \Gamma \}\quad \mbox{for} \quad s\in(\frac{1}{2},\frac{3}{2}),
 \vspace{1.mm}\\
 \displaystyle{V}^{s} (\Gamma)= \{{y} \in {H}^{s} (\Gamma;\mathbb{R}^{2})\suchthat \int\limits_{\Gamma}y\cdot n\,dx=0 \}\quad\mbox{for}\quad s\geqslant 0.
 \end{array}
 \end{equation}
 The spaces ${V}^{s}(\Omega)$ and ${V}^{s}(\Gamma)$ are respectively equipped with the usual norms of ${H}^{s}(\Omega)$ and ${H}^{s}(\Gamma),$ which will be denoted by $\|\cdot \|_{{V}^{s}(\Omega)}$ and $\|\cdot \|_{{V}^{s}(\Gamma)}.$\\ 
 From now onwards we will identify the space $V^{0}_{n}(\Omega)$ with its dual.\\
 For $0<T\leqslant\infty$ let us introduce the following functional spaces adapted to deal with functions of the time and space variables.
 \begin{equation}\nonumber
 \begin{array}{l}
 {V}^{s,\tau}(Q_{T}) = {{H}^\tau}(0,T;{V}^0 (\Omega))\cap { L}^2 (0,T;{V}^{s}(\Omega)) \quad \mbox{for} \quad s ,  \tau \geq 0, 
 \vspace{1.mm}\\	
 {V}^{s,\tau}(\Sigma_{T}) = {{H}^\tau}(0,T;{V}^0 (\Gamma))\cap {L}^2 (0,T;{V}^s(\Gamma)) \quad \mbox{for} \quad s ,  \tau \geq 0.
 \end{array}
 \end{equation}
 We also fix the convention that for any two Banach spaces $\mathcal{X}$ and $\mathcal{Y},$ the product space $\mathcal{X}\times\mathcal{Y}$ is endowed with the norm
 $$\forall\,\, (x,y)\in\mathcal{X}\times\mathcal{Y},\,\,\|({x},{y})\|_{\mathcal{X}\times\mathcal{Y}}=\|{x}\|_{\mathcal{X}}+\|{y}\|_{\mathcal{Y}},$$
 where $\|.\|_{\mathcal{X}}$ and $\|.\|_{\mathcal{Y}}$ denotes the norms in the corresponding spaces.\\
 \subsection{The main result} 
 We now precisely state our main result in form of the following theorem.
 \begin{thm}\label{main}
 	Let $\beta>0,$ $A_{1}\in (0,\frac{1}{2}).$ There exist a constant $\delta>0$ such that for all $(\rho_{0},{v}_{0})\in L^{\infty}(\Omega)\times{V}^{1}_{0}(\Omega)$ satisfying
  \begin{equation}\label{1-4}
  \begin{array}{l}
  \mbox{supp}(\rho_{0})\subset [0,d]\times (A_{1},1-A_{1}),
  \end{array}
  \end{equation}
  and 
  \begin{equation}\nonumber
  \begin{array}{l}
  \| (\rho_{0},{v}_{0})\|_{L^{\infty}(\Omega)\times{V}^{1}_{0}(\Omega)}\leqslant \delta,
  \end{array}
  \end{equation}
 	there exists a control ${u}_{c}\in H^{1}(0,\infty;C^{\infty}(\overline{\Gamma}_{c})),$ for which the system \eqref{1-3} admits a solution  
 	$$(\rho,v)\in  L^{\infty}(Q_{\infty})\times{V}^{2,1}(Q_{\infty}),$$ satisfying the following stabilization requirement
 	\begin{equation}\label{1-6}
 	\begin{array}{l}
 	\| e^{\beta t}(\rho-\rho_{s},{v}-v_{s}) \|_{L^{\infty}(Q_{\infty})\times{V}^{2,1}(Q_{\infty})} \leqslant C \|  ({\rho_{0}},{v}_{0})\|_{L^{\infty}(\Omega)\times{V}^{1}_{0}(\Omega)},    
 	\end{array}
 	\end{equation}
 	 for some constant $C>0.$ Moreover, $\rho=\rho_{s}$ for $t$ sufficiently large.
 \end{thm}
We now make precise the structure of the control function $u_{c}$ we are going to construct. We will show the existence of a natural number $N_{c},$ and a family $$\{{g_{j}}\suchthat 1\leqslant j\leqslant N_{c}\},$$ of smooth functions supported on $\Gamma_{c}$ such that the control ${u}_{c}$ acting on the velocity is given as follows
 \begin{equation}\label{findcon}
 \begin{array}{l}
 {u}_{c}(x,t)=e^{-\beta t}\sum\limits_{j=1}^{N_{c}}{w_{j}}(t){{g}_{j}}(x),
 \end{array}
 \end{equation}
 where $w_{c}(t)=(w_{1}(t),....,w_{N_{c}}(t))$ is the control variable and is given in terms of a feedback operator $\mathcal{K}.$ More precisely, $w_{c}=(w_{1},...,w_{N_{c}})$ satisfies the following ODE
 \begin{equation}\nonumber
 \begin{array}{l}
 w_{c}^{'}=-\gamma w_{c}+\mathcal{K}\begin{pmatrix}
 P(v-v_{s})\\w_{c}
 \end{pmatrix}\quad \mbox{in}\quad (0,\infty),\quad
 w_{c}(0)=0,
 \end{array}
 \end{equation}
 where $\gamma$ is a positive constant, $P$ is the Leray projector from $L^{2}(\Omega)$ to $V^{0}_{n}(\Omega)$ (\cite[Section 1.4]{temam}) and $\mathcal{K}\in\mathcal{L}(V^{0}_{n}(\Omega)\times\mathbb{R}^{N_{c}},\mathbb{R}^{N_{c}})$ (the feedback operator $\mathcal{K}$ is determined in Section \ref{efcl}).\\
  The boundary control \eqref{findcon} we construct has a finite dimensional range and resembles with the control designed in \cite{ray2f}. The construction of our control basis $\{g_{j}\suchthat 1\leqslant j \leqslant N_{c}\}$ is different from the one done in \cite{ray2f}. In \cite{ray2f} it is constructed using generalized eigenvectors of the adjoint of Oseen operator while we construct it only by using eigenvectors of adjoint of Oseen operator relying on the construction of \cite{raym}. We will not consider any control on the transport equation modeling the density and as for the homogeneous Navier-Stokes equations, we show that considering a control $u_{c}$ of the velocity is enough to stabilize the whole system \eqref{1-3}.\\  
 The stabilizability of the constant density (or homogeneous) incompressible Navier-Stokes equation (with Dirichlet or mixed boundary condition) by a finite dimensional feedback Dirichlet boundary control has already been studied in the literature. For instance in \cite{ray2f} it is proved that in a $C^{4}$ domain the velocity profile $v,$ solution to system \eqref{1-3}$_{4}$-\eqref{1-3}$_{7}$ with $\rho=1$ is locally stabilizable around a steady state ${v}_{s}$ (${v}_{s}\in H^{3}(\Omega;\mathbb{R}^{2})$) by a finite dimensional Dirichlet boundary control localized in a portion of the boundary and moreover the control $u_{c}$ is given as a feedback of the velocity field.\\ 
 Unlike the constant density incompressible Navier-Stokes equations (which is of parabolic nature), the system \eqref{1-3} obeys a coupled parabolic-hyperbolic dynamics. Local exact controllability to trajectories of the system \eqref{1-3} was studied in \cite{erv1}. In the present article we answer the question posed in \cite{erv1} on the stabilizability of the system \eqref{1-3} around the Poiseuille profile. In proving the controllability results one of the main geometric assumptions of \cite{erv1} is that 
 \begin{equation}\label{geoasm}
 \begin{array}{l}
 \overline{\Omega}=\Omega^{T}_{{out}}=\{x\in \overline{\Omega}\suchthat \exists t\in (0,T),\,\mbox{s.t}\, \overline{X}(t,0,x)\in \mathbb{R}^{d}\setminus\overline{\Omega}\},
 \end{array}
 \end{equation}
 where $\overline{X}$ is the flow corresponding to the target velocity trajectory $\overline{v}_{s}$ defined as
 \begin{equation}\nonumber
 \begin{array}{l}
 \forall (x,t,s)\in \mathbb{R}^{d}\times [0,T]^{2},\quad \partial_{t}\overline{X}(x,t,s)=\overline{v}_{s}(\overline{X}(x,t,s),t),\quad \overline{X}(x,s,s)=x.
 \end{array}
 \end{equation}
 In the article \cite{erv1} the assumption \eqref{geoasm} plays the key role in controlling the density of the fluid. In our case since the target velocity trajectory is $v_{s}$ (defined in \eqref{vs}) the assumption \eqref{geoasm} is not satisfied because $v_{s}$ vanishes at the lateral boundary of the domain $\Omega.$ Hence to control the density we make a parallel assumption \eqref{1-4}. Indeed, the assumption \eqref{1-4} implies that $supp(\rho_{0})\Subset\Omega^{T}_{{out}}.$ The assumption \eqref{1-4} exploits the hyperbolic nature of the continuity equation \eqref{1-3}$_{1}$ in order to control the coupled system \eqref{1-3}. The condition \eqref{1-4} in fact guarantees that the density exactly equals ${\rho}_{s}=1,$ after some time $T_{1}=T_{A_{1}}>\frac{d}{\inf\limits_{x_{2}\in[A_{1},1-A_{1}]}v_{s}}$ (will be detailed in Section \ref{density}) so that the non-homogeneous Navier-Stokes equations become homogeneous after some finite time. In \cite{erv1} the authors uses two control functions (one for the density and one for velocity) for the purpose of controlling the non-homogeneous fluid. Contrary to that we use only one control acting on the velocity to stabilize the coupled system \eqref{1-3}.
 \subsection{Decomposition of the boundary $\Gamma$ and comment on the support of control}
  Based on the velocity profile $v_{s}$ (as defined in \eqref{vs}) we can rewrite the boundary of $\Omega$ as follows
$$	\Gamma= \Gamma_{in}\cup \Gamma_{out}\cup \Gamma_0,$$ where
\begin{equation}\label{1-2}
\begin{array}{l}
\Gamma_{in}\,\mbox{is\,defined\,in}\,\eqref{inflow},
\vspace{1.mm}\\
\Gamma_{out}=\{x\in \Gamma \mid (v_{s}\cdot n)(x) >0\}=\{d\}\times (0,1),
\vspace{1.mm}\\
\Gamma_{0}=((0,d)\times \{0\})\cup ((0,d)\times \{1\})=\Gamma_{b}\cup\Gamma_{h}\quad(\mbox{Figure}\,1).
\end{array}
\end{equation} 
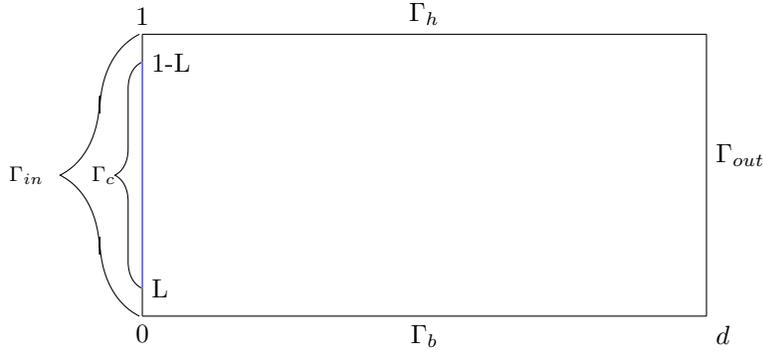
\begin{figure}[h!]\label{pic}
	\centering
	\begin{tikzpicture}[scale=0.75]
	\draw (5,0)node [below] {$\Gamma_b$};
	\draw (0,0)node [below] {$0$};
	\draw (10,0)node [below right] {$d$};
	\draw (0,5)node [above] {$1$};
	\draw (10,2.5)node [above right] {$\Gamma_{out}$};
	\draw (0,4.5)node [right] {1-L};
	\draw (0,.5)node [right] {L};
	\draw (5,5) node [above] {$\Gamma_h$};
	\draw (0,0) -- (10,0);
	\draw (10,0) -- (10,5);
	\draw (10,5) -- (0,5);
	\draw (0,5) -- (0,4.5);
	\draw (0,0.5) -- (0,0);
	\draw[blue] (0,4.5) -- (0,0.5);
	\draw [decorate,decoration={brace,amplitude=10pt},xshift=-0.5pt,yshift=0pt]
	(0,.5) -- (0,4.5) node [black,midway,xshift=-0.5cm]
	{\footnotesize $\Gamma_{c}$};
	\draw [decorate,decoration={brace,amplitude=30pt},xshift=-1.5pt,yshift=0pt]
	(0,0) -- (0,5) node [black,midway,xshift=-1.5cm]
	{\footnotesize $\Gamma_{in}$};
	\end{tikzpicture}
	\caption{Picture of the domain $\Omega$.}
\end{figure}
\begin{remark}\label{ngin}
	From now onwards we will use the notation $\Gamma_{in}$ to denote the inflow boundary of both the vector fields $v_{s}$ and $v.$ This is a slight abuse of notation but we will prove the existence of the controlled trajectory $v$ in a small neighborhood (in a suitable norm) of $v_{s}$ provided the perturbation $v_{0}$ is small. This will guarantee that $\Gamma_{in}$ and the inflow boundary of the vector field $v_{s}$ are identical. For the details we refer the reader to the Corollary \ref{p3.0.2}.
\end{remark}  
We will look for a control function ${u}_{c}$ of the form \eqref{findcon} which is compactly supported in $\Gamma_{c}.$ More particularly we will construct the finite dimensional basis $\{{{g}_{j}}\suchthat{1\leqslant j \leqslant N_{c}}\}$ of the control space in such a way that $g_{j}$ ($\forall\, 1\leqslant j \leqslant N_{c}$) is smooth and supported in $\Gamma_{c}.$ 
\subsection{Strategy} 	
(i) As our goal is to stabilize the solution $(\rho,v)$ of \eqref{1-3} around the stationary solution $(1,v_{s})$ with a rate $e^{-\beta t}$ we introduce
\begin{equation}\label{chun}
\begin{array}{l}
y=e^{\beta t}({v}-{v}_s),\quad \sigma=e^{\beta t}(\rho - 1),\quad q=e^{\beta t}(p-p_{s}),\quad u=e^{\beta t}{u}_{c}.
\end{array}
\end{equation}
To be consistent with the notations $y$ and $\sigma,$ we further introduce the following
\begin{equation}\label{y0s0}
\begin{array}{l}
\sigma_{0}=\rho_{0},\quad y_{0}=v_{0}.
\end{array}
\end{equation} 
As in our case the control \eqref{findcon} is supported in the inflow boundary, in view of the notations introduced in \eqref{1-2} and the Remark \ref{ngin} we use \eqref{findcon} to rewrite the system \eqref{1-3} in the following form
\begin{equation}\label{2-1}
	\left\{\begin{array}{lll}
		\displaystyle
		&\displaystyle\frac{\partial \sigma}{\partial t}+(({v}_s+e^{-\beta t}{y}) \cdot \nabla)\sigma-\beta\sigma=0\quad &\mbox{in}\quad Q_{\infty},
		\vspace{1.mm}\\
		&\displaystyle\sigma (x,t)=0\quad&\mbox{on}\quad \Gamma_{in} \times(0,\infty),
		\vspace{1.mm}\\
		&\displaystyle\sigma (x,0)=\sigma_0\quad&\mbox{in}\quad\Omega,
		\vspace{1.mm}\\
		&\displaystyle \frac{\partial {y}}{\partial t}-\beta {y}-\nu \Delta {y}+ ({v}_s \cdot \nabla){y}+({y} \cdot \nabla)v_s +\nabla  q=\mathcal{F}(y,\sigma)\quad& \mbox{in}\quad Q_{\infty},\\[2.mm]
		&\displaystyle\mbox{div}\,{y}=0\quad &\mbox{in}\quad Q_{\infty},
		\vspace{1.mm}\\
		&\displaystyle{y}=0\quad &\mbox {on} \quad (\Gamma_0\cup \Gamma_{out}) \times (0,\infty),\\[1.mm]
		&\displaystyle{y}=\sum\limits_{j=1}^{N_{c}}{w_{j}}(t){{g}_{j}}(x)\quad& \mbox{on} \quad \ \Gamma_{in} \times (0,\infty),
		\vspace{1.mm}\\
		&\displaystyle{y}(x,0)={y}_0\quad&\mbox{in}\quad\Omega,
	\end{array}\right.
	\end{equation}
	where 
	$$\mathcal{F}({y},\sigma)=-e^{-\beta t}{\sigma}\frac{\partial {y}}{\partial t}-e^{-\beta t}({y}\cdot \nabla){y}-e^{-\beta t}{\sigma}(v_{s}\cdot \nabla){y} -e^{-\beta t}{\sigma}({y}\cdot \nabla)v_{s}-e^{-2\beta t}{\sigma}({y}\cdot \nabla){y}+\beta e^{-\beta t} \sigma {y}.$$
To solve a nonlinear stabilization problem the usual method is to first solve the stabilization problem for the linearized system and then use a fixed point method to conclude the stabilizability of the original nonlinear problem \eqref{2-1}. In this article due to regularity issues of the transport equation we avoid linearizing the whole system. Instead, we only linearize the equation \eqref{2-1}$_{4}$ satisfied by $y$ $i.e.$ we replace the nonlinear terms appearing in the equation \eqref{2-1}$_{4}$ by a non homogeneous source term ${f}$ and we leave the equation of the density \eqref{2-1}$_{1}$ unchanged. Hence we start by analyzing the stabilizability of the system
\begin{equation}\label{2-2}
\left\{\begin{array}{ll}
\displaystyle
\frac{\partial \sigma}{\partial t}+(({v}_s+e^{-\beta t}{y}) \cdot \nabla)\sigma-\beta\sigma=0\quad &\mbox{in} \quad Q_{\infty},
\vspace{1.mm}\\
\sigma (x,t)=0 \quad& \mbox{on} \quad \Gamma_{in} \times(0,\infty),
\vspace{1.mm}\\
\sigma (x,0)=\sigma_0\quad&\mbox{in}\quad\Omega,\\[1.mm]
\displaystyle \frac{\partial {y}}{\partial t}-\beta {y}-\nu \Delta {y}+ ({v}_s \cdot \nabla){y}+({y} \cdot \nabla){v}_s +\nabla  q={f}\quad &\mbox{in}\quad Q_{\infty},
\vspace{2.mm}\\
\mbox{div}\,{y}=0\quad& \mbox{in}\quad Q_{\infty},
\vspace{2.mm}\\
{y}=0\quad& \mbox {on} \quad (\Gamma_{0}\cup\Gamma_{out}) \times (0,\infty),\\[1.mm]
{y}=\sum\limits_{j=1}^{N_{c}}{w_{j}}(t){{g}_{j}}(x) \quad& \mbox{on} \quad \ \Gamma_{in} \times (0,\infty),
\vspace{1.mm}\\
{y}(x,0)={y}_0\quad&\mbox{in}\quad\Omega.
\end{array}\right.
\end{equation}
(ii) Section \ref{velocity} is devoted to study the stabilization of the linearized Oseen equations \eqref{2-2}$_{4}$-\eqref{2-2}$_{8}$. In that direction we first write \eqref{2-2}$_{4}$-\eqref{2-2}$_{8}$ using operator notations. This is done in the spirit of \cite{raye} but with suitable modifications which are necessary since our domain is Lipschitz. To prove the stabilizability of this system we look for a control of the form \eqref{findcon}. We will choose the functions $\{g_{j}\suchthat 1\leqslant j\leqslant N_{c}\},$ supported on $\Gamma_{c},$ so that we can prove some unique continuation property equivalent to the stabilizability of the system under consideration. This is inspired from \cite{raym}. Using the fact that $g_{j}$ (for all $1\leqslant j\leqslant N_{c}$) is supported on a smooth subset of $\Gamma$ we further show that $g_{j}$ is in $C^{\infty}(\Gamma).$ This in particular implies that the control $u_{c},$ of the form \eqref{findcon}, is smooth in the space variable.\\
 (iii) Next our aim is to find a boundary control which is given in terms of a feedback law. At the same time we have to design the control such that the velocity $y$ belongs to the space $V^{2,1}(Q_{\infty}).$ Indeed the $H^{2}(\Omega)$ regularity of the velocity field will be used later to prove the stabilization of the continuity equation. This creates another difficulty because to prove the $V^{2,1}(Q_{\infty})$ regularity of $y$ solution of \eqref{2-2}$_{4}$-\eqref{2-2}$_{8}$, one must have a compatibility between the initial velocity $y_{0},$ assumed to be in $V^{1}_{0}(\Omega)$ and the boundary condition ($i.e.$ the control $u$). We deal with this issue by adding a system of ordinary differential equations satisfied by $w_{c}.$ The corresponding extended system satisfied by $(y,w_{c})$ reads as follows
 \begin{equation}\label{2-7}
 \left\{\begin{array}{ll}
 \displaystyle
 \frac{\partial{y}}{\partial t}-\beta{y}-\nu\Delta{y}+(v_{s}\cdot \nabla){y}+({y}\cdot \nabla)v_{s}+\nabla q= {f}\quad& \mbox{in} \quad Q_{\infty},
 \vspace{1.mm}\\
 \mbox{div}\,{y}=0\quad& \mbox{in}\quad Q_{\infty},
 \vspace{1.mm}\\
 {y}=0\quad &\mbox{on}\quad ({\Gamma}_{0}\cup \Gamma_{out})\times (0,\infty),
 \vspace{1.mm}\\
 {y}=\sum\limits_{j=1}^{N_{c}}{w_{j}}(t){{g}_{j}}(x)\quad &\mbox{on}\quad \Gamma_{in}\times (0,\infty),
 \vspace{1.mm}\\
 {y}(x,0)={y}_{0}\quad&\mbox{in}\quad\Omega,
 \vspace{1.mm}\\
 w_{c}^{'}+{\gamma }w_{c}=\varphi_{c}\quad& \mbox{in}\quad (0,\infty),
 \vspace{1.mm}\\
 w_{c}(0)=0\quad&\mbox{in}\quad\Omega,
 \end{array} \right.
 \end{equation}
 where $\gamma>0$ is a positive constant and $\varphi_{c}(\in \mathbb{R}^{N_{c}})$ is a new control variable which will be determined later as a feedback of the pair $(y,w_{c}).$ Since $y(.,0)=0,$ imposing $w_{c}(0)=0$ furnishes the desired compatibility between the initial and boundary conditions of $y$ which is necessary to obtain the $V^{2,1}(Q_{\infty})$ regularity of $y.$\\
  First we will construct the control $\varphi_{c}$ given in terms of a feedback operator which is able to stabilize the homogeneous ($i.e.$ when $f=0$) extended system \eqref{2-7} by solving a Riccati equation. Then we show that the same control stabilizes the entire non-homogeneous ($i.e.$ with the non-homogeneous source term $f$) extended system \eqref{2-7} by assuming that the non-homogeneous term $f$ belongs to some appropriate space.\\[1.mm]
(iv) In Section \ref{density}, we study the stability of the continuity equation \eqref{2-2}$_{1}$-\eqref{2-2}$_{3}$. We assume the velocity field in ${V}^{2,1}(Q_{\infty})$ and $\sigma_{0}\in L^{\infty}(\Omega)$ such that \eqref{1-4} (recall from \eqref{y0s0} that $\sigma_{0}=\rho_{0}$) holds. Since $\sigma_{0}\in L^{\infty}(\Omega)$ and the transport equation has no regularizing effect we expect that $\sigma\in L^{\infty}_{loc}(Q_{\infty}).$ The Cauchy problem for the continuity equation in the presence of an inflow boundary is rather delicate. In our case we use results from \cite{boy} for the existence of a unique renormalized weak solution of the problem \eqref{2-2}$_{1}$-\eqref{2-2}$_{3}$ in the space $L^{\infty}(Q_{\infty}).$ Our proof of the stabilization of the transport equation satisfied by the density relies on the fact that the characteristics equation corresponding to the velocity field is well posed. As we are dealing with velocity fields in $L^{2}(0,\infty,H^{2}(\Omega)),$ which is not embedded in $L^{1}_{loc}(0,\infty,W^{1,\infty}(\Omega))$ in dimension two, our analysis relies on \cite{zuazua} (see also \cite[Theorem 3.7]{Bahouri-Chemin-Danchin}), stating the well-posedness of the equation of the flow as a consequence of Osgood condition. Then considering the velocity field $(v_{s}+e^{-\beta t}y)$ as a small perturbation of $v_{s}$ (see \eqref{vs} for the definition) we prove that the characteristic curves corresponding to the perturbed velocity field stay close to that of $v_{s}$ in a suitable norm. 
Using the fact that the characteristics corresponding to the velocity fields $v_{s}$ and $(v_{s}+e^{-\beta t}y)$ are close we show that the particles initially lying in the support of $\sigma_{0}$ are transported out of the domain in some finite time $T>T_{A_{1}}=\frac{d}{A_{1}(1-A_{1})}$ along the flow corresponding to the perturbed velocity field. Consequently, 
the solution $\rho$ of the equation \eqref{1-3}$_{1}$-\eqref{1-3}$_{3}$ reaches exactly the target density $\rho_{s}=1$ after the time $T.$ \\
(v) Finally in Section \ref{final}, we will use Schauder's fixed point theorem to conclude that the control designed in step (iii) locally stabilizes the non linear coupled system \eqref{2-2} and consequently Theorem \ref{main} follows. 
\subsection{Bibliographical comments}
In the literature many works have been dedicated to the study of incompressible Navier-Stokes equations. For the classical results concerning the existence-uniqueness and regularity issues of the constant density incompressible Navier-Stokes equations we refer the reader to \cite{temam}. The reader can also look into \cite{galdi} for a thorough analysis of the subject. Intricate situations may arise due to the lack of regularity when special geometric assumptions are imposed on the boundary $\partial\Omega.$ For example, the domain can have corners or edges of prescribed geometric shape. For the analysis of these situations the interested reader may look into \cite{mazya} and \cite{deuring}. In the present article the functional settings for the incompressible Navier-Stokes equations is motivated from  \cite{raye}. The results of \cite{raye} are stated in a domain with smooth boundary. Thus to adapt the functional framework from \cite{raye} in the case of a rectangular domain we have used some results from \cite{gris} and \cite{Osborn}.\\
 Regarding the Cauchy problem of the non-homogeneous Navier-Stokes equations, the existence of classical solution for the non-homogeneous Navier-Stokes equations with homogeneous Dirichlet boundary condition for velocity in space dimension three is studied in \cite{anton}. 
 Results concerning the existence-uniqueness of global in time strong solution (with small initial data and small volume force) in space dimension three can be found in \cite{lady}. In dimension two the existence and uniqueness of global in time solution (without any smallness restriction on the data) is also proved in \cite{lady}. In both of these references the velocity field is Lipschitz and the initial condition of the density is smooth enough, hence the transport equation satisfied by the density can be classically solved using the method of characteristics. To deal with less regular velocity field the concept of renormalized solution was initially developed in \cite{Lio} and later suitably adapted in several contexts. For instance, one can find an application of a suitable variation of the Di-Perna-Lions theory to prove an existence and uniqueness result for the inhomogeneous Navier-Stokes equation in \cite{des}. 
All of these articles assume that the velocity field satisfies $v\cdot n=0.$ In the present article we are dealing with the target velocity $v_{s},$ which is inflow on a part of the boundary $\partial\Omega.$ For a velocity field with inflow, one must assume a suitable boundary condition for the density so that the transport equation satisfied by the density is well posed. This problem is analyzed in the articles \cite[Chapter VI]{boy} and \cite{boy2}, where the authors suitably define the trace for the weak solution of the transport equation. They also prove that these traces enjoy the renormalization property. In the present article we use the existence, uniqueness and stability results for the transport equation from \cite{boy} and \cite{boy2}. For a more intricate case involving nonlinear outflow boundary condition, similar results can be found in \cite{boy1}.\\ 
There is a rich literature where the question of the feedback boundary stabilization of the constant density incompressible Navier-Stokes equation is investigated. For the feedback boundary stabilization of a general semilinear parabolic equation one can look into the article \cite{fur2}. The feedback stabilization of the 2D and 3D constant density Navier-Stokes equations can be found in the articles \cite{fur1} and \cite{fur3} respectively. Concerning the stabilization of homogeneous Navier-Stokes equations one can also consult \cite{ray2f} and \cite{ray3} where the feedback boundary controls are achieved by solving optimal control problems. We would also like to mention the articles \cite{munt} and \cite{barb} where the authors prove the feedback stabilization of the same model around the Poiseuille profile by using normal velocity controllers. The idea of constructing a finite dimensional boundary feedback control to stabilize a linear parabolic equation dates back to the work \cite{tri1}. In our case we adapt the ideas from the articles \cite{raym} and \cite{ray2f} in order to construct a feedback boundary control with finite dimensional range to stabilize the linear Oseen equations. Actually for constant density fluids, the article \cite{raym} deals with a more intricate case involving mixed boundary conditions. Control properties of the variable density Navier-Stokes equations have been studied in the article \cite{frcara}, which proves several optimal control results in the context of various cost functionals. We also refer to the article \cite{erv1} where the authors prove the local exact controllability to a smooth trajectory of the non-homogeneous incompressible Navier-Stokes equation.\\
The study of the controllability and stabilizability issues of a system coupling equations of parabolic and hyperbolic nature is relatively new in the literature. We would like to quote a few articles in that direction. Null-controllability of a system of linear thermoelasticity (coupling wave and heat equations) in a $n-$ dimensional, compact, connected $C^{\infty}$ Riemannain manifold is studied in \cite{lebeauzua}. Controllability and stabilizability issues of compressible Navier-Stokes equations are investigated in \cite{raymcho}, \cite{rammy}, \cite{eggp} (in dim $1$) and \cite{ervgugla} (in dim $2$ and $3$). The compressible Navier-Stokes equations are also modeled by a coupled system of momentum balance and mass balance equations but the coupling is different from the one we consider in system \eqref{1-3}.\\
Let us emphasize that in the system \eqref{1-3} the control acts only on the velocity of the fluid and not on the density. In the literature there are articles dealing with controllability issues of a system of PDEs in which the controls act only on some components of the system. We would like to quote a few of them. We refer to \cite{lissy} where the authors prove local null-controllability of the three dimensional incompressible Navier-Stokes equations using distributed control with two vanishing components. A related result concerning the stabilizability of $2-$d incompressible Navier-Stokes equations using a control acting on the normal component of the upper boundary is proved in \cite{ervcho}. In \cite{lebeauzua} to prove the null-controllability of a system of linear thermoelasticity the authors consider the control on the wave equation $i.e.$ on the hyperbolic part and not on the parabolic equation modeling the temperature. On the other hand controllability and stabilizability issues of one dimensional compressible Navier-Stokes equations have been studied in \cite{raymcho} and \cite{rammy} by using only a control acting on the velocity. In the present article we also consider the control on the velocity and not on the density but our approach exploits more directly and in a more intuitive manner the geometry of the flow of the target velocity in order to control the hyperbolic transport equation modeling the density.
\subsection{Outline} In section \ref{velocity} we study the feedback stabilization of the velocity. Section \ref{density} is devoted to the stabilization of the density. In Section \ref{final} we use a fixed point argument to prove the stabilizability of the coupled system \eqref{1-3}. Finally in Section \ref{furcom} we briefly comment on how to adapt our analysis if one wishes to control the outflow boundary $\Gamma_{out}$ or the lateral boundary $\Gamma_{0}$ of the channel $\Omega.$ 
\section{Stabilization of the Oseen equations}\label{velocity}
The goal of this section is to discuss the stabilization of the Oseen equations \eqref{2-2}$_{4}$-\eqref{2-2}$_{8}$. We will first design a localized boundary control with finite dimensional range to stabilize the linear Oseen equation \eqref{2-2}$_{4}$-\eqref{2-2}$_{8}.$ We will then construct the control as a feedback of $(y,w_{c}),$ where the pair $(y,w_{c})$ solves the extended system \eqref{2-7}. The plan of this section is as follows\\
(i) In Section \ref{stablin}, we study the stabilization of the homogeneous linear system (with $f=0$) \eqref{2-2}$_{4}$-\eqref{2-2}$_{8}$, using a finite dimensional boundary control.\\
(ii) We will analyze the feedback stabilization of the extended system \eqref{2-7} in Section \ref{stabext}. Moreover with this feedback control we will prove the $V^{2,1}(Q_{\infty})$ regularity of the solution of linear Oseen equations \eqref{2-2}$_{4}$-\eqref{2-2}$_{8}$. Using a further regularity regularity estimate (see \eqref{dofcon}) of the control $u$ we show that $(e^{-\beta t}y+v_{s})$ has the same inflow and outflow as that of $v_{s},$ provided the initial condition $y_{0}$ and the non-homogeneous source term $f$ (appearing in \eqref{2-2}$_{4}$-\eqref{2-2}$_{8}$) are suitably small (see Corollary \ref{p3.0.2} ).
\subsection{Stabilization of the linear Oseen equations}\label{stablin}
In the following section we will define some operators and present some of their properties which helps in studying the linearized Oseen equations \eqref{2-2}$_{4}$-\eqref{2-2}$_{8}.$ 
\subsubsection{Writing the equations with operators}
The following results are taken from \cite{raye} where they are stated in a $C^{2}$ domain. It is necessary to make suitable changes to adapt those results in our case since the domain $\Omega$ in our case is Lipschitz. Without going into the details of the proofs we will just comment on how to adapt those results in our case.\\
Let $P$ be the orthogonal projection operator from $L^{2}(\Omega)$ onto ${V}^{0}_{n}(\Omega)$ known as Helmholtz or Leray projector (see \cite[Section 1.4]{temam}).
\newline 
We denote by $ (A,\mathcal{D}(A)) $ (the Oseen operator) and $(A^*,\mathcal{D}(A^*))$ the unbounded operators in ${V}^{0}_{n}(\Omega),$ defined by 
\begin{equation}\label{oprA}
\begin{array}{ll} 
\mathcal{D}({A})={V}^{2}(\Omega)\cap {V}^{1}_{0}(\Omega),\quad & A{y}=\nu P \Delta {y}+\beta {y}-P(({v}_s\cdot \nabla){y})-P(({y}\cdot \nabla){v}_s),  
\vspace{1.mm}\\
{\mathcal{D}}({A}^*)={V}^{2}(\Omega)\cap {V}^{1}_{0}(\Omega), \quad & A^{*}{y}=\nu P\Delta {y} +\beta {y}+P(({v}_s\cdot \nabla){y})-P((\nabla {v}_s)^T){y}.
\end{array}
\end{equation}
For the $H^{2}(\Omega)$ regularity of the solutions of the homogeneous Dirichlet boundary value problems corresponding to the operators $A$ and $A^{*}$ in a rectangular domain $\Omega,$ one can apply \cite[Theorem 3.2.1.3]{gris}.
Since $v_{s}$ is smooth with div$(v_{s}) =0,$ we can prove the following lemma.
\begin{lem}\label{estres}
	\cite[Section 2.2]{ray2} There exists $\lambda_0 >0$ in the resolvent set of $A$ such that the following hold 
	\begin{equation}\label{2-4}
	\begin{array}{l}
	\langle (\lambda_0I -A){y},{y}\rangle _{{V}^{0}_{n}(\Omega)}\geq \frac{1}{2}|{y} |^{2}_{V^{1}_{0}(\Omega)}\quad \mbox{for all}\quad  {y} \in \mathcal{D}(A), 
	\vspace{.1cm}\\
	{\mbox{and}}\\[1.mm]
	\langle(\lambda_0 I-A^*){y},{y}\rangle_{V^{0}_{n}}(\Omega)\geq\frac{1}{2}|{y} |^{2}_{V^{1}_{0}(\Omega)}\quad \mbox{for all} \quad  {y} \in \mathcal{D}(A^*).
	\end{array}
	\end{equation}
\end{lem}
In Lemma \ref{estres} we can always choose $\lambda_{0}>\beta,$ taking $\lambda_{0}$ larger if necessary. Throughout this article we will stick to this assumption.
Now, Lemma \ref{estres} can be used to prove the following.
\begin{lem}\label{t2-1}
	 The unbounded operator $(A,\mathcal{D}(A))$ (respectively $(A^{*},\mathcal{D}(A^{*}))$) is the infinitesimal generator of an analytic semi group on $V^{0}_{n}(\Omega)$. Moreover the resolvent of $A$ is compact. 	
\end{lem}  
\begin{proof}
	The proof of the fact that $(A,\mathcal{D}(A))$ (respectively $(A^{*},\mathcal{D}(A^{*}))$) generates an analytic semigroup on $V^{0}_{n}(\Omega)$ uses the resolvent estimate \eqref{2-4} and can be found in \cite[Lemma 4.1]{raye}. One can mimic the arguments used in \cite[Lemma 3.1]{fur1} to show that the resolvent of $A$ is compact. The reader can also look into \cite[Section 3]{ray2f}. 
\end{proof}
Now we want to find a suitable operator $B$ to write down the Oseen equation as a boundary control system.\\
Consider the following system of equations
\begin{equation}\label{liftopt}
\left\{ \begin{array}{ll}
\lambda_0 {y} - \nu  \Delta {y}-\beta {y}+(({v}_s\cdot \nabla){y})+(({y}\cdot \nabla){v}_s)+\nabla q=0\quad &\mbox{in}\quad \Omega,
\vspace{1.mm}\\
\mbox{div}({y})=0 \quad& \mbox{in}\quad \Omega, 
\vspace{1.mm}\\
{y}={u} \quad& \mbox{on}\quad \Gamma.
\end{array}\right.
\end{equation}
\begin{lem}\label{h2reg}
	Let $\lambda_{0}$ be as in Lemma \ref{estres}. For $u\in V^{3/2}(\Gamma),$ the system \eqref{liftopt} admits a unique solution $(y,q)\in V^{2}(\Omega)\times {H^{1}(\Omega)}/{\mathbb{R}}$ and moreover the following inequality holds
	\begin{equation}\label{estyq}
	\begin{array}{l}
	\|y\|_{V^{2}(\Omega)}+\|q\|_{H^{1}(\Omega)/\mathbb{R}}\leqslant C\|u\|_{H^{3/2}(\Gamma)},
	\end{array}
	\end{equation}
	for some constant $C>0.$
\end{lem}

	\begin{remark}
	The Lemma \ref{h2reg} is inspired from \cite[Lemma B.1.]{raye} where it is proved in the case of a $C^{2}$ domain.
\end{remark}

\begin{proof}[Proof of Lemma \ref{h2reg}]
	We write $(y,q)=(y_{1},q_{1})+(y_{2},q_{2}),$ such that $(y_{1},q_{1})$ satisfies
	\begin{equation}\label{estyq1}
\left\{	\begin{array}{ll}
	\lambda_{0}y_{1}-\nu\Delta y_{1}-\beta y_{1}+\nabla q_{1}=0\quad&\mbox{in}\quad\Omega,\\
	\mbox{div}(y_{1})=0\quad&\mbox{in}\quad\Omega,\\
	y_{1}=u\quad&\mbox{on}\quad\Gamma
	\end{array}\right.
	\end{equation}
and $(y_{2},q_{2})$ satisfies
	\begin{equation}\label{estyq2}
	\left\{ \begin{array}{ll}
	\lambda_0 {y}_{2} - \nu  \Delta {y}_{2}-\beta {y}_{2}+(({v}_s\cdot \nabla){y}_{2})+(({y}_{2}\cdot \nabla){v}_s)+\nabla q_{2}=-(({v}_s\cdot \nabla){y}_{1})-(({y}_{1}\cdot \nabla){v}_s)\quad& \mbox{in}\quad \Omega,
	\vspace{1.mm}\\
	\mbox{div}({y}_{2})=0 \quad &\mbox{in}\quad \Omega, 
	\vspace{1.mm}\\
	{y}_{2}=0 \quad &\mbox{on}\quad \Gamma.
	\end{array}\right.
	\end{equation}
	 As $u\in H^{3/2}(\Omega),$ the solution to \eqref{estyq1} satisfies $(y_{1},q_{1})\in V^{2}(\Omega)\times H^{1}(\Omega)/\mathbb{R}$ (see \cite{Osborn}) and the following inequality is true
	 \begin{equation}\label{estyq3}
	 \begin{array}{l}
	 	\|y_{1}\|_{V^{2}(\Omega)}+\|q_{1}\|_{H^{1}(\Omega)/\mathbb{R}}\leqslant C\|u\|_{H^{3/2}(\Gamma)},
	 \end{array}
	 \end{equation}
	 where $C>0$ is a constant. Using \eqref{estyq3} we observe that the right hand side of \eqref{estyq2}$_{1}$ is in $H^{1}(\Omega).$ Hence we get that $y_{2}\in V^{2}(\Omega)\cap V^{1}_{0}(\Omega).$ Then the corresponding pressure $q_{2}\in H^{1}(\Omega)/\mathbb{R}$ can be recovered using De Rham's theorem (see \cite[Section 1.4]{temam}). Using \eqref{estyq3} one also has the following inequality
	  \begin{equation}\label{estyq4}
	  \begin{array}{l}
	  \|y_{2}\|_{V^{2}(\Omega)}+\|q_{2}\|_{H^{1}(\Omega)/\mathbb{R}}\leqslant C\|u\|_{H^{3/2}(\Gamma)},
	  \end{array}
	  \end{equation}
	 for some positive constant $C.$
	 The inequalities \eqref{estyq3} and \eqref{estyq4} together yield \eqref{estyq}. 
	\end{proof}
Now for $u\in V^{3/2}(\Gamma),$ we define the Dirichlet lifting operators $D_{A}u=y$ and $D_{p}u=q,$ where $(y,q)$ is the solution of \eqref{liftopt} with Dirichlet data $u.$
\begin{lem}\label{l2-2} (i) The operator $D_A$ can be extended as a bounded linear map from $V^{0}(\Gamma)$ to $V^{1/2}(\Omega).$ Moreover $D_{A}\in \mathcal{L}({V}^{s}(\Gamma),{V}^{s+1/2}(\Omega))$ for all $0\leqslant s \leqslant 3/2.$\\ 
	(ii)The operator $D_{A}^{*},$ the adjoint of $D_A$ computed as a bounded operator from $V^{0}(\Gamma)$ to $V^{0}(\Omega)$ is a bounded linear operator from $V^{0}(\Omega)$ to $V^{0}(\Gamma)$ and is given as follows  
	\begin{equation}\label{representation}
	\begin{array}{l}
		\displaystyle
	 D^{*}_{A}{g} = -\nu \frac{\partial {z}}{\partial {n}}+ \pi {n} -{\frac{1}{|\Gamma|}}\left(\int\limits_{\Gamma}\pi \right){n},
	 \end{array}
	 \end{equation}
	where $(z,\pi)$ is the solution of
	\begin{equation}\label{defda*}
	\left\{	\begin{array}{ll}
	\lambda_{0} {z} - \nu  \Delta {z}-\beta {z}-({v}_s\cdot \nabla){z}+(\nabla {v}_s)^{T}{z}+\nabla \pi=g\quad&\mbox{in}\quad\Omega,
	\vspace{1.mm}\\
	\textup{div}{z}=0 \quad& \mathrm{in}\quad \Omega, 
	\vspace{1.mm}\\
	{z}=0 \quad& \mathrm{on}\quad \Gamma,
	\end{array}\right.
	\end{equation}
	Here $|\Gamma| $ is the one dimensional Lebesgue measure of $\Gamma$. Moreover $D_{A}^{*}\in\mathcal{L}({V}^{0}(\Omega),{V}^{1/2}(\Gamma)).$\\
	 (iii) The operator $D_{A}^{*}$ can be extended as a bounded linear operator from $H^{-\frac{1}{2}+\kappa}(\Omega)$ to $V^{\kappa}(\Gamma),$ for all $0<\kappa<\frac{1}{2},$ $i.e.$
	 \begin{equation}\label{regD*}
	 \begin{array}{{l}}
	 D_{A}^{*}\in\mathcal{L}({H}^{-\frac{1}{2}+\kappa}(\Omega),{V}^{\kappa}(\Gamma))\quad\mbox{for all}\quad 0<\kappa<\frac{1}{2}.
	 \end{array}
	 \end{equation} 
	\end{lem}
	\begin{remark}
	The first two parts of the Lemma \ref{l2-2} are inspired from \cite[Lemma B.4.]{raye} where it is proved in the case of a $C^{2}$ domain.
	\end{remark}
\begin{proof}[Proof of Lemma \ref{l2-2}]
	(i) In Lemma \ref{h2reg} we have proved that $D_{A}$ is a bounded operator from $V^{3/2}(\Gamma)$ into $V^{2}(\Omega).$ Following \cite[Theorem B.1]{raye} one obtains that the operator $D_{A}$ can be extended as a bounded linear map from $V^{0}(\Gamma)$ into $V^{1/2}(\Omega)$ (in the sense of variational formulation). Hence one can use interpolation to prove that $D_{A}$ is bounded from $V^{s}(\Gamma)$ to $V^{s+1/2}(\Omega),$ for all $0\leqslant s\leqslant 3/2.$\\
	(ii) The second part can be done following the proof of \cite[Lemma B.4]{raye}. It is therefore left to the reader.\\
	(iii) For the final part, in view of \cite[Appendix B, Lemma B.1.]{raye} one first observes that the map $\mathcal{M}_{g\rightarrow (z,\pi)},$ mapping $g$ to $(z,\pi)$ (where $g,$ $z$ and $\pi$ are as in \eqref{defda*}) satisfies the following
	\begin{equation}\label{interreg}
	\begin{array}{l}
	\mathcal{M}_{g\rightarrow (z,\pi)}\in \mathcal{L}(L^{2}(\Omega),(V^{2}(\Omega)\cap V^{1}_{0}(\Omega))\times H^{1}(\Omega)/\mathbb{R}).
	\end{array}
	\end{equation}
	Now following \cite[Appendix B, Theorem B.1.]{raye} one can use method of transposition to define weak type solution of the problem \eqref{defda*} when $g\in (H^{2}(\Omega)\cap H^{1}_{0}(\Omega))',$ where $(H^{2}(\Omega)\cap H^{1}_{0}(\Omega))'$ is the dual of the space $H^{2}(\Omega)\cap H^{1}_{0}(\Omega)$ provided that $L^{2}(\Omega)$ is identified with its dual. In particular one has the following
	\begin{equation}\label{interreg1}
	\begin{array}{l}
    \mathcal{M}_{g\rightarrow (z,\pi)}\in \mathcal{L}((H^{2}(\Omega)\cap H^{1}_{0}(\Omega))',V^{0}(\Omega)\times (H^{1}(\Omega)/\mathbb{R})').
	\end{array}
	\end{equation}
	Now let us assume $g\in H^{-1}(\Omega),$ where $H^{-1}(\Omega)$ denotes the dual of $H^{1}_{0}(\Omega)$ with $L^{2}(\Omega)$ as the pivot space. Using \eqref{interreg1} and the fact that $H^{-1}(\Omega)\subset (H^{2}(\Omega)\cap H^{1}_{0}(\Omega))'$ (since $(H^{2}(\Omega)\cap H^{1}_{0}(\Omega))$ is dense in $H^{1}_{0}(\Omega)$) one can write \eqref{defda*} as follows
	\begin{equation}\label{defda**}
	\left\{	\begin{array}{ll}
	- \nu  \Delta {z}+\nabla \pi=g^{*}\quad&\mbox{in}\quad\Omega,
	\vspace{1.mm}\\
	\textup{div}{z}=0 \quad& \mathrm{in}\quad \Omega, 
	\vspace{1.mm}\\
	{z}=0 \quad& \mathrm{on}\quad \Gamma,
	\end{array}\right.
	\end{equation}
	where 
	$$g^{*}=g-\lambda_{0} {z}+\beta {z}+({v}_s\cdot \nabla){z}-(\nabla {v}_s)^{T}{z}\in H^{-1}(\Omega).$$
    Now \cite[Theorem IV.5.2]{boy} furnishes the following regularity
    \begin{equation}\label{penulest}
    \begin{array}{l}
    \mathcal{M}_{g\rightarrow (z,\pi)}\in \mathcal{L}(H^{-1}(\Omega),V^{1}_{0}(\Omega)\times L^{2}(\Omega)/\mathbb{R}).
    \end{array}
    \end{equation}
    Now from \eqref{interreg}, \eqref{penulest} and using the interpolation result  \cite[Theorem 5.1.]{liomag} one has
    \begin{equation}\label{penulest*}
    \begin{array}{l}
    \mathcal{M}_{g\rightarrow (z,\pi)}\in \mathcal{L}(H^{-\frac{1}{2}+\kappa}(\Omega),(V^{\frac{3}{2}+\kappa}(\Omega)\cap V^{1}_{0}(\Omega))\times H^{\frac{1}{2}+\kappa}(\Omega)/\mathbb{R}),
    \end{array}
    \end{equation}
    for $-\frac{1}{2}\leqslant\kappa\leqslant\frac{1}{2}.$\\
    Finally the definition \eqref{representation} of $D^{*}_{A}$ and \eqref{penulest*} in particular provide that
    $$D_{A}^{*}\in\mathcal{L}({H}^{-\frac{1}{2}+\kappa}(\Omega),{V}^{\kappa}(\Gamma)),\,\,\mbox{for all}\,\,0<\kappa<\frac{1}{2}.$$
    Hence we are done with the proof of Lemma \ref{l2-2}.
   	\end{proof}
   	\begin{remark}
   		In part $(ii)$ of Lemma \ref{l2-2}, the operator $D_{A}^{*}$ is defined on the space of divergence free functions but in part $(iii)$ we extended this definition by removing the divergence free constraint on the elements of the domain of $D^{*}_{A}.$ This is possible since it is not necessary to have a divergence free function $g$ in order to solve \eqref{defda*}.
   	\end{remark}
In order to localize the control of the velocity on $\Gamma_{c}$ (defined in \eqref{dgc}), we introduce the operator $M,$ which is defined as follows
\begin{equation}\label{docn}
\begin{array}{l}
\displaystyle
M{g}(x)=m(x){g}(x)- \frac {m}{\displaystyle\int\limits_{\Gamma}m}\left(\int\limits_{\Gamma}m{g}\cdot {n} \right){n}(x)\quad\mbox{for all}\quad x\in\Gamma.
\end{array}
\end{equation}
In the expression \eqref{docn} the weight function $m \in C^{\infty} (\Gamma)$  takes values in $[0,1]$ and is supported in $\Gamma_{c}\subset \Gamma_{in}.$ Moreover, $m$ equals 1 in some open connected component 
\begin{equation}\label{gammac+}
\begin{array}{l}
\Gamma_{c}^{+}\Subset \Gamma_{c}.
\end{array}
\end{equation} 
So the operator $M$ localizes the support of the control on $\Gamma_{in}$ and also guarantees that $Mg\in V^{0}(\Gamma)$ for any $g\in L^{2}(\Gamma).$ 
\begin{lem}\label{l2-3} \cite[Lemma 2.3]{ray2f} The operator $M \in \mathcal{L}({V}^{0}(\Gamma))$ (defined in \eqref{docn}) is symmetric. 
\end{lem}
Sometimes we might use the notation 
\begin{equation}\label{stten}
\begin{array}{l}
\mathbb{T}(v,p)=\nu(\nabla{v}+(\nabla v)^{T})-pI,
\end{array}
\end{equation}
to denote the Cauchy stress tensor corresponding to a vector field $v$ and a pressure $p.$\\
We now define the operator
\begin{equation}\label{DEFB}
\begin{array}{l}
B=(\lambda_{0}I-A)PD_{A}M\in \mathcal{L}({V}^{0}(\Gamma),(\mathcal{D}(A^{*}))'),
\end{array}
\end{equation}
where $(\mathcal{D}(A^{*}))'$ denotes the dual of the space $\mathcal{D}(A^{*})$ with $V^{0}_{n}(\Omega)$ as the pivot space.
\begin{prop}\label{p2-4} 
	$(i)$ The adjoint of the operator $B,$ computed for the duality structure $\langle \cdot,\cdot\rangle_{(\mathcal{D}(A^{*})',\mathcal{D}(A^{*}))},$ that we will denote by $B^{*}$ in the following, satisfies $B^{*}\in \mathcal{L} (\mathcal{D}(A^{*}),{V}^{0}(\Gamma))$ and for all $\Phi\in\mathcal{D}(A^{*}),$
		\begin{align}
		\label{exB*phi1}
		B^{*}\Phi&=M\left(-\nu \frac{\partial \Phi}{\partial {n}}+\left(\psi  -{\frac{1}{|\Gamma|}}\left(\int\limits_{\Gamma}\psi \right)\right){n}\right)\\
		\label{exB*phi2}
		&=-M\mathbb{T}\left(\Phi,\left(\psi  -{\frac{1}{|\Gamma|}}\left(\int\limits_{\Gamma}\psi \right)\right)\right)n,
		\end{align}
	where
	\begin{equation}\label{exB*phi3}
	\begin{array}{l}
	\nabla \psi =(I-P)[\nu \Delta \Phi +({v}_s \cdot \nabla)\Phi-(\nabla {v}_s)^{T}\Phi],
	\end{array}
	\end{equation}
	and $\mathbb{T}$ denotes the stress tensor as defined in \eqref{stten}.\\
	$(ii)$ There exists a positive constant $\omega>0$ such that the operator $B^{*}$ can be extended as a bounded linear map from $\mathcal{D}((\omega I-A^{*})^{\frac{3}{4}+\frac{\kappa}{2}})$ to $V^{\kappa}(\Gamma),$ for all $0<\kappa<\frac{1}{2}$ $i.e.$
	\begin{equation}\label{extnregB*}
	\begin{array}{l}
	B^{*}\in\mathcal{L}(\mathcal{D}((\omega I-A^{*})^{\frac{3}{4}+\frac{\kappa}{2}}),V^{\kappa}(\Gamma))\quad\mbox{for all}\quad 0<\kappa<\frac{1}{2}.
	\end{array}
	\end{equation} 
\end{prop}
\begin{proof}
	(i) From Lemma \ref{l2-2}, we know that 
	\begin{equation}\label{B*phi}
	\begin{array}{l}
	\displaystyle B^{*}\Phi=MD^{*}_{A}P(\lambda_{0}I-A^{*})\Phi=M\left(-\nu \frac{\partial \hat\Phi}{\partial {n}}+\left(\psi  -{\frac{1}{|\Gamma|}}\left(\int\limits_{\Gamma}\psi \right)\right){n}\right),
	\end{array}
	\end{equation}
	 where
	\begin{equation}\label{defda*1}
	\left\{	\begin{array}{ll}
	\lambda_{0} {\hat{\Phi}} - \nu  \Delta {\hat{\Phi}}-\beta {\hat{\Phi}}-(({v}_s\cdot \nabla){\hat{\Phi}})+(\nabla {v}_s)^{T}{\hat{\Phi}}+\nabla \psi=P(\lambda_{0}I-A^{*})\Phi\quad&\mbox{in}\quad\Omega,
	\vspace{1.mm}\\
	\textup{div}{\hat{\Phi}}=0 \quad& \mathrm{in}\quad \Omega, 
	\vspace{1.mm}\\
	{\hat{\Phi}}=0 \quad& \mathrm{on}\quad \Gamma.
	\end{array}\right.
	\end{equation}
	This gives $\hat{\Phi}=\Phi$ and the expression \eqref{exB*phi3}. Hence the representation \eqref{exB*phi1} directly follows from \eqref{B*phi}. Also \eqref{exB*phi2} follows from \eqref{exB*phi1} because $(\nabla\Phi)^{T}n=0$ on $\Gamma$ (this can be easily deduced from the fact that $\Phi$ on $\Gamma$ is zero and $\mbox{div}(\Phi)=0$ on $\Omega$).\\
	(ii) Recall from Lemma \ref{t2-1} that $(A^{*},\mathcal{D}(A^{*}))$ generates an analytic semigroup on $V^{0}_{n}(\Omega).$ Hence one can always choose a large enough positive constant $\omega$ from the resolvent set of $A,$ such that the spectrum of $(A^{*}-\omega I)$ lies in the open left half-plane. Now following the definition \cite[p. 329, Section 7.4, Eq. 7.4.3]{fattorini} one can define the operator $(\omega I-A^{*})^{\frac{3}{4}+\frac{\kappa}{2}}$ where $0<\kappa<\frac{1}{2}.$ Let us consider $\Phi\in \mathcal{D}((\omega I-A^{*})^{\frac{3}{4}+\frac{\kappa}{2}})$ where $0<\kappa<\frac{1}{2}.$ Since
	$$\mathcal{D}((\omega I-A^{*})^{\frac{3}{4}+\frac{\kappa}{2}})=[V^{2}(\Omega)\cap V^{1}_{0}(\Omega),V^{0}_{n}(\Omega)]_{\frac{1}{4}-\frac{\kappa}{2}}=V^{\frac{3}{2}+\kappa}(\Omega)\cap V^{1}_{0}(\Omega)$$
	(for details on the characterization of domains of fractional powers we refer to \cite{lionsinterpole}), one observes the following
	\begin{equation}\label{incregB*}
	\begin{array}{l}
	P(\lambda_{0}I-A^{*})\Phi\in H^{-\frac{1}{2}+\kappa}(\Omega).
	\end{array}
	\end{equation}
	Now one can use the expression of $B^{*}$ as given by \eqref{B*phi} and part $(iii)$ of the Lemma \ref{l2-2} to prove \eqref{extnregB*}.
	\end{proof}
Now following \cite{raye} the Oseen equations
\begin{equation}\label{2-5}
\left\{   \begin{array}{ll}
\displaystyle
\frac{\partial {y}}{\partial t}-\beta {y}-\nu \Delta {y}+ ({v}_s \cdot \nabla){y}+({y} \cdot \nabla)v_s +\nabla  q=0\quad& \mbox{in}\quad Q_{\infty},
\vspace{1.mm}\\
\displaystyle
\mbox{div}{y}=0\quad& \mbox{in}\quad Q_{\infty},
\vspace{1.mm}\\
\displaystyle
{y}=0\quad& \mbox {on} \quad (\Gamma_0\cup\Gamma_{out}) \times (0,\infty),
\vspace{1.mm}\\
\displaystyle
{y}=M{u} \quad &\mbox{on} \quad \ \Gamma_{in} \times (0,\infty),
\vspace{1.mm}\\
\displaystyle
{y}(x,0)={y}_0\quad & \mbox{on}\quad \Omega,
\end{array} \right.
\end{equation}
can be written in the following evolution equation form
\begin{equation}\label{2-6}
\left\{ \begin{array}{ll}
P{y}'=AP{y}+B{u}\quad& \mbox{in}\quad(0,\infty), 
\vspace{.1cm}\\
P{y}(0)={y}_{0},&\\[1.mm]
(I-P){y}=(I-P)D_{A}M{u} \quad& \mbox{in}\quad (0,\infty).
\end{array}\right.
\end{equation}
In the following section we discuss some spectral properties of the Oseen operator $A$ and then we define a suitable control space in order to construct a control function which stabilizes the Oseen equations.
\subsubsection{Spectral properties of $A$ and the stabilizability criterion}
Since the resolvent of $A$ is compact (see Lemma \ref{t2-1}), the spectrum $spec(A)$ of the operator $A$ is discrete. Moreover since $A$ is the generator of an analytic semi group (see Lemma \ref{t2-1}), $spec(A)$ is contained in a sector. Also the eigenvalues are of finite multiplicity and appear in conjugate pairs when they are not real.\\ 
We denote by $(\lambda_{k})_{k\in N}$ the eigenvalues of $A.$ Without loss of generality we can always assume that there is no eigenvalue of $A$ with zero real part by fixing a slightly larger $\beta,$ if necessary. So we choose $N_{u}\in\mathbb{N}$ such that
\begin{equation}\label{coev}
\begin{array}{l}
...\mbox{Re}\lambda_{N_{u}+1}<0< \mbox{Re}\lambda_{N_{u}}\leqslant ... \leqslant \mbox{Re}\lambda_{1}.
\end{array}
\end{equation}
 	Following \cite{raym}, we now choose the control space as follows
 	\begin{align}\label{U0}
 	 U_{0}=\mbox{vect}\oplus^{N_{u}}_{k=1}(\mbox{Re}B^{*}\mbox{ker}(A^{*}-\lambda_{k}I)\oplus \mbox{Im}B^{*}\mbox{ker}(A^{*}-\lambda_{k}I)).
 	\end{align}
 The choice \eqref{U0} of the control space plays an important role in proving a unique continuation property which implies the stabilizability of the pair $(A,B).$ Let us choose the functions $g_{j}$ in \eqref{findcon} such that
 \begin{align}\label{basisU0}
 \{{{g}_{j}}\suchthat 1\leqslant j \leqslant N_{c}\}\, \mbox{is an orthonormal basis of}\, {U}_{0}.
 \end{align}
 For later use we now prove an additional regularity result for the elements of the control space $U_{0}.$ The following regularity result is true only because the elements of $U_{0}$ are supported on a smooth subset of $\Gamma.$
 \begin{lem}\label{smgci}
 	The set $U_{0},$ defined in \eqref{U0}, is a subspace of $C^{\infty}(\Gamma).$
 \end{lem}
 \begin{proof}
 	The function $m$ is supported on $\Gamma_{c},$ which is $C^{\infty}$. In view of the representation \eqref{exB*phi1} of the operator $B^{*}$, we observe that to prove Lemma \ref{smgci} it is enough to show that for each $1\leqslant k\leqslant N_{u},$ any solution $(\phi,\psi)$ to the system \eqref{eivp1} is $C^{\infty}$ in some open set $\Omega_{\Gamma_{c}}$ ($\subset\Omega$) such that $\partial\Omega_{\Gamma_{c}}$ contains $\Gamma_{c}.$ Let us consider $k\in\{1,...,N_{u}\}$ and $(\phi,\psi)$ solves the following
 	\begin{equation}\label{eivp1}
 	\left\{ \begin{array}{ll}
 	\lambda_{k}\phi-\nu\Delta\phi-\beta\phi-((v_{s}\cdot\nabla)\phi)+(\nabla v_{s})^{T}\phi+\nabla\psi=0\quad&\mbox{in}\quad\Omega,\\[1.mm]
 	\mbox{div}\phi=0\quad&\mbox{in}\quad\Omega,\\[1.mm]
 	\phi=0\quad&\mbox{on}\quad\Gamma.
 	\end{array}\right.
 	\end{equation}
 	 We thus apply the elliptic regularity result \cite[Theorem 3.2.1.3]{gris} to show that 
 	\begin{equation}\label{phipsi}
 	\begin{array}{l}
 	\phi\in\mathcal{D}({A^{*}})= V^{2}(\Omega)\cap V^{1}_{0}(\Omega)\quad\mbox{and}\quad \psi\in H^{1}(\Omega).
 	\end{array}
 	\end{equation}
 	We will work in a neighborhood of $\Gamma_{c}$ in order to avoid the singularities due to the presence of the corners $(0,0)$ and $(0,1).$ First consider a neighborhood $N^{b}_{\Gamma_{c}}$ of $\Gamma_{c}$ such that neither of the points $(0,0)$ and $(0,1)$ belong to $N^{b}_{\Gamma_{c}}.$ Now we consider an open set $\Omega_{\Gamma_{c}}$ such that $\Omega_{\Gamma_{c}}\subset {\Omega},$ $\partial\Omega_{\Gamma_{c}}$ (the boundary of $\Omega_{\Gamma_{c}}$) is $C^{\infty}$ and $\partial\Omega_{\Gamma_{c}}\cap \Gamma=N^{b}_{\Gamma_{c}}.$ Let $\Theta\in C^{\infty}(\bar{\Omega}_{\Gamma_{c}})$ be such that $\Theta=1$ on a subset of $\bar\Omega_{\Gamma_{c}}$ containing $\Gamma_{c}$ and $\Theta=0$ on $\partial\Omega_{\Gamma_{c}}\setminus N^{b}_{\Gamma_{c}}.$ One can check that the function $(\Theta\phi,\Theta\psi)$ satisfies the following 
 	\begin{equation}\label{eqpro}
 	\begin{array}{l}
 	- \nu  \Delta ({\Theta{\phi}})+\nabla (\Theta\psi)=F(\Theta,\phi,\psi)\quad\mbox{in}\quad \Omega_{\Gamma_{c}},
 	\end{array}
 	\end{equation}
 	where $$F(\Theta,\phi,\psi)=-\nu\Delta\Theta\phi-2\nu\nabla\Theta\nabla\phi-\phi(v_{s}\cdot\nabla)\Theta+\psi\nabla\Theta,$$
 	and also $\Theta\phi=0$ on $\partial\Omega_{\Gamma_{c}},$ which implies $\displaystyle \int_{\Omega_{\Gamma_{c}}}\textup{div}{({\Theta{\phi}})}=\int_{\partial\Omega_{\Gamma_{c}}}{({\Theta{\phi}})}\cdot\,n=0.$
 	Using \eqref{phipsi} one verifies that $$F(\Theta,\phi,\psi)\in H^{1}(\Omega_{\Gamma_{c}})\quad\mbox{and}\quad\mbox{div}(\Theta\phi)=\phi\cdot\nabla\Theta\in H^{2}(\Omega_{\Gamma_{c}}).$$ Now we apply \cite[Theorem {IV.5.8}]{boy} to obtain,
 	$(\Theta\phi,\Theta\psi)\in H^{3}(\Omega_{\Gamma_{c}})\times H^{2}(\Omega_{\Gamma_{c}}).$
 	We can use a bootstrap argument to conclude that, $(\Theta\phi,\Theta\psi)\in C^{\infty}(\Omega_{\Gamma_{c}}).$ Hence we finally have ${g_{j}}\in C^{\infty}(\Gamma),$ for all $1\leqslant j\leqslant N_{c} .$
 \end{proof}
 We are looking for a control $u$ taking values in $U_{0}.$ We write
 \begin{equation}\label{fdc}
 \begin{array}{l}
 u(x,t)=\sum\limits_{j=1}^{N_{c}}{w_{j}}(t){g_{j}}(x),
 \end{array}
 \end{equation}
 where $w_{c}=(w_{1},...,w_{N_{c}})\in L^{2}(0,\infty;\mathbb{R}^{N_{c}})$ is the control variable. Again in view of \cite{raym} we define a new control operator $\mathcal{B}\in \mathcal{L}(\mathbb{R}^{N_{c}},(\mathcal{D}(A^{*}))')$ as
 \begin{equation}\label{decB}
 \begin{array}{l}
 \mathcal{B}w_{c}=\sum\limits_{j=1}^{N_{c}}{w_{j}}B{g_{j}}=\sum\limits_{j=1}^{N_{c}}{w_{j}}(\lambda_{0}I-A)PD_{A}{g_{j}}.
 \end{array}
 \end{equation}
 Observe that $\mathcal{B}$ is defined by restricting the action of the operator $B$ to $U_{0}.$\\ 
 Let us consider the controlled system
 \begin{equation}\label{consans} 
 \begin{array}{l}
 Py'=APy+Bu\,\,\mbox{in}\,\,(0,\infty),\quad Py(0)=y_{0},
 \end{array}
 \end{equation}
 which we obtain from \eqref{2-6}$_{1}$-\eqref{2-6}$_{2}.$ With the definition \eqref{decB} and a control of the form \eqref{fdc}, the system \eqref{consans} takes the form
 $$Py'=APy+\mathcal{B}w_{c}\,\,\mbox{in}\,\,(0,\infty),\quad Py(0)=y_{0}.$$
\begin{thm}\label{thstab}
	Assume that the spectrum of $A$ obeys the condition \eqref{coev}, we choose $\{{g_{j}}\suchthat 1\leqslant j\leqslant N_{c}\}$ as \eqref{basisU0} and the operator $\mathcal{B}$ is as defined in \eqref{decB}. Then the pair $(A,\mathcal{B})$ is stabilizable in $V^{0}_{n}(\Omega).$ 
\end{thm}
Before going into the proof of Theorem \ref{thstab}, let us recall that the pair $(A,\mathcal{B})$ is stabilizable in $V^{0}_{n}(\Omega)$ iff for all $y_{0}\in V^{0}_{n}(\Omega),$ there exists a control $w_{c}\in L^{2}(0,\infty;\mathbb{R}^{N_{c}})$ such that the controlled system
$$Py'=APy+\mathcal{B}w_{c}\,\,\mbox{in}\,\,(0,\infty),\quad Py(0)=y_{0},$$
obeys
$$\int\limits_{0}^{\infty}\|Py(t)\|^{2}_{V^{0}_{n}(\Omega)}dt<\infty.$$
The proof of Theorem \ref{thstab} in a more intricate situation involving mixed boundary condition can be found in \cite{raym}. In \cite{raym} the localization operator $M,$ localizing the control, is simply the cutoff function $m$ whereas in our case $M$ is as defined in \eqref{docn}. For the sake of completeness, we present the proof of Theorem \ref{thstab} below, which follows step by step the one of \cite{raym} up to minor modifications. 
\begin{proof}[Proof of Theorem \ref{thstab}]
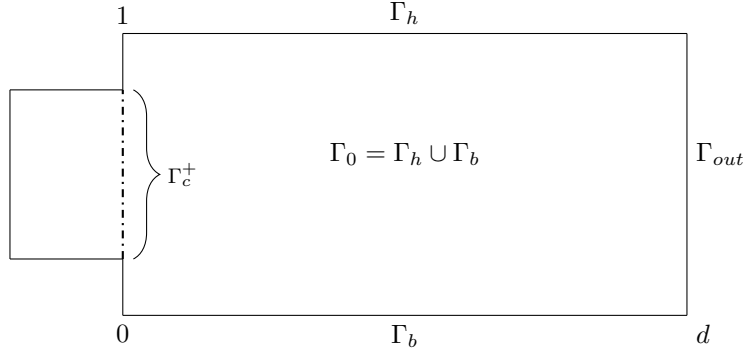
\begin{figure}[h!]\label{pic2}
	\centering
	\begin{tikzpicture}[scale=0.75]
	\draw (5,0)node [below] {$\Gamma_b$};
	\draw (0,0)node [below] {$0$};
	\draw (10,0)node [below right] {$d$};
	\draw (0,5)node [above] {$1$};
	\draw (10,2.5)node [above right] {$\Gamma_{out}$};
	\draw (5,5) node [above] {$\Gamma_h$};
	\draw (5,2.5) node [above] {$\Gamma_{0}=\Gamma_{h}\cup\Gamma_{b}$};
	\draw (0,0) -- (10,0);
	\draw (10,0) -- (10,5);
	\draw (10,5) -- (0,5);
	\draw (0,5) -- (0,4);
	\draw (0,1) -- (0,0);
	\draw [thick,dash dot] (0,1) -- (0,4);
	\draw [decorate,decoration={brace,amplitude=10pt,mirror,raise=4pt},yshift=0pt]
	(0,1) -- (0,4) node [black,midway,xshift=0.8cm] {\footnotesize
		$\Gamma^{+}_{c}$};
	\draw (0,1) -- (-2,1);
	\draw (-2,1) -- (-2,4);
	\draw (-2,4) -- (0,4);
	\end{tikzpicture}
	\caption{Domain $\Omega_{ex}$.}
\end{figure}
According to \cite[Theorem 1.2]{raypa} (one can also consult \cite[Chapter V]{ben} for related results) proving the stabilizability of the pair $(A,\mathcal{B})$ is equivalent to verifying the Hautus criterion:
\begin{align}
	\mbox{ker}(\lambda_{k}I-A^{*})\cap \mbox{ker}(\mathcal{B}^{*})=\{0\},\quad\mbox{for all}\quad 1\leqslant k \leqslant N_{u}.\label{Hutus}
	\end{align}
	Let $\phi\in \mbox{ker}(\lambda_{k}I-A^{*}).$ Also suppose that $\psi$ is the pressure associated with $\phi,$ $i.e.$ the pair $(\phi,\psi)$ solves \eqref{eivp1}.
		Now one can use \eqref{decB} and Proposition \ref{p2-4} in order to verify that
	\begin{align}\label{exB*}
	\displaystyle
	\mathcal{B}^{*}\phi&=-\left(\int\limits_{\Gamma_{c}}{g_{j}}\overline{M\mathbb{T}(\phi,\psi)n}\,dx \right)_{1\leqslant j \leqslant N_{c}}\\
	&=-\left(\int\limits_{\Gamma_{c}}{g_{j}}M\mbox{Re}\mathbb{T}(\phi,\psi)n\, dx \right)_{1\leqslant j \leqslant N_{c}}+i\left(\int\limits_{\Gamma_{c}}{g_{j}}M\mbox{Im}\mathbb{T}(\phi,\psi)n\, dx \right)_{1\leqslant j \leqslant N_{c}}.
	\end{align}
One can notice that $M\mbox{Re}\mathbb{T}(\phi,\psi)n\in U_{0}$ and $M\mbox{Im}\mathbb{T}(\phi,\psi)n\in U_{0}.$ On the other hand we know that $\{{g_{j}}\}_{1\leqslant j\leqslant N_{c}}$ forms a basis of $U_{0}.$ Hence $\mathcal{B}^{*}\phi=0$ implies that
$$M(\mathbb{T}(\phi,\psi)n)\mid_{\Gamma_{c}}=0.$$	
 This implies that
 \begin{equation}\label{sgmagc}
 \begin{array}{l}
 \mathbb{T}(\phi,\psi)n=C_{0}n\quad\mbox{on}\quad supp\,(m),
 \end{array}
 \end{equation}
 where $C_{0}$ is a constant given by
 \begin{align}
 \displaystyle C_{0}=\frac{1}{\displaystyle\int\limits_{\Gamma}m}\left(\int\limits_{\Gamma}m\mathbb{T}(\phi,\psi)n\right).\nonumber
 \end{align}
   Now recall that $\phi=0$ on $\Gamma$ and the unit outward normal on $\Gamma_{c}^{+}$ is $(-1,0).$ Also since $\phi\in V^{2}(\Omega),$ one can consider the trace of $\mbox{div}\phi$ on $\Gamma$ to obtain that $\mbox{div}\phi=0$ on $\Gamma.$ Using these facts one can at once deduce from \eqref{sgmagc} that $\displaystyle\frac{\partial\phi}{\partial n}=0$ and $\psi=C_{0}$ on $\Gamma^{+}_{c}.$\\  
   	Now consider the domain 
   	$\Omega_{ex}$
   	 which is an extension of the domain $\Omega$ (see Figure 2). Extend the function $\phi$ into $\Omega_{ex}$ by defining it zero outside $\Omega,$ denote the extension also by $\phi.$ Extend $\psi$ into $\Omega_{ex}$ by the constant $C_{0}$ outside $\Omega.$ We denote the extension of $\psi$ by $\psi$ itself. It is not hard to verify that the extended pair $(\phi,\psi)\in V^{2}(\Omega_{ex})\times H^{1}(\Omega_{ex})/\mathbb{R},$ solves the eigenvalue problem \eqref{eivp1} in the extended domain $\Omega_{ex}.$ Finally the unique continuation property from \cite{lebeau} shows that $\phi=0$ in $\Omega_{ex},$ thus in particular on $\Omega.$ Hence we are done with the proof of the Hautus test \eqref{Hutus}.    
	\end{proof} 
From Theorem \ref{thstab} we know that the pair $(A,\mathcal{B})$ is stabilizable by a control $w_{c}\in L^{2}(0,\infty;\mathbb{R}^{N_{c}}).$ Hence there exists a control $u$ (of the form \eqref{fdc}) which belongs to the finite dimensional space $U_{0}$ (see \eqref{U0}) and stabilizes the pair $(A,B).$\\
 Now our aim is to construct $w_{c}$ such that it is given in terms of a feedback control law. For that we will study the stabilization of the extended system \eqref{2-7} in the following section. 
\subsection{Stabilization of the extended system \eqref{2-7} by a feedback control}\label{stabext}
\subsubsection{Evolution equation associated with the extended system \eqref{2-7}}
We set 
\begin{equation}\label{deftZ}
\begin{array}{l}
{\widetilde Z}=V^{0}_{n}(\Omega)\times \mathbb{R}^{N_{c}}.
\end{array}
\end{equation}
 Depending on the context the notation $I$ denotes the identity operator for all of the spaces $V^{0}_{n}(\Omega),$ $\mathbb{R}^{N_{c}}$ and $\widetilde{Z}.$ We equip the space ${\widetilde Z}$ with the inner product
$$({\widetilde {\zeta}_{1}},{\widetilde \zeta}_{2})_{\widetilde Z}=(\zeta_{1},\zeta_{2})_{V^{0}_{n}(\Omega)}+(w_{1},w_{2})_{\mathbb{R}^{N_{c}}},$$
where ${\widetilde {\zeta}_{1}}=({\zeta_{1}},w_{1})$ and ${\widetilde \zeta_{2}}=(\zeta_{2},w_{2}).$ 
\newline
We fix a positive constant $\gamma$ (where $\gamma$ is the constant appearing in the extended system \eqref{2-7}).\\
Now let us recall the representation \eqref{2-6} of the system \eqref{2-5}. In the same note it follows that $\widetilde{y}=(P{y},w_{c})$ is a solution to equation \eqref{2-7} iff $(Py,w_{c})$ solves the following set of equations
\begin{equation}\label{3-4}
\left\{ \begin{array}{ll}
{\widetilde {y}'}= {{\begin{pmatrix}
		P{y}\\
		w_{c}
		\end{pmatrix}}}'=\begin{pmatrix}
A &  \mathcal{B}\\
0 &  -{\gamma I}
\end{pmatrix}
\begin{pmatrix}
P{y}\\
w_{c}
\end{pmatrix}+\begin{pmatrix}
0\\
I
\end{pmatrix}\varphi_{c}+
\widetilde{{f}}\quad&\mbox{in}\quad (0,\infty),\\[1.mm]
{\widetilde {y}}(0)=\widetilde{y}_{0}=\begin{pmatrix}
{y}_{0}\\0
\end{pmatrix},&
\vspace{1.mm}\\
(I-P){y}=\sum\limits_{j=1}^{N_{c}}{w_{j}}(I-P)D_{A}{{g}_{j}}\quad&\mbox{in}\quad (0,\infty),
\end{array}\right.
\end{equation}
where ${\widetilde{f}}=(P{f},0)$ and recall the definition of $\mathcal{B}$ from \eqref{decB}. Now we define the operator $(\widetilde{A},\mathcal{D}(\widetilde{A}))$ in $\widetilde{Z}$ as follows
\begin{equation}\label{DAt}
\begin{array}{l}
\mathcal{D}(\widetilde A)=\{(\zeta,w_{c})\in {\widetilde Z}\mid A\zeta+\mathcal{B}w_{c}\in {V}^{0}_{n}(\Omega)\}\quad \mbox{and} \quad {\widetilde A}=\begin{pmatrix}
A & \mathcal{B}\\
0 & -{\gamma I}
\end{pmatrix}.
\end{array}
\end{equation}
As we have identified $V^{0}_{n}(\Omega)$ with its dual, the space $\widetilde{Z}$ and $\widetilde{Z}^{*}$ are also identified. We define the adjoint of $({\widetilde A},\mathcal{D}({\widetilde A}))$ in $\widetilde{Z}$ as follows
\begin{equation}\label{DAt*}
\begin{array}{l}
\mathcal{D}({\widetilde A}^{*})=\mathcal{D}(A^{*})\times {\mathbb R}^{N_{c}}=\mathcal{D}(A)\times {\mathbb R}^{N_{c}}\quad \mbox{and}\quad {\widetilde A}^{*}=\begin{pmatrix}
A^{*} & 0\\
\mathcal{B}^{*} & -{\gamma I}
\end{pmatrix}.
\end{array}
\end{equation}
\begin{remark}
	We emphasize that due to the compatibility condition involved in the definition \eqref{DAt}, $\mathcal{D}(\widetilde{A})$ can not be written as $\mathcal{D}(A)\times\mathbb{R}^{N_{c}}.$ Contrary to that one does not require any compatibility condition in defining the domain of $\widetilde{A}^{*}$ which is given by \eqref{DAt*}.
\end{remark}
\begin{thm}\label{t3-3}
	The operator $({\widetilde A},\mathcal{D}({\widetilde A}))$ is the infinitesimal generator of an analytic semigroup on ${\widetilde Z}.$
\end{thm}	
\begin{proof}
 We will prove that $({\widetilde A},\mathcal{D}({\widetilde A}))$ generates an analytic semigroup on ${\widetilde Z}$ by proving that $({\widetilde A}^{*},\mathcal{D}({\widetilde A}^{*}))$ generates an analytic semigroup on $\widetilde{Z}.$ This is enough since one has the following by using \cite[Theorem 2.16.5, p. 56]{hilleph} $$\|\mathcal{R}(\lambda,\widetilde{A})\|_{\mathcal{L}(\widetilde{Z})}=\|(\lambda I-\widetilde{A})^{-1}\|_{\mathcal{L}(\widetilde{Z})}=\|((\lambda I-\widetilde{A})^{-1})^{*}\|_{\mathcal{L}(\widetilde{Z})}=\|(\overline{\lambda} I-\widetilde{A}^{*})^{-1}\|_{\mathcal{L}(\widetilde{Z})}=\|\mathcal{R}(\overline{\lambda},\widetilde{A}^{*})\|_{\mathcal{L}(\widetilde{Z})},$$ where $\mathcal{R}(\lambda,\cdot)$ denotes the resolvent of the respective operator (see \cite[Section 2.16]{hilleph} for details on resolvent) and hence $({\widetilde A},\mathcal{D}({\widetilde A}))$ generates an analytic semigroup on $\widetilde{Z}$ follows from the fact that $({\widetilde A}^{*},\mathcal{D}({\widetilde A}^{*}))$ generates an analytic semigroup on $\widetilde{Z}$ as a consequence of \cite[p. 163, Def. 5.4.5]{tucsnak}.\\   
	Let us notice that the operator $\widetilde{A}^{*}$ can be decomposed as follows
	$$\widetilde{A}^{*}=\widetilde{A}_{1}+\widetilde{A}_{2},$$
	where 
	$$\widetilde{A}_{1}=\begin{pmatrix}
	A^{*} &  0\\
	0 &  -\gamma I
	\end{pmatrix}\quad\mbox{and}\quad \widetilde{A}_{2}=\begin{pmatrix}
	0 &  0\\
	\mathcal{B}^{*} &  0
	\end{pmatrix}.$$
Since $(A,\mathcal{D}(A))$ generates an analytic semigroup on $Z=V^{0}_{n}(\Omega)$ (see Lemma \ref{t2-1}), $(A^{*},\mathcal{D}(A^{*}))$ generates an analytic semigroup on $Z$ (follows from the argument used in the beginning of the proof). Consequently the operator 
$$\widetilde{A}^{0}_{1}=\begin{pmatrix}
A^{*} &  0\\
0 &  0
\end{pmatrix}$$
 generates an analytic semigroup on $\widetilde{Z}.$ Since $\widetilde{A}_{1}$ is a bounded perturbation of the operator $\widetilde{A}^{0}_{1},$ one uses \cite[Corollary 2.2, Section 3.2]{pazy} to conclude that $\widetilde{A}_{1}$ with domain $\mathcal{D}(A^{*})\times\mathbb{R}^{N_{c}}$ generates an analytic semigroup on $\widetilde{Z}.$ On the other hand the definition \eqref{decB} of $\mathcal{B}$ and part $(ii)$ of Proposition \ref{p2-4} furnish that $$\mathcal{B}^{*}\in\mathcal{L}(\mathcal{D}((\omega I-A^{*})^{\frac{3}{4}+\frac{\kappa}{2}},\mathbb{R}^{N_{c}})\quad \mbox{for all}\quad 0<\kappa<\frac{1}{2}.$$ 
 This implies the following
 \begin{equation}\label{A2d}
 \begin{array}{l}
 \widetilde{A}_{2}\in \mathcal{L}(\mathcal{D}((\omega I-A^{*})^{\frac{3}{4}+\frac{\kappa}{2}}\times \mathbb{R}^{N_{c}},\widetilde{Z})\quad \mbox{for all}\quad 0<\kappa<\frac{1}{2}.
 \end{array}
 \end{equation}
 Now observe that $(\widetilde{A}_{1}-\omega I)$ is a diagonal operator, hence the semigroup $e^{t(\widetilde{A}_{1}-\omega I)}$ on $\widetilde{Z}$ generated by $(\widetilde{A}_{1}-\omega I)$ is of the form
 $$e^{t(\widetilde{A}_{1}-\omega I)}(\zeta_{1},w_{1})=(e^{t(A^{*}-\omega I)}\zeta_{1},e^{t(-\gamma-\omega)I}w_{1}),\quad\mbox{for all}\quad (\zeta_{1},w_{1})\in\widetilde{Z}.$$
 Hence one can use the definition \cite[p.329]{fattorini} of the domain of fractional power to have the following
 \begin{equation}\label{fracpwr}
 \begin{array}{l}
 \mathcal{D}((\omega I-\widetilde{A}_{1})^{\frac{3}{4}+\frac{\kappa}{2}})=\mathcal{D}((\omega I-A^{*})^{\frac{3}{4}+\frac{\kappa}{2}})\times\mathbb{R}^{N_{c}}\quad\mbox{for all}\quad 0<\kappa<\frac{1}{2}.
 \end{array}
 \end{equation} 
 Finally in view of \eqref{A2d} and \eqref{fracpwr}, the result \cite[p. 420, Lemma 12.38]{renardy} furnish that $(\widetilde{A}^{*},\mathcal{D}(\widetilde{A}^{*}))$ is the infinitesimal generator of an analytic semigroup on $\widetilde{Z}.$ This in turn gives that $({\widetilde A},\mathcal{D}({\widetilde A}))$ generates an analytic semigroup on ${\widetilde Z}.$
\end{proof}
From the definition \eqref{DAt} of the operator $\widetilde{A}$ one can easily observe that the spectrum of $\widetilde{A}$ is discrete and is explicitly given as follows
$$spec(\widetilde{A})=spec(A)\cup\{-\gamma\}.$$
\subsubsection{Existence of a feedback control law}\label{efcl}
We introduce the notation
${{\widetilde J}}=(0,I).
$
Let us notice that ${{\widetilde J}}$ belongs to ${\mathcal L}(\mathbb{R}^{N_{c}},{\widetilde Z}).$ This section is devoted to the construction of a feedback control $\varphi_{c}$ which is able to stabilize the linear equation
\begin{equation}\label{linsys}
\begin{array}{l}
\widetilde{y}{'}=\widetilde{A}\widetilde{y}+{\widetilde J}\varphi_{c}\quad\mbox{in}\quad (0,\infty),\qquad \widetilde{y}(0)=\widetilde{y}_{0},
\end{array}
\end{equation}
which is obtained from \eqref{3-4}$_{1}$-\eqref{3-4}$_{2}$ after neglecting the non-homogeneous source term $\widetilde{f}.$
\begin{prop}\label{stabextr}
	The pair $({\widetilde {A}},{{\widetilde J}})$ is stabilizable. More precisely there exists a feedback operator $\mathcal{K}\in\mathcal{L}(\widetilde{Z},\mathbb{R}^{N_{c}})$ such that the operator $(\widetilde A+\widetilde{J}\mathcal{K})$ with domain $\mathcal{D}(\widetilde{A})$ generates an exponentially stable analytic semigroup on $\widetilde{Z}.$ 
\end{prop}
Before going into the proof of Proposition \ref{stabextr}, let us recall that the pair $(\widetilde{A},\widetilde{J})$ is stabilizable in $\widetilde{Z}$ iff for all $\widetilde{y}_{0}\in \widetilde{Z},$ there exists a control $\varphi_{c}\in L^{2}(0,\infty;\mathbb{R}^{N_{c}})$ such that the controlled system 
$${\widetilde{y}}'=\widetilde{A}\widetilde{y}+\widetilde{J}\varphi_{c}\,\,\mbox{in}\,\,(0,\infty),\qquad \widetilde{y}=\widetilde{y}_{0},$$
satisfies
$$\displaystyle \|\widetilde{y}\|_{L^{2}(0,\infty;\widetilde{Z})}<\infty.$$
\begin{proof}[Proof of Proposition \ref{stabextr}] 
	We check the stabilizability of the pair $(\widetilde{A},{\widetilde J})$ by verifying the Hautus criterion \cite[Theorem 1.2]{raypa} (one can also consult \cite[Chapter V]{ben} for related results): 
	\begin{equation}\label{Hutus2}
	\begin{array}{l}
	\mbox{ker}(\widetilde\lambda_{k} I-\widetilde{A}^{*})\cap\mbox{Ker}({\widetilde J}^{*})=\{0\}\quad\mbox{for all}\quad \widetilde\lambda_{k}\in spec(\widetilde{A})\quad\mbox{with}\quad\mbox{Re}\widetilde{\lambda}_{k}>0.
	\end{array}
	\end{equation}
	 Let us prove \eqref{Hutus2}. We consider $$\begin{pmatrix}
	 \phi\\w
	 \end{pmatrix}\in \mbox{ker}(\widetilde\lambda_{k} I-\widetilde{A}^{*})\cap\mbox{Ker}({\widetilde J}^{*}). $$
Recall that ${\widetilde J}^{*}(
\phi,w
)=w.$
This gives $w=0.$\\
Now use the relation $(
\phi,0
)\in\mbox{ker}(\widetilde\lambda_{k} I-\widetilde{A}^{*}),$ to obtain
 $$(\widetilde\lambda_{k} I-A^{*})\phi=\mathcal{B}^{*}\phi=0.$$
 Hence $\phi=0,$ since the pair $(A,\mathcal{B})$ is stabilizable (see Theorem \ref{thstab}). This furnishes the stabilizability of the pair $(\widetilde{A},{\widetilde J}).$\\
 We consider the following Riccati equation
 	\begin{equation}\label{are}
 	\left\{ \begin{array}{l}
 	\widetilde{\mathcal{P}}\in \mathcal{L}(\widetilde{Z},\widetilde{Z}),\quad \widetilde{\mathcal{P}}=\widetilde{\mathcal{P}}^{*}>0,\\
 	\widetilde{\mathcal{P}}\widetilde{A}+\widetilde{A}^{*}\widetilde{\mathcal{P}}-\widetilde{\mathcal{P}}{\widetilde J}{\widetilde J}^{*}\widetilde{{P}}=0,\\
 	\widetilde{\mathcal{P}}\,\mathrm{is\, invertible}.
 	\end{array}\right.
 	\end{equation}	 
 Using \cite[Theorem 3]{kes}, there exists a solution $\widetilde{\mathcal{P}}$ to the Riccati equation \eqref{are} and the operator $\mathcal{K}=-\widetilde{J}^{*}\widetilde{\mathcal{P}}\in\mathcal{L}(\widetilde{Z},\mathbb{R}^{N_{c}}),$ provides a stabilizing feedback for $(\widetilde{A},\widetilde{J}).$ The operator $(\widetilde{A}+\widetilde{J}\mathcal{K})$ with domain
 $$\mathcal{D}(\widetilde{A}+\widetilde{J}\mathcal{K})=\mathcal{D}(\widetilde{A})$$
 is the generator of an exponentially stable analytic semigroup on $\widetilde{Z}.$
\end{proof}
From now onwards we will not use the explicit expression of the feedback controller $\mathcal{K}$ which was constructed in the proof of Proposition \ref{stabextr}, in fact we will only use that $\mathcal{K}\in\mathcal{L}(\widetilde{Z},\mathbb{R}^{N_{c}})$ and $\mathcal{D}(\widetilde{A}+\widetilde{J}\mathcal{K})=\mathcal{D}(\widetilde{A}).$

\subsubsection{Stabilization of the closed loop extended system with a non homogeneous source term}
Using the feedback control $\mathcal{K}$, we write the equation \eqref{3-4}$_{1}$-\eqref{3-4}$_{2}$ as the following closed loop system
\begin{equation}\label{closedloop}
\left\{ \begin{array}{l}
{\widetilde {y}'}= \widetilde{A}\widetilde{y}+{\widetilde J}\mathcal{K}\widetilde{y}+
\widetilde{{f}}\quad\mbox{in}\quad (0,\infty),\\
{\widetilde {y}}(0)=\widetilde{y}_{0}.
\end{array}\right.
\end{equation}

 From now on the constant $K(>0)$ appearing in the inequalities will denote a generic positive constant which may change from line to line. If we want to specify a constant (to use it for later purpose) we will denote it by $K_{i},$ for some natural number $i.$
 \begin{lem}\label{l3-8}	
 	Let the following hold
 	\begin{equation}\label{cofy0}
 	\begin{array}{l}
 	\widetilde{f}\in L^{2}(0,\infty;V^{0}_{n}(\Omega)\times\mathbb{R}^{N_{c}})\quad\mbox{and}\quad \widetilde{y}_{0}\in V^{1}_{0}(\Omega)\times\{0\},
 	\end{array}
 	\end{equation}
 	 where $\{0\}$ denotes the zero element of $\mathbb{R}^{N_{c}}.$ Then the equation \eqref{closedloop} admits a unique solution in $$\widetilde{y}\in H^{1}(0,\infty;{\widetilde Z})\cap L^{2}(0,\infty;{\mathcal{D}}({\widetilde{A}}))$$ which obeys
 	\begin{equation}\label{3-24}
 	\begin{array}{l}
 	\| \widetilde{y} \|_{H^{1}(0,\infty;{\widetilde Z})\cap L^{2}(0,\infty;{\mathcal{D}}({\widetilde{A}}))}\leqslant K(\| \widetilde{y}_{0} \|_{{V}^{1}_{0}(\Omega)\times\mathbb{R}^{N_{c}}}+\| \widetilde{f} \|_{L^{2}(0,\infty;L^{2}(\Omega)\times{R}^{N_{c}})}),
 	\end{array}
 	\end{equation}
 	for some positive constant $K.$
 \end{lem}
 \begin{proof}
 	Observe that
 	$$\widetilde{y}_{0}\in V^{1}_{0}(\Omega)\times\{0\}=[\mathcal{D}(A),V^{0}_{n}(\Omega)]_{1/2}\times\{0\}\underset{(i)}{=} [\mathcal{D}(A)\times\{0\},V^{0}_{n}(\Omega)\times\{0\}]_{1/2}\underset{(ii)}{\subset}[\mathcal{D}(\widetilde{A}),\widetilde{Z}]_{1/2},$$
 	the steps $(i)$ and $(ii)$ in the calculation above directly follows by using the definition of interpolation spaces provided by using \cite[p. 92, Theorem 14.1]{liomag} and \cite[Remark 14.1]{liomag}.\\
 	Since
 	$$\widetilde{f}\in L^{2}(0,\infty;\widetilde{Z})\quad\mbox{and}\quad \widetilde{y}_{0}\in [\mathcal{D}(\widetilde{A}),\widetilde{Z}]_{1/2},$$ one can use the isomorphism theorem \cite[Part II, Section 3.6.3, Theorem 3.1]{ben} to conclude the proof of Lemma \ref{l3-8}.
 \end{proof}
\begin{corollary}\label{c3-8}	
	Let the following hold
	\begin{equation}\label{cofy00}
	\begin{array}{l}
	{f}\in L^{2}(Q_{\infty})\quad\mbox{and}\quad {y}_{0}\in V^{1}_{0}(\Omega).
	\end{array}
	\end{equation}
	 Then the equation
	\begin{equation}\label{114r}
    \left\{ \begin{array}{ll}
    \displaystyle \frac{\partial {y}}{\partial t}-\beta {y}-\nu \Delta {y}+ ({v}_s \cdot \nabla){y}+({y} \cdot \nabla){v}_s +\nabla  q={f}\quad &\mbox{in}\quad Q_{\infty},
    \vspace{1.mm}\\
    \mathrm{div}{y}=0\quad &\mbox{in}\quad Q_{\infty},
    \vspace{2.mm}\\
    {y}=0\quad &\mbox {on} \quad(\Gamma_{0}\cup\Gamma_{out}) \times (0,\infty),\\[1.mm]
    {y}=\sum\limits_{j=1}^{N_{c}}w_{j}(t){{g}_{j}}(x) \quad &\mbox{on} \quad \ \Gamma_{in} \times (0,\infty),
    \vspace{1.mm}\\
    {y}(x,0)={y}_0\quad &\mbox{in}\quad\Omega,\\[2.mm]
    \displaystyle
    w_{c}^{'}+{\gamma }w_{c}-\mathcal{K}(Py,w_{c})=0\quad &\mbox{in}\quad (0,\infty),
    \vspace{1.mm}\\
    w_{c}(0)=0\quad&\mbox{in}\quad\Omega,
    \end{array}\right.
    \end{equation}
	$w_{c}=(w_{1},...,w_{N_{c}})$ and $g_{j},$ for all $1\leqslant j\leqslant N_{c}$ are defined in \eqref{basisU0},
	 admits a unique solution $(
	 y,w_{c}
	 )$ in $V^{2,1}(Q_{\infty})\times H^{1}(0,\infty;\mathbb{R}^{N_{c}})$ and the pair $
	 (y,w_{c})
	 $ obeys the following estimate
	\begin{equation}\label{esty}
	\begin{array}{l}
	\| (y,w_{c})\|_{V^{2,1}(Q_{\infty})\times H^{1}(0,\infty;\mathbb{R}^{N_{c}})}\leqslant K_{1}(\| {y}_{0} \|_{{V}^{1}_{0}(\Omega)} +\| {f} \|_{L^{2}(Q_{\infty})}),
	\end{array}
	\end{equation}
	for some positive constant $K_{1}.$\\
	In addition, there exists a constant $K_{2}>0$ such that the control
	 \begin{equation}\label{concor}
	 \begin{array}{l}
	 u(x,t)=\sum\limits_{j=1}^{N_{c}}w_{j}(t){g_{j}}(x),
	 \end{array}
	 \end{equation}
	 satisfies the following estimate 
	 	\begin{equation}\label{dofcon}
	 	\begin{array}{l}
	 	\|{u}(x,t) \|_{L^{\infty}(\Sigma_{\infty})}\leqslant K_{2}(\|{y}_{0} \|_{{V}^{1}_{0}(\Omega)}+\|{f} \|_{L^{2}(Q_{\infty})}).
	 	\end{array}
	 	\end{equation}
\end{corollary}
\begin{proof}
Using the notations used in \eqref{3-4}, one observes that $\|\widetilde{y}_{0}\|_{V^{1}_{0}(\Omega)\times\mathbb{R}^{N_{c}}}=\|(y_{0},0)\|_{V^{1}_{0}(\Omega)\times\mathbb{R}^{N_{c}}}=\|y_{0}\|_{V^{1}_{0}(\Omega)}$ and $\|\widetilde{f}\|_{L^{2}(0,\infty;L^{2}(\Omega)\times\mathbb{R}^{N_{c}})}=\|(Pf,0)\|_{L^{2}(0,\infty;L^{2}(\Omega)\times\mathbb{R}^{N_{c}})}=\|Pf\|_{L^{2}(Q_{\infty})}.$ Since the closed loop system \eqref{closedloop} along with \eqref{3-4}$_{3}$ is the operator representation of \eqref{114r}, one can use Lemma \ref{l3-8} (particularly the estimate \eqref{3-24}) to obtain the following 
	\begin{equation}\label{esPy}
	\begin{array}{l}
	\| (P{y},w_{c}) \|_{H^{1}(0,\infty;V^{0}_{n}(\Omega)\times\mathbb{R}^{N_{c}})\cap L^{2}(0,\infty;\mathcal{D}({\widetilde{A}}))}\leqslant K(\|y_{0}\|_{V^{1}_{0}(\Omega)}+\|Pf\|_{L^{2}(Q_{\infty})}).
	\end{array}
	\end{equation}
    Now we estimate 
	$$(I-P)y=\sum\limits_{j=1}^{N_{c}}w_{j}(I-P)D_{A}{{g}_{j}}.$$
	We know that there exists a positive constant $K$ such that for all $1\leqslant j \leqslant N_{c}$
	\begin{equation}\label{estI-P1}
	\begin{array}{l}
	\|D_{A}{g_{j}}\|_{V^{2}(\Omega)}\leqslant K\|{g_{j}}\|_{H^{3/2}(\Gamma)}\leqslant K.
	\end{array}
	\end{equation}
	Estimates \eqref{estI-P1} and \eqref{esPy} yield
	\begin{equation}\label{estI-P2}
	\begin{array}{l}
	\|(I-P)y\|_{H^{1}(0,\infty;H^{2}(\Omega))}\leqslant K(\|y_{0}\|_{V^{1}_{0}(\Omega)}+\|Pf\|_{L^{2}(Q_{\infty})}).
	\end{array}
	\end{equation} 
	Once again using \eqref{esPy} and \eqref{estI-P2} one has
	\begin{equation}\label{inregy}
	\begin{split}
	\displaystyle
	\|y\|_{H^{1}(0,\infty;L^{2}(\Omega))}&\leqslant \|Py\|_{H^{1}(0,\infty;V^{0}_{n}(\Omega))}+\|(I-P)y\|_{H^{1}(0,\infty;H^{2}(\Omega))}\\
	&\leqslant K(\|y_{0}\|_{V^{1}_{0}(\Omega)}+\|Pf\|_{L^{2}(Q_{\infty})}).
	\end{split}
	\end{equation}
	To prove higher regularity of $y$ we will use a bootstrap argument. First we write \eqref{114r}$_{1}$-\eqref{114r}$_{5}$ as follows
	\begin{equation}\label{114r*}
	\left\{ \begin{array}{ll}
	\displaystyle -\nu \Delta {y}+\nabla  q=f^{*}\quad &\mbox{in}\quad Q_{\infty},
	\vspace{1.mm}\\
	\mathrm{div}{y}=0\quad &\mbox{in}\quad Q_{\infty},
	\vspace{2.mm}\\
	{y}=0\quad &\mbox {on} \quad(\Gamma_{0}\cup\Gamma_{out}) \times (0,\infty),\\[1.mm]
	{y}=\sum\limits_{j=1}^{N_{c}}w_{j}(t){{g}_{j}}(x) \quad &\mbox{on} \quad \ \Gamma_{in} \times (0,\infty),
	\vspace{1.mm}\\
	{y}(x,0)={y}_0\quad &\mbox{in}\quad\Omega,
	\end{array}\right.
	\end{equation}
	where $$f^{*}=f-\frac{\partial {y}}{\partial t}+\beta {y}- ({v}_s \cdot \nabla){y}-({y} \cdot \nabla){v}_s.$$
	Using \eqref{cofy00} and \eqref{inregy} we obtain that $f^{*}\in L^{2}(0,\infty;H^{-1}(\Omega))$ and the following holds
	\begin{equation}\label{estf*}
	\begin{array}{l}
	\|f^{*}\|_{L^{2}(0,\infty;H^{-1}(\Omega))}\leqslant K(\|f\|_{L^{2}(Q_{\infty})}+\|y_{0}\|_{V^{1}_{0}(\Omega)}).
	\end{array}
	\end{equation}
	Also 
	\begin{equation}\label{regcon}
	\begin{array}{l}
	\sum\limits_{j=1}^{N_{c}}w_{j}(t){{g}_{j}}(x)\in H^{1}(0,\infty;V^{0}(\Gamma)\cap C^{\infty}(\Gamma)),
	\end{array}
	\end{equation}
	($g_{j}\in V^{0}(\Gamma)$ follows from the definition \eqref{basisU0} and Proposition \ref{p2-4} whereas $C^{\infty}(\Gamma)$ regularity follows from Lemma \ref{smgci}). Hence one can use \cite[Theorem IV.5.2]{boy}, \eqref{estf*} and \eqref{esPy} to get $y\in L^{2}(0,\infty;V^{1}(\Omega))$ and the following inequality
	\begin{equation}\label{incregy}
	\begin{array}{l}
	\|y\|_{L^{2}(0,\infty;V^{1}(\Omega))}\leqslant K(\|f\|_{L^{2}(Q_{\infty})}+\|y_{0}\|_{V^{1}_{0}(\Omega)}).
	\end{array}
	\end{equation}
	The regularity \eqref{inregy}, \eqref{incregy} and \eqref{cofy00} furnish $f^{*}\in L^{2}(Q_{\infty})$ and
	\begin{equation}\label{incregf}
	\begin{array}{l}
	\|f^{*}\|_{L^{2}(Q_{\infty})}\leqslant K(\|f\|_{L^{2}(Q_{\infty})}+\|y_{0}\|_{V^{1}_{0}(\Omega)}).
	\end{array}
	\end{equation} 
	In view of \eqref{regcon} and \eqref{incregf} one further obtains that $y\in L^{2}(0,\infty;V^{2}(\Omega))$ (using the regularity result from \cite{Osborn}) and the following
	\begin{equation}\label{finalregy}
	\begin{array}{l}
	\|y\|_{L^{2}(0,\infty;V^{2}(\Omega))}\leqslant K(\|f\|_{L^{2}(Q_{\infty})}+\|y_{0}\|_{V^{1}_{0}(\Omega)}).
	\end{array}
	\end{equation}
	Hence $y\in V^{2,1}(Q_{\infty})$ and using \eqref{inregy} and \eqref{finalregy} one has the following
	\begin{equation}\label{esty1}
	\begin{array}{l}
	\|y\|_{V^{2,1}(Q_{\infty})}\leqslant K(\|y_{0}\|_{V^{1}_{0}(\Omega)}+\|f\|_{L^{2}(Q_{\infty})}).
	\end{array}
	\end{equation}
	Finally, \eqref{esPy} and \eqref{esty1} provides the desired estimate \eqref{esty}.\\
	Since ${g_{j}}\in L^{\infty}(\Gamma),$ the estimate \eqref{dofcon} readily follows by using \eqref{esty}.
\end{proof}
The following result justifies our choice of denoting the inflow and outflow boundary of $v_{s}$ and a perturbation of $v_{s}$ using the same notation.
\begin{corollary}\label{p3.0.2} If we take 
	\begin{equation}\label{smalasum}
	(\|{y}_{0} \|_{{V}^{1}_{0}(\Omega)}+\| {f} \|_{L^{2}(Q_{\infty})})\leqslant \frac{L(1-L)}{2K_{2}},
	\end{equation}
	 where $K_{2}$ is the constant in \eqref{dofcon}, then 
	 $$\|{y}\mid_{\Sigma_{\infty}}\|_{L^{\infty}(\Sigma_{\infty})}\leqslant \frac{L(1-L)}{2}$$ and hence in particular for all $t>0,$ 
	\begin{equation}\label{iny0f}
	\left\{ \begin{array}{ll}
	(e^{-\beta t}{y}(\cdot,t)+{v}_{s})\cdot{n}<0\,\,&\mbox{on}\,\,\Gamma_{in},\\
	(e^{-\beta t}{y}(\cdot,t)+v_{s})\cdot {n}=0\,\,&\mbox{on}\,\, \Gamma_{0},\\
    (e^{-\beta t}{y}(\cdot,t)+{v}_{s})\cdot{n}>0\,\,&\mbox{on}\,\,\Gamma_{out},
	\end{array}\right.
	\end{equation}
	where $(
	y,w_{c}
	)$ is the solution to \eqref{114r}. 
	This means that for all time $t>0$, $\Gamma_{in}$ and $\Gamma_{out}$ are still the inflow and the outflow boundary for the perturbed vector field $(v_{s}+e^{-\beta t}{y}).$
\end{corollary}	
\begin{proof}
	The proof is a direct consequence of Corollary \ref{c3-8}, in particular the estimate \eqref{dofcon}.
\end{proof}
\section{Stability of the continuity equation}\label{density}
This section is devoted to the study of the transport equation satisfied by density which is modeled by {\eqref{2-2}}$_{1}$ together with \eqref{2-2}$_{2}$ and \eqref{2-2}$_{3}$. This equation is linear in $\sigma$ but nonlinear in $(\sigma,{y}).$ First let us briefly discuss the stabilization of the linearized transport equation modeling the density with zero inflow boundary condition. This will give us an idea about how to obtain analogous results for its nonlinear counterpart. 
\subsection{Comments on the linear transport equation at velocity $v_{s}$}\label{lte}
The linearized continuity equation with the zero inflow boundary condition
is given by
\begin{equation}\label{4-1}
\left\{ \begin{array}{ll}
\displaystyle
\frac{\partial \sigma}{\partial t}+({v}_s \cdot \nabla)\sigma-\beta\sigma=0\quad&\mbox{in}\quad Q_{\infty},
\vspace{1.mm}\\
\displaystyle
\sigma(x,t)=0 \quad &\mbox{on} \quad \Gamma_{in} \times(0,\infty),
\vspace{1.mm}\\
\displaystyle
\sigma(x,0)=\sigma_0\quad&\mbox{in}\quad\Omega.
\end{array}\right.
\end{equation}
We can explicitly solve \eqref{4-1} to obtain
\begin{equation}\label{4-2}
\sigma(x,t)=\left\{ \begin{array}{l}
\displaystyle
e^{\beta t}\sigma_{0}(x_{1}-(x_{2}(1-x_{2}))t,x_{2})\quad \mbox{for} \quad t\leqslant \frac{1}{(x_{2}(1-x_{2}))}x_{1},
\vspace{4.mm}\\
0\quad \mbox{for}\quad t>\displaystyle \frac{1}{(x_{2}(1-x_{2}))}x_{1},
\end{array}\right.
\end{equation}
for all $(x_{1},x_{2})\in\Omega$.
In particular if we assume that $\sigma_{0}$ satisfies the condition \eqref{1-4}, the solution $\sigma$ to \eqref{4-1} vanishes after some finite time $T_{A_{1}}=\frac{d}{A_{1}(1-A_{1})}.$ Hence we see that with zero inflow boundary condition the solution of the linearized transport equation is automatically stabilized (in fact controlled) after some finite time. The equation \eqref{4-1} is just a prototype of the transport equation \eqref{2-2}$_{1,2,3}$ exhibiting similar property and we will discuss this in the following section. 	
\subsection{Stability of the transport equation \eqref{2-2} satisfied by density}\label{nlte}
We consider the transport equation satisfied by the density with the nonlinearity $({y}\cdot\nabla)\sigma.$ We assume that $\|y\|_{V^{2,1}(Q_{\infty})}$ is small enough and the following holds
\begin{equation}\label{cony.n}
\begin{array}{l}
(e^{-\beta t}{y}+{v}_{s})\cdot{n}<0\,\,\mbox{on}\,\,\Gamma_{in},\,\,(e^{-\beta t}{y}+v_{s})\cdot {n}=0\,\,\mbox{on}\,\,\Gamma_{0},\,\,\mbox{and}\,\,(e^{-\beta t}{y}+{v}_{s})\cdot{n}>0\,\, \mbox{on}\,\, \Gamma_{out}.
\end{array}
\end{equation}
Recalling Corollary \ref{p3.0.2}, the condition \eqref{cony.n} is automatically satisfied when $y$ solves \eqref{2-2}$_{4}$-\eqref{2-2}$_{8}$ and \eqref{smalasum} holds. Notice that the role of the condition \eqref{cony.n} is only to guarantee that even if we perturb the vector field $v_{s}$ by adding $e^{-\beta t}{y},$ the inflow boundary of the fluid remains unchanged.
\newline 
Here the transport equation satisfied by the density is given by
\begin{equation}\label{4-3}
\left\{ \begin{array}{ll}
\displaystyle
\frac{\partial \sigma}{\partial t}+(({v}_s+e^{-\beta t}{y}) \cdot \nabla)\sigma-\beta\sigma=0\quad&\mbox{in}\quad Q_{\infty},
\vspace{1.mm}\\
\displaystyle
\sigma (x,t)=0 \quad &\mbox{on} \quad \Gamma_{in} \times(0,\infty),
\vspace{1.mm}\\
\displaystyle
\sigma (x,0)=\sigma_0\quad&\mbox{in}\quad\Omega,
\end{array}\right.
\end{equation}
where ${y}$ is in ${V}^{2,1}(Q_{\infty}),$ \eqref{cony.n} holds, $\sigma_{0}\in L^{\infty}(\Omega)$ and satisfies the condition \eqref{1-4} (recall from \eqref{y0s0} that $\sigma_{0}=\rho_{0}$). Provided $y$ is suitably small in the norm $V^{2,1}(Q_{\infty})$, \eqref{4-1} can be seen as an approximation of \eqref{4-3}, and as we will see in Theorem \ref{t4-1}, solutions of \eqref{4-1} and of \eqref{4-3} share some similar behavior.
\newline
We are in search of a unique solution of \eqref{4-3} in the space $L^{\infty}(Q_{\infty})$. In the following discussion we will borrow several results from \cite{boy} on the existence, uniqueness and stability of the continuity equation. For later use, we shall consider a general transport equation of the form
\begin{equation}\label{4-3-bis}
\left\{ \begin{array}{ll}
\displaystyle
\frac{\partial \sigma}{\partial t}+(v \cdot \nabla)\sigma-\beta\sigma=0\quad&\mbox{in}\quad Q_{\infty},
\vspace{1.mm}\\
\displaystyle
\sigma (x,t)=0 \quad &\mbox{on} \quad \Sigma_{in,v,\infty},
\vspace{1.mm}\\
\displaystyle
\sigma (x,0)=\sigma_0\quad&\mbox{in}\quad\Omega,
\end{array}\right.
\end{equation}
where $v$ is a divergence free vector field in $L^2(0,\infty; V^2(\Omega))$, and 
$$
	\Sigma_{in,v,T} = \{(x,t) \in   \Gamma \times (0,T)\suchthat v(x,t) \cdot n(x) < 0 \}
$$
First let us define the notion of weak solution for the transport equation \eqref{4-3-bis}$_{1}$.
\begin{mydef}\label{dd1}
	Let $T>0$ and $v$ a divergence free vector field such that $v \in L^{2}((0,T); V^2(\Omega))$. A function $\sigma \in L^{\infty}( Q_{T})$ is said to be a weak solution of \eqref{4-3-bis}$_{1}$ if the following is true
	$$ \int\limits_{0}^{T}\int\limits_{\Omega}\sigma(\partial_{t}\phi +v\cdot \nabla \phi+\beta\phi)dxdt=0,$$
	for any test function $\phi \in C^{\infty}(\bar{\Omega}\times [0,T])$ with $\phi (\cdot,T)=0=\phi(\cdot,0)$ in $\Omega$ and $\phi=0$ on $\Sigma_{T}.$ 
\end{mydef}
One can interpret the boundary trace of a weak solution (as defined in Definition \ref{dd1}) of \eqref{4-3-bis}$_{1}$ in a weak sense. Following \cite{boy} we introduce some notations which will be used to define the trace of a weak solution of \eqref{4-3-bis}$_{1}.$\\
Let $m$ denote the boundary Lebesgue measure on $\Gamma.$ Now for any $T>0,$ associated to the vector field $v$, we introduce the measure 
$$d\mu_{v}=(v \cdot {n})dmdt\quad \mbox{on}\quad \Sigma_{T}$$
and denote by $d\mu_{v}^{+}$ (respectively $d\mu_{v}^{-}$) its positive (resp. negative) part in such a way that $|d\mu_{v}| =d\mu_{v}^{+}+d\mu_{v}^{-}.$ The support of $d\mu_{v}^{+}$ (resp. $d\mu_{v}^{-}$) is the outflow (resp. inflow) part of $\Sigma_{T}$ corresponding to the vector field $v$.\\
 The following two theorems, Theorem \ref{trace} and Theorem \ref{l4-2}, are stated in \cite{boy} for a weaker assumption on the velocity field $v.$ Here we state the results with $v\in L^2(0,T; V^{2}(\Omega))$ for the particular equation \eqref{4-3-bis}.
\begin{thm}\label{trace}
\cite[Theorem \textrm{VI.1.3}]{boy} Let $T>0,$ $v\in L^{2}(0,T; V^2(\Omega))$ and $\sigma\in L^{\infty}(Q_{T})$ be a weak solution of \eqref{4-3-bis}$_{1}$ in the sense of the Definition \ref{dd1}. Then the following hold:\\
$(i)$\,The function $\sigma$ lies in $C^{0}([0,T],L^{p}(\Omega))$ for all $1\leqslant p<+\infty.$\\
$(ii)$\,There exists a unique function $\gamma_{\sigma}\in L^{\infty}(\Sigma_{T} ,|d\mu_{v}|)$ such that for any test function $\phi\in C^{0,1}(\bar{Q}_{T})$ and for any $[t_{0},t_{1}]\subset [0,T]$ we have 
\begin{equation}\label{4-4}
\begin{split}
\int\limits_{t_{0}}^{t_{1}}\int\limits_{\Omega}\sigma\left(\frac{\partial \phi}{\partial t}+v\cdot \nabla\phi +\beta \phi \right)dxdt&-\int\limits_{t_{0}}^{t_{1}}\int\limits_{\Gamma}\gamma_{\sigma} \phi d\mu_{v}\\
& +\int\limits_{\Omega}\sigma(t_{0})\phi(t_{0}) dx
-\int\limits_{\Omega}\sigma(t_{1})\phi(t_{1}) dx=0.
\end{split}
\end{equation}	
$(iii)$ The renormalization property: For any function $\xi: \mathbb{R}\to \mathbb{R}$ of class $C^1$, for any $\phi\in C^{0,1}(\bar{Q}_{T})$ and for any $[t_{0},t_{1}]\subset [0,T]$ we have 
\begin{multline}\label{Renormalized}
\int\limits_{t_{0}}^{t_{1}}\int\limits_{\Omega}\xi(\sigma)\left(\frac{\partial \phi}{\partial t}+v\cdot \nabla\phi \right)dxdt
+ \int\limits_{t_{0}}^{t_{1}}\int\limits_{\Omega}\beta \sigma \xi'(\sigma) \phi 
-\int\limits_{t_{0}}^{t_{1}}\int\limits_{\Gamma}\xi(\gamma_{\sigma})\phi d\mu_{v}\\
 +\int\limits_{\Omega}\xi(\sigma(t_{0}))\phi(t_{0}) dx
-\int\limits_{\Omega}\xi(\sigma(t_{1}))\phi(t_{1}) dx=0.
\end{multline}	
\end{thm}
The following theorem states some results on the well posedness of the weak solution $\sigma$ of the Cauchy-Dirichlet transport problem \eqref{4-3-bis}. 
\begin{thm}\label{l4-2}
	\cite[Theorem \textrm{VI.1.6}]{boy} Let $T>0,$ $\sigma_{0}\in L^{\infty}(\Omega)$ and $v\in L^{2}(0,T; V^2(\Omega))$. Then there exists a unique function $\sigma\in L^{\infty}(Q_{T})$ such that\\ 
	(i) The function $\sigma$ is a weak solution of the problem \eqref{4-3-bis}$_{1}$ in $Q_{T}$ in the sense of Definition \ref{dd1}.\\[1.mm] 
	(ii) The trace $\gamma_{\sigma}$ of $\sigma$ satisfies the inflow boundary condition,
	$\gamma_{\sigma}=0,$ $d\mu^{-}_{v}$ almost everywhere on $\Sigma_{in,v,T}$ and $\sigma$ satisfies the initial condition $\sigma(x,0)=\sigma_{0}$ in $\Omega$.\\
    In the following, we call this function $\sigma$ satisfying $(i)$ and $(ii),$ the solution of \eqref{4-3-bis}.\\[1.mm]
	(iii) Moreover for $0<t<T,$ the solution $\sigma$ of \eqref{4-3-bis} satisfies 
	\begin{equation}\label{esorho}
	\begin{array}{l}
	\|\sigma(\cdot,t)\|_{L^{\infty}(\Omega)}\leqslant \|\sigma_{0}\|_{L^{\infty}(\Omega)}e^{\beta t}.
	\end{array}
	\end{equation}
	\end{thm}
	
	Let us also recall, for later purpose, the following stability result for the transport equation with respect to its velocity field:
	\begin{lem}\label{l4-4}
	\cite[Theorem \textrm{VI.1.9}]{boy}	Let $T>0.$ Suppose that $\sigma_{0}\in L^{\infty}(\Omega)$ and let $\{{v}_{m}\}_{m}$ be a sequence of functions in ${L}^{2}(0,T; V^2(\Omega))$ such that there exists $v \in L^2(0,T; V^2(\Omega))$ such that 
		$${v}_{m} \xrightarrow[m\rightarrow \infty]{} {v}\quad\mbox{in}\quad L^{1}(Q_{T}),
		\, \text{ and } \, 
			v_m \cdot n \xrightarrow[m\rightarrow \infty]{} {v.n} \quad\mbox{in}\quad L^{1}(\Sigma_{T}).
		$$
		Now suppose that $\sigma_{m}\in L^{\infty}(Q_{T})$ is the unique weak solution (in sense of Definition \ref{dd1}.) of the following initial and boundary value problem 
		\begin{equation}\label{4-12}
		\left\{ \begin{array}{ll}
		\displaystyle
		\frac{\partial \sigma_{m}}{\partial t}+(v_m \cdot \nabla)\sigma_{m}- \beta\sigma_{m} = 0\quad&\mbox{in}\quad Q_{T},
		\vspace{1.mm}\\
		\sigma_{m} (x,t)=0 \quad& \mbox{on}\quad \Sigma_{in,v_{m},T},
		\vspace{1.mm}\\
		\sigma_{m} (x,0)=\sigma_0\quad &\mbox{in}\quad\Omega.
		\end{array}\right.
		\end{equation}
		If we denote by $\sigma$ the unique solution to the transport problem \eqref{4-3-bis} in $Q_{T},$ then we have
		\begin{equation}\label{4-13}
		\begin{array}{l}
		\sigma_{m}\xrightarrow[m\rightarrow \infty]{} \sigma\quad\mbox{in}\quad C^{0}([0,T],L^{p}(\Omega)),\quad\mbox{for any}\quad 1\leqslant p<+\infty.
		\end{array}
		\end{equation} 			
	\end{lem}

Now we state the main theorem of this section:
\begin{thm}\label{t4-1}
   Let $A_{1}\in(0,\frac{1}{2})$ and $T_{1}>T_{A_{1}}=\frac{d}{A_{1}(1-A_{1})}.$ There exists a constant $K_{3}>0$ such that if $y\in V^{2,1}(Q_{\infty})$ satisfies
   \begin{equation}\label{yK3}
   \begin{array}{l}
	\|{y}\|_{{V}^{2,1}(Q_{\infty})}<K_{3},
	\end{array}
	\end{equation}
	\eqref{cony.n} holds, ${\sigma}_{0}\in L^{\infty}(\Omega)$ and satisfies the condition \eqref{1-4}, the solution $\sigma$ of equation \eqref{4-3} satisfies the following
	\begin{equation}\label{fiessig}
	\begin{array}{l}
	(i)\, \forall t<T_{1},\,\, \sigma(\cdot,t)\,\mbox{satisfies the estimate }\,\eqref{esorho},\\
	(ii)\, \forall t\geqslant T_{1},\,\, \|\sigma(\cdot,t)\|_{L^{\infty}(\Omega)}=0.
	\end{array}
	\end{equation}
\end{thm}
\begin{proof}[Proof of Theorem \ref{t4-1}]
	 Item $(i)$ of \eqref{fiessig} is automatically satisfied as a consequence of item $(iii),$ Theorem \ref{l4-2}. 
	 \newline
	 We thus focus on the proof of item $(ii)$ of Theorem \ref{t4-1}. Let $T_{1}>T_{A_{1}}=\frac{d}{A_{1}(1-A_{1})}$ be fixed. Our approach will be based on the flow $X$ corresponding to the vector field  $v_{s}+e^{-\beta t}{y}$. In order to introduce it in a more convenient manner, we first extend the domain into $\R^2$.
	Observe that the definition of $v_{s}$ can be naturally extended to $\R^2$ into a Lipschitz function by setting $v_s(x_1,x_2) = v_s(x_2)$ if $x_2 \in (0,1)$ and $0$ if $x_2 \in \R \setminus(0,1)$. We denote this extension by $v_{s}$ itself. For the following analysis we use the functional space $$H^{2,1}(\R^2 \times (0,\infty))=L^{2}(0,\infty;H^{2}(\R^2)\cap H^{1}(0,\infty;L^{2}(\R^2))$$
	(this is consistent with the notations defined in Section \ref{funcframe}). Now we introduce an extension operator $\mathbb{E}$ from $\Omega$ to $\R^2$.
	$$
	\mathbb{E}:L^2(\Omega)\longrightarrow L^2(\R^2)
	$$
	such that: 
	\begin{itemize}
		\item for every $y\in L^{2}(\Omega),$ $\mathbb{E}y\mid_{\Omega}=y$, 	
		\item the restriction of $\mathbb{E}$ to $H^2(\Omega)$ defines a linear operator from $H^2(\Omega)$ to $H^2(\R^2)$,
		\item the restriction of $\mathbb{E}$ to $H^2(\Omega)\cap W^{1,\infty} (\Omega)$ defines a linear operator from $H^2(\Omega)\cap W^{1,\infty} (\Omega)$ to $H^2(\R^2)\cap W^{1,\infty} (\R^2)$,		 
%
	\end{itemize}
	 The existence of such an extension operator is a direct consequence of \cite[Theorem 2.2]{liomag}. 
	\\
	We now introduce the flow $X(x,t,s)$ defined for $x \in \R^2$ and  $(t,s)\in[0,\infty)^{2},$ by the following differential equation:
	\begin{equation}\label{4-5}
	\left\{ \begin{array}{l}
	\displaystyle
	\frac{\partial X(x,t,s)}{\partial t}=(v_{s}+e^{-\beta t}\mathbb{E}{y})(X{(x,t,s)},t),
	\vspace{1.mm}\\
	X(x,t,s)\suchthat_{t=s}=x\in \R^2.
	\end{array}\right.
	\end{equation}
    The integral formulation of \eqref{4-5} can be written as follows
	\begin{equation}\label{4-6}
	\begin{array}{l}
	\displaystyle
	\forall (x,t,s) \in \R^2 \times [0,\infty)^{2}, 
	\quad X(x,t,s)=x+\int\limits_{s}^{t}(v_{s}+e^{-\beta t}\mathbb{E}{y})(X(x,\theta,s),\theta)d\theta.
	\end{array}
	\end{equation}
	As the vector field
	$$(v_{s}+e^{-\beta t}\mathbb{E}{y})\in L^2(0,\infty; {W}^{1,\infty}(\R^2)) + {H}^{2,1}(\R^2\times(0,\infty)),$$  
	due to the Osgood condition (see \cite{zuazua} and \cite[Theorem 3.7]{Bahouri-Chemin-Danchin}) we know that equation \eqref{4-6} has a unique continuous solution. 
	Similarly, we introduce the flow $X_0$ corresponding to the vector field $v_{s}$ as the solution of the following differential equation:
	\begin{equation}\label{4-7}
	\left\{ \begin{array}{l}
	\displaystyle
	\frac{\partial X_{0}(x,t,s)}{\partial t}=v_{s}(X_{0}{(x,t,s)},t),
	\vspace{1.mm}\\
	X_{0}(x,t,s)\suchthat_{t=s}=x\in\R^2.
	\end{array}\right.
	\end{equation}
	As $v_{s}$ is Lipschitz, the flow, which can also be seen as the solution of 
	\begin{equation}\label{4-8}
	\begin{array}{l}
	\displaystyle
	X_{0}(x,t,s)=x+\int\limits_{s}^{t}v_{s}(X_{0}(x,\theta,s),\theta)d\theta, \quad (x,t,s) \in \R^2 \times (0,\infty)^2,
	\end{array}
	\end{equation}
	is well defined in classical sense.
	\begin{lem}\label{ld3}
		 Let $T>0.$ There exists a constant $K_{4}=K_{4}(T)>0$ such that for all $y\in V^{2,1}(Q_\infty),$ $(t,s)\in[0,T]^{2}$ and $x\in \R^2,$ the solutions of \eqref{4-5} and \eqref{4-7} satisfy the following
		\begin{equation}\label{4-9}
		\begin{array}{l}
		\mid X(x,t,s)-X_{0}(x,t,s)\mid < K_{4}(T)\|{y}\|_{{V}^{2,1}(Q_{\infty})}.
		\end{array}
		\end{equation}
	\end{lem}
	\begin{proof}
		The proof of Lemma \ref{ld3} can be performed by using arguments which are very standard in the literature. For the convenience of the reader we include the proof. \\ 
		{\bf 1.} 
		As $H^{2}(\R^2)$ is embedded in $L^{\infty}(\R^2),$ using H\"{o}lder's inequality we can at once obtain the following estimate for all $(t,s)\in[0,T]^{2}$ and $x\in\R^2$,
		$$	\left|\int\limits_{s}^{t}e^{-\beta \theta} \mathbb{E}{y}(X(x,\theta,s),\theta) d\theta\right| \leqslant {K}\|\mathbb{E}{y}\|_{{H}^{2,1}(\R^2\times(0,\infty))},$$ for some constant $K>0.$
		\vspace{1.mm}\\
		{\bf 2.} Subtracting \eqref{4-6} from \eqref{4-8}, we get, for all $(t,s)\in[0,\infty)^{2}$ and $x\in\R^2$,
		\begin{equation}\nonumber
		\begin{aligned}
		|X(x,t,s)-X_{0}(x,t,s)|  & \leqslant \left|\int\limits_{s}^{t} |v_{s}(X(x,\theta,s),\theta)-v_{s}(X_{0}(x,\theta,s),\theta)| d\theta\right| + \left|\int\limits_{s}^{t}e^{-\beta \theta} |\mathbb{E}{y}(X(x,\theta,s),\theta)|d\theta\right|\\
		& \leqslant \|\nabla v_{s}(.) \|_{L^{\infty}(\Omega)} \left|\int\limits_{s}^{t} |X(x,\theta,s)-X_{0}(x,\theta,s)| d\theta\right| +{K}\|\mathbb{E}{y}\|_{{H}^{2,1}(\R^2 \times(0,\infty))}.
		\end{aligned}
		\end{equation}
		Since $\mathbb{E}$ is a bounded operator from $L^{2}(\Omega)$ to $L^{2}(\mathbb{R}^{2})$ and from $H^{2}(\Omega)$ to $H^{2}(\mathbb{R}^{2}),$ there exists a constant $K>0$ such that
		\begin{equation}\label{pregronlt}
		\begin{array}{l}
		\displaystyle
		|X(x,t,s)-X_{0}(x,t,s)|\leqslant  \|\nabla v_{s}(.) \|_{L^{\infty}(\Omega)} \left|\int\limits_{s}^{t} |X(x,\theta,s)-X_{0}(x,\theta,s)| d\theta\right|+ K\|y\|_{V^{2,1}(Q_{\infty})}.
		\end{array}
		\end{equation}
		Now we can use Gr\"{o}nwall's inequality to obtain \eqref{4-9}. 
	\end{proof}
	Recall that the solution of \eqref{4-1} vanishes after some finite time $T_{A_{1}}=\frac{d}{A_{1}(1-A_{1})}.$ At the same time Lemma \ref{ld3} suggests that for any finite time $T>0,$ the flow $X_{0}(x,t,s)$ stays uniformly close to $X(x,t,s)$ in $\R^2\times(0,T)$ provided $\|y\|_{V^{2,1}(Q_{\infty})}$ is small enough. In view of these observations, in the following we design a Lyapunov functional corresponding to a localized energy, to prove that $\sigma$ vanishes after the time $T_1 > T_{A_1}$ when $\|y\|_{V^{2,1}(Q_{\infty})}$ is small enough, which will prove  Theorem \ref{t4-1}.\\ 
	Let $\varepsilon$ be a fixed positive constant in $(0,A_1)$ such that
	\begin{equation}\label{varepsilon}
	\begin{array}{l}
	\displaystyle
	T_{1}=\frac{d+\varepsilon}{(A_{1}-\varepsilon)(1-A_{1}+\varepsilon)}.
	\end{array}
	\end{equation}
	Our primary goal is to prove that, for a velocity field $y$ satisfying \eqref{cony.n} and such that $ \|y\|_{V^{2,1}(Q_{T_1})}$ is small enough and an initial condition ${\sigma}_{0}\in L^{\infty}(\Omega)$ satisfying \eqref{1-4}, the solution $\sigma$ of \eqref{4-3} satisfies 
	\begin{equation}
		\label{Sigma-t-1=0}
		\displaystyle \sigma(x,T_{1})=0\quad\mbox{for all}\quad x\in\Omega.
	\end{equation}
In fact, the condition \eqref{cony.n} does not play any role. We shall thus prove a slightly more general result: there exists $K_{3}>0,$ such that for any velocity field $y$ such that $ \|y\|_{V^{2,1}(Q_{T_1})}\leq K_3$ and any initial condition ${\sigma}_{0}\in L^{\infty}(\Omega)$ satisfying \eqref{1-4}, the solution $\sigma$ of 
	\begin{equation}\label{4-3-ter}
\left\{ \begin{array}{ll}
\displaystyle
\frac{\partial \sigma}{\partial t}+(({v}_s+e^{-\beta t}{y}) \cdot \nabla)\sigma-\beta\sigma=0\quad&\mbox{in}\quad Q_{\infty},
\vspace{1.mm}\\
\displaystyle
\sigma (x,t)=0 \quad &\mbox{on} \quad \Sigma_{in,y,\infty},
\vspace{1.mm}\\
\displaystyle
\sigma (x,0)=\sigma_0\quad&\mbox{in}\quad\Omega,
\end{array}\right.
\end{equation}
where 
$$
	\Sigma_{in,y,\infty} =  \{(x,t) \in  \Gamma  \times (0,T) \suchthat (v_s(x) + y(x,t) e^{-\beta t}) \cdot n(x) < 0 \},
$$	
	satisfies \eqref{Sigma-t-1=0}.
	
	We will achieve this goal using two steps. In the first one, we shall consider smooth ($\in V^{2,1} (Q_{T_1}) \cap L^2(0,T_1; W^{1,\infty}(\Omega))$) vector field $y$. In the second one, we will explain how the same result can be obtained for all vector fields $y \in V^{2,1}(Q_{T_1})$.
	\smallskip

	{\it Case $y \in V^{2,1} (Q_{T_1}) \cap L^2(0,T_1; W^{1,\infty}(\Omega))$.} Here we assume that 
\begin{equation}
	\label{AdditionalAssumption}
	y \in V^{2,1} (Q_{T_1}) \cap L^2(0,T_1; W^{1,\infty}(\Omega)).
\end{equation}
	With $\varepsilon >0$ given by \eqref{varepsilon}, we then define a function $\vartheta \in C^{\infty}(\R^2)$ and $\vartheta(x_{1},x_{2})\in [0,1]$ such that
	\begin{equation}\label{vartheta}
	\vartheta(x_{1},x_{2})=\left\{ \begin{array}{ll}
	0\quad&\mbox{if}\quad (x_{1},x_2) \in [0,d] \times [A_1,1-A_1],\\
	1\quad&\mbox{if}\quad (x_{1},x_2) \in \R^2 \setminus [-\frac{\varepsilon}{2}, d + \frac{\varepsilon}{2}] \times [A_1 -\frac{\varepsilon}{2} ,1-A_1 + \frac{\varepsilon}{2}].
	\end{array}\right.
	\end{equation}
	We consider the following auxiliary transport problem
\begin{equation}\label{auxtrans}
\left\{ \begin{array}{ll}
\displaystyle
\frac{\partial\Psi}{\partial t}+((v_{s}+e^{-\beta t}\mathbb{E}y)\cdot\nabla)\Psi=0\quad&\mbox{in}\quad \R^2\times (0,T_{1}),\\
\Psi(\cdot,0)=\vartheta\quad&\mbox{in}\quad \R^2.
\end{array}\right.
\end{equation}	
Since $v_{s}+e^{-\beta t}\mathbb{E}y$ belongs to $L^2(0,T_1; W^{1,\infty}(\R^2))$ the system \eqref{auxtrans} can be solved using the characteristics formula to obtain
\begin{equation}\label{reprP}
\begin{array}{l}
	\Psi(x,t)=\vartheta(X(x,0,t))\quad\mbox{for all}\quad (x,t)\in\R^2\times[0,T_{1}],
\end{array}
\end{equation}
where the flow $X(\cdot,\cdot,\cdot)$, defined by \eqref{4-5}, is globally Lipschitz in $\R^2\times[0,T_{1}]$. It follows that $\Psi$ is also globally Lipschitz in $\R^2\times[0,T_{1}]$. Besides, this formula immediately provides the non-negativity of $\Psi$ in $\R^2\times[0,T_{1}]$.

We now introduce the following quantity:
\begin{equation}\label{lyap}
\begin{array}{l}
\displaystyle
E_{loc}(t)=\frac{1}{2}\int_{\Omega} \Psi(x,t) |\sigma(x,t)|^{2} dx \quad\mbox{for all}\quad t\in [0,T_{1}].
\end{array}
\end{equation}	
The idea is that this quantity will measure the $L^2$ norm of $\sigma(\cdot,t)$ localized in the support of $\Psi(\cdot,t)$.
\\
In order to evaluate how the quantity $E_{loc}$ evolves, we use the renormalization property \eqref{Renormalized} with $\xi(s) = s^2$ and we compute the time derivative of $E_{loc}$ (in $\mathcal{D}'(0,T)$):
\begin{equation}\label{timeder}
\begin{array}{ll}
\displaystyle
\displaystyle \frac{d}{dt}E_{loc}(t)
& \displaystyle =\frac{1}{2}\int\limits_{\Omega} (\frac{\partial\Psi }{\partial t}+({v}_s+e^{-\beta t}\mathbb{E}{y})\cdot\nabla\Psi)|\sigma|^{2}dx+\beta\int\limits_{\Omega}\Psi|\sigma|^{2} dx
\\
& \displaystyle \quad  
-\frac{1}{2}\int\limits_{\Gamma}\Psi |\gamma_{\sigma}|^{2} (({v}_s+e^{-\beta t}\mathbb{E}{y})\cdot n)dm
\\
& \displaystyle\leqslant\beta\int\limits_{\Omega}\Psi |\sigma|^{2}dx = 2 \beta E_{loc}(t).
\end{array}
\end{equation}
In the above calculation we have used that $\Psi$ solves the equation \eqref{auxtrans}$_{1},$ $\gamma_{\sigma}$ (the trace of $\sigma,$ see Theorem \ref{trace}, item $(ii)$) vanishes on $\Sigma_{in,y,T_1}$, and that $\Psi$ stays non-negative in $(0,T_{1})\times \R^2$. Now using Gr\"{o}nwall's inequality in \eqref{timeder}, we get
	\begin{equation}\label{pregron}
	\begin{array}{l}
	\displaystyle \frac{1}{2}\int\limits_{\Omega}\Psi(x,T_1) |\sigma(x,T_{1})|^{2}dx=E_{loc}(T_{1})\leqslant e^{2\beta T_{1}}E_{loc}(0) = 0,
	\end{array}
	\end{equation}
	where the last identity comes from the fact that ${\sigma}_{0}\in L^{\infty}(\Omega)$ satisfies the condition \eqref{1-4} and the choice of $\Psi$ in \eqref{vartheta}, \eqref{auxtrans}.\\ 
	We now prove that 
	\begin{equation}
		\label{Psi-T1}
		\forall x \in \Omega, \quad \Psi(x, T_1) = 1.
	\end{equation}
	In order to prove \eqref{Psi-T1}, we will rely on the formula \eqref{reprP}, and Lemma \ref{ld3}.
	Indeed, for $x = (x_1,x_2) \in \Omega$, we have 
	$$
		X_0(x,0,T_1) = 
		\begin{pmatrix}
			x_{1}-T_{1}(x_{2}(1-x_{2}))\\
			x_{2}
		\end{pmatrix}.
	$$
	Therefore, if $x = (x_1,x_2) \in \Omega$ satisfies $x_2 \in (A_1 - \varepsilon, 1- A_1+\varepsilon)$, as one has $x_{2}(1-x_{2})>(A_{1}- \varepsilon)(1-A_{1} + \varepsilon)$, $(X_0(x, 0,T_1))_1 \leq d - T_1 (A_1-\varepsilon)(1-A_1+\varepsilon) \leq - \varepsilon$. Similarly, if $x_2 \in [0,1] \setminus (A_1-\varepsilon, 1-A_1+\varepsilon)$, $(X_0(x,0,T_1))_2 \in [0,1] \setminus (A_1-\varepsilon, 1-A_1+\varepsilon)$.
In particular, one obtains that for all $x=(x_{1},x_{2})\in \Omega$
\begin{equation}\label{X0es}
\begin{array}{l}
	X_{0}(x,0,T_{1})\in \R^2 \setminus (-\varepsilon, d+\varepsilon) \times (A_{1}-\varepsilon,1-A_{1} + \varepsilon).
\end{array}
\end{equation}
Now set $\displaystyle K_{3}=K_{3}(T_{1})=\frac{\varepsilon}{2K_{4}(T_{1})}>0,$ where $\displaystyle K_{4}(T_{1})$ is the constant appearing in Lemma \ref{ld3}, and  assume that
\begin{equation}\label{smally}
\begin{array}{l}
\displaystyle
\|y\|_{V^{2,1}(Q_{T_1})}<K_{3}.
\end{array}
\end{equation}
The inequality \eqref{4-9}, \eqref{X0es} and the assumption \eqref{smally} furnish that for all $x\in \Omega$,
\begin{equation}\label{charnear}
\begin{array}{l}
\displaystyle
	X(x,0,T_{1})\in \R^2 \setminus [-\frac{\varepsilon}{2},d+\frac{\varepsilon}{2}]\times[A_1 -\frac{\varepsilon}{2} ,1-A_1 + \frac{\varepsilon}{2}].
\end{array}
\end{equation}
Now using the representation \eqref{reprP} of $\Psi$, we immediately deduce \eqref{Psi-T1}. The estimate \eqref{pregron} then yields that $\sigma$ vanishes at time $T_1$ in the whole set $\Omega$, i.e. the identity \eqref{Sigma-t-1=0}.\\

{\it The general case $y \in V^{2,1}(Q_{T_1})$.} We now discuss the case in which $y$ does not satisfy the regularity \eqref{AdditionalAssumption} and $y$ only belongs to $V^{2,1}(Q_\infty)$ as stated in Theorem \ref{t4-1}. In order to deal with this case, we use the density of $V^{2,1}(Q_{T_1})\cap L^2(0,T_1; W^{1, \infty}(\Omega))$ in $V^{2,1}(Q_{T_1}) $. In particular, if $y$ belongs to $V^{2,1}(Q_\infty)$ and satisfies \eqref{smally}, we can find a sequence $y_n$ of functions of $V^{2,1}(Q_{T_1}) \cap L^2(0,T_1; W^{1, \infty}(\Omega))$ such that $y_n$ strongly converges to $y$ in $V^{2,1}(Q_T)$ and for all $n$, $\| y_n \|_{V^{2,1}(Q_{T_1})} < K_3$. Using then the previous arguments, we can show that for all $n$, $\sigma_n(x,T_1) = 0$ for all $x \in \Omega$, where $\sigma_n$ denotes the solution of \eqref{4-12} on the time interval $(0,T_1)$. The strong convergence of $(y_n)$ to $y$ in $V^{2,1}(Q_{T_1})$, hence of $y_n$ to $y$ in $L^1(Q_{T_1})$ and of $y_n\cdot n$ to $y\cdot n$ in $L^1(\Sigma_{T_1})$, and Lemma \ref{l4-4} then imply \eqref{Sigma-t-1=0}.\\

{\it End of the proof of Theorem \ref{t4-1}}. We shall then show that, when $y \in V^{2,1}(Q_\infty)$ satisfies the condition \eqref{smally}, the solution $\sigma$ of \eqref{4-3} stays zero for times larger than $T_1$. This is obvious, as one can replace \eqref{4-3}$_{3}$ by $\sigma(x,T_{1})=0$ on $\Omega$ and solve the Cauchy problem \eqref{4-3} in the time interval $[T_{1},\infty)$ to obtain that $\sigma$ is the trivial solution
 $$\sigma(x,s)=0\quad\mbox{for all}\quad (x,s)\in\Omega\times[T_{1},\infty).$$
This concludes the proof of Theorem \ref{t4-1}.
	\end{proof}
	
	\begin{remark}
		In the above proof, we have handled separately the case $y \in V^{2,1}(Q_{T_1}) \cap L^2(0,T_1; W^{1, \infty}(\Omega))$ from the case of a general vector field $y \in V^{2,1}(Q_{T_1})$, because the solution $\Psi$ of \eqref{auxtrans} for a vector field $y \in V^{2,1}(Q_{T_1})$ has \emph{a priori} only H\"older regularity (see in particular \cite[Theorem 3.7]{Bahouri-Chemin-Danchin}), and thus cannot be used directly as a test function in the weak formulation \eqref{Renormalized} to obtain \eqref{timeder}.
	\end{remark}
	
	\begin{remark}
	   In general to prove the stabilizability of a non-linear problem it is usual to first study the stabilizability of the corresponding linear problem and then consider the non-linear term as a source term to obtain analogous stabilizability result corresponding to the complete non-linear system. But the reader may notice that contrary to the usual method we did not consider the non-linear term $({y}\cdot\nabla)\sigma$ (nonlinear in $(\sigma,y)$ but linear in $\sigma$) as a source term while dealing with the system \eqref{4-3}. This is because the transport equation has no regularizing effect on its solution, hence it is not possible to consider the non-linear term in \eqref{4-3} as a source term and to recover the solution in $L^{\infty}(Q_{\infty}).$ 
	\end{remark}

	\section{Stabilization of the two dimensional Navier-Stokes equations.}\label{final}
	\begin{proof}[Proof of Theorem \ref{main}]
	We will prove Theorem \ref{main} using the Schauder fixed point theorem. We now discuss the strategy of the proof.\\[1.mm]
	(i) First we define an appropriate fixed point map. This will be done in Section \ref{dfp}.\\
	(ii) Then we fix a suitable ball which is stable by the map defined in step (i). This is done in the Section \ref{sbit}.\\
	(iii) In Section \ref{coco} we show that the ball defined in step (ii), is compact in some appropriate topology. We then prove that the fixed point map from step (i) in that topology is continuous.\\
	(iv) At the end we draw the final conclusion to prove Theorem \ref{main}.  
	\subsection{Definition of a fixed point map}\label{dfp}
	Let us recall the fully non linear system (including the boundary controls) under consideration:
	\begin{equation}\label{5-2*}
		\left\{\begin{array}{ll}
			\displaystyle
			\frac{\partial \sigma}{\partial t}+(({v}_s+e^{-\beta t}{y}) \cdot \nabla)\sigma-\beta\sigma=0\quad& \mbox{in}\quad Q_{\infty},
			\vspace{1.mm}\\
			\sigma(x,t)=0 \quad &\mbox{on} \quad \Gamma_{in} \times(0,\infty),
			\vspace{1.mm}\\
			\sigma (x,0)=\sigma_0\quad&\mbox{in}\quad\Omega,
			\vspace{2.mm}\\
			\displaystyle
			\frac{\partial {y}}{\partial t}-\beta {y}-\nu \Delta {y}+ ({v}_s \cdot \nabla){y}+({y} \cdot \nabla)v_s +\nabla  q=\mathcal{F}({y},\sigma)\quad& \mbox{in}\quad Q_{\infty},
			\vspace{1.mm}\\
			\mbox{div}({y})=0\quad& \mbox{in}\quad Q_{\infty},
			\vspace{1.mm}\\
			{y}=0\quad& \mbox {on} \quad (\Gamma_0\cup \Gamma_{out} )\times (0,\infty),
			\vspace{1.mm}\\
			{y}=\sum\limits_{j=1}^{N_{c}}{w}_{j}(t){{g}_{j}}(x) \quad &\mbox{on} \quad \Gamma_{in} \times (0,\infty),
			\vspace{1.mm}\\
			{y}(x,0)={y}_0\quad&\mbox{in}\quad\Omega,
			\vspace{1.mm}\\
			{w}_{c}^{'}+{\gamma }{w}_{c}-\mathcal{K}(Py,w_{c})=0\quad& \mbox{in}\quad (0,\infty),
			\vspace{1.mm}\\
			{w}_{c}(0)=0\quad&\mbox{in}\quad\Omega,
		\end{array}\right.
	\end{equation}
	where 
	$$\mathcal{F}({y},\sigma)=-e^{-\beta t}{\sigma}\frac{\partial {y}}{\partial t}-e^{-\beta t}({y}\cdot \nabla){y}-e^{-\beta t}{\sigma}(v_{s}\cdot \nabla){y} -e^{-\beta t}{\sigma}({y}\cdot \nabla)v_{s}-e^{-2\beta t}{\sigma}({y}\cdot \nabla){y}+\beta e^{-\beta t} \sigma {y},$$
	and $w_{c}=(w_{1},...,w_{N_{c}}).$ To prove the existence of a solution of the system \eqref{5-2*} we are going to define a suitable fixed point map.\\
	Now assume that $\sigma_{0}\in L^{\infty}(\Omega)$ and satisfies \eqref{1-4}. Recall the definition of ${g_{j}}$'s from \eqref{basisU0}. Let us suppose that $\widehat{y}\in {V}^{2,1}(Q_{\infty})$ satisfies \eqref{yK3} and on the boundary it is given in the following form
	 \begin{equation}\label{hatybd}
	 \widehat{y}\mid_{\Sigma_{\infty}}=\sum\limits_{j=1}^{N_{c}}\widehat{w}_{j}(t){g_{j}}(x),
	 \end{equation} 
	 where $\widehat{w}_{c}=(\widehat{w}_{1},...,\widehat{w}_{N_{c}})\in H^{1}(0,\infty;\mathbb{R}^{N_{c}}).$ In addition the coefficients $\widehat{w}_{c}$ are assumed to be such that $\widehat{y}$ satisfies the following boundary condition
	 \begin{equation}\label{concor**}
	 \|\widehat{y}\mid_{\Sigma_{\infty}}\|_{L^{\infty}(\Sigma_{\infty})}\leqslant \frac{L(1-L)}{2},
	 \end{equation}
	 where the constant $L$ was fixed in \eqref{dgc}. We further assume that $y_{0}\in V^{1}_{0}(\Omega).$\\  
	 	We consider the following set of equations
		\begin{equation}\label{5-2}
		\left\{\begin{array}{ll}
		\displaystyle
		\frac{\partial {\widehat{\sigma}}}{\partial t}+(({v}_s+e^{-\beta t}\widehat{y}) \cdot \nabla){\widehat{\sigma}}-\beta{\widehat{\sigma}}=0\quad &\mbox{in}\quad Q_{\infty},
		\vspace{1.mm}\\
		{\widehat{\sigma}} (x,t)=0 \quad &\mbox{on} \quad\Gamma_{in} \times(0,\infty),
		\vspace{1.mm}\\
		{\widehat{\sigma}} (x,0)=\sigma_0\quad&\mbox{in}\quad\Omega,
		\vspace{2.mm}\\
		\displaystyle
		\frac{\partial y}{\partial t}-\beta y-\nu \Delta y+ ({v}_s \cdot \nabla)y+(y \cdot \nabla)v_s +\nabla  q=\mathcal{F}(\widehat{y},{\widehat{\sigma}})\quad &\mbox{in}\quad Q_{\infty},
		\vspace{1.mm}\\
		\mbox{div}(y)=0\quad& \mbox{in}\quad Q_{\infty},
		\vspace{1.mm}\\
		y=0\quad &\mbox {on} \quad (\Gamma_0\cup\Gamma_{out}) \times (0,\infty),
		\vspace{1.mm}\\
		y=\sum\limits_{j=1}^{N_{c}}{w}_{j}(t){{g}_{j}}(x) \quad &\mbox{on} \quad \Gamma_{in} \times (0,\infty),
		\vspace{1.mm}\\
		y(x,0)={y}_0\quad&\mbox{in}\quad\Omega,
		\vspace{1.mm}\\
		{w}_{c}^{'}+{\gamma }{w}_{c}-\mathcal{K}(Py,w_{c})=0\quad& \mbox{in}\quad (0,\infty),
		\vspace{1.mm}\\
		{w}_{c}(0)=0\quad&\mbox{in}\quad\Omega,
		\end{array}\right.
		\end{equation}
		where 	$$\mathcal{F}(\widehat{y},{\widehat{\sigma}})=-e^{-\beta t}{\widehat{\sigma}}\frac{\partial \widehat{y}}{\partial t}-e^{-\beta t}(\widehat{y}\cdot \nabla)\widehat{y}-e^{-\beta t}{\widehat{\sigma}}(v_{s}\cdot \nabla)\widehat{y} -e^{-\beta t}{\widehat{\sigma}}(\widehat{y}\cdot \nabla)v_{s}-e^{-2\beta t}{\widehat{\sigma}}(\widehat{y}\cdot \nabla)\widehat{y}+\beta e^{-\beta t} \widehat{\sigma} \widehat{y}$$
		and $w_{c}=(w_{1},...,w_{N_{c}}).$ Since \eqref{hatybd} and \eqref{concor**} hold, one can verify that $\widehat{y}$ satisfies \eqref{cony.n}. Hence we can solve \eqref{5-2}$_{1}$-\eqref{5-2}$_{3}$ for $\widehat{\sigma}$ in $L^{\infty}(Q_{\infty})$ (see Section \ref{density}). Now using this $\widehat{\sigma}$ and $\widehat{y}$ one can solve \eqref{5-2}$_{4}$-\eqref{5-2}$_{10}$ (see Section \ref{velocity}) for $(y,w_{c})$ provided  $\mathcal{F}(\widehat{y},\widehat{\sigma})\in L^{2}(Q_{\infty}).$ This is indeed the case since we have $\widehat{y}\in V^{2,1}(Q_{\infty})$ and $\widehat{\sigma}\in L^{\infty}(Q_{\infty})$ and the detailed estimates are done in Lemma \ref{l5-2}.\\
    At this point we fix $T_{1}>T_{A_{1}}=\frac{d}{A_{1}(1-A_{1})}$ in Theorem \ref{t4-1}. We also fix the constant $K_{3}$ appearing in Theorem \ref{t4-1}. Let $0<\mu<K_{3}.$ We define a convex set $D_{\mu}$ as follows
   \begin{equation}\label{Dmu}
   D_{\mu}=\left\{ \begin{split}
   & \begin{pmatrix}
   \widehat{y}\\\widehat{w}_{c}
   \end{pmatrix}\in {V}^{2,1}(Q_{\infty})\times H^{1}(0,\infty;\mathbb{R}^{N_{c}})\suchthat\|(\widehat{y},\widehat{w}_{c})\|_{{V}^{2,1}(Q_{\infty})\times H^{1}(0,\infty;\mathbb{R}^{N_{c}})}\leqslant\mu\\
   &\,\,\,\,\mbox{and}\,
   \widehat{y}\mid_{\Sigma_{\infty}} \mbox{is of the form}\,\,\eqref{hatybd}\,\, \mbox{and satisfies the condition}\,\,\eqref{concor**}
   \end{split}\right\}.
   \end{equation} 
   Notice that $(0,0)$ belongs to $D_{\mu},$ hence $D_{\mu}$ is non-empty.\\
    Let $(\widehat{\sigma},y,w_{c})\in L^{\infty}(Q_{\infty})\times{V}^{2,1}(Q_{\infty})\times H^{1}(0,\infty;\mathbb{R}^{N_{c}})$ be the solution of system \eqref{5-2} corresponding to $(\widehat{y},\widehat{w}_{c})\in D_{\mu}.$ We consider the following map 
    \begin{equation}\label{chi}
    \begin{matrix}
    \displaystyle
	\chi: & D_{\mu} & \longrightarrow & {V}^{2,1}(Q_{\infty})\times H^{1}(0,\infty;\mathbb{R}^{N_{c}})\\
	\displaystyle &(\widehat{y},\widehat{w}_{c})
	&\mapsto & (
	\displaystyle {{y},w_{c}}
).
	\end{matrix}
	\end{equation}
	 In the sequel we will choose the constant $\mu\in(0,K_{3}),$ small enough such that $\chi$ maps $D_{\mu}$ into itself.\\
	 We will then look for a fixed point of the map $\chi.$ Indeed if $(
	 y_{f},w_{{f},c}
	 )$ is a fixed point of the map $\chi,$ by construction, there exists a function $\sigma_{f}$ such that the triplet $(\sigma_{{f}},{y}_{f},w_{{f},c})$ solves \eqref{5-2*}. Hence in order to prove Theorem \ref{main} it is enough to show that the map $\chi$ has a fixed point in $D_{\mu}$.
	\subsection{$\chi$ maps $D_{\mu}$ into itself}\label{sbit}
	In this section we will choose a suitable constant $\mu$ such that $\chi$ maps $D_{\mu}$ into itself, provided the initial data are small enough.\\ 
	 Now given $(
	 \widehat{y},\widehat{w}_{c}
	 )\in D_{\mu},$ we can use \eqref{fiessig} in order to show that $\widehat{\sigma},$ the solution of \eqref{5-2}$_{1}$-\eqref{5-2}$_{3}$ satisfies the following
	\begin{equation}\label{bdtran}
		\begin{array}{l}
			\|\widehat{\sigma}\|_{L^{\infty}(Q_{\infty})}\leqslant e^{\beta T_{1}}\|\sigma_{0}\|_{L^{\infty}(\Omega)}.
		\end{array}
	\end{equation}
	\begin{lem}\label{l5-2}
		 If $(
		\widehat{y},\widehat{w}_{c}
		)$ belongs to $D_{\mu}$ (defined in \eqref{Dmu}) and $\widehat{\sigma}$ is the solution of the problem \eqref{5-2}$_{1}$-\eqref{5-2}$_{3}$ then $\mathcal{F}(\widehat{y},\widehat{\sigma})\in L^{2}(Q_{\infty}).$ Besides there exist constants $K_{5}>0,$ $K_{6}>0$ such that for all $(\widehat{y},\widehat{w}_{c})\in D_{\mu}$ and for all $(\sigma_{0},y_{0})$ with $\sigma_{0}$ satisfying \eqref{1-4} and $e^{\beta T_{1}}\|\sigma_{0}\|_{L^{\infty}(\Omega)}<1,$ the following estimate is true:
		\begin{equation}\label{5-9}
			\begin{array}{l}
				\| \mathcal{F}(\widehat{y},\widehat{\sigma})\|_{L^{2}(Q_{\infty})}\leqslant K_{5}e^{\beta T_{1}}\|\sigma_{0}\|_{L^{\infty}(\Omega)}
				+K_{6}\|\widehat{y}\|_{{V}^{2,1}(Q_{\infty})}^{2}.
			\end{array}
		\end{equation}
	\end{lem}
	\begin{proof}
		First use \eqref{bdtran} to show
		\begin{equation}\label{5-3}
		\begin{split}
		\|{\widehat{\sigma}}\frac{\partial\widehat{y}}{\partial t}\|_{L^{2}(Q_{\infty})} &\leqslant e^{\beta T_{1}}\|\sigma_{0}\|_{L^{\infty}(\Omega)}\|\widehat{y}\|_{{V}^{2,1}(Q_{\infty})}.  
		\end{split}
		\end{equation}
		Recall that $v_{s}\in C^{\infty}(\bar{\Omega}).$ Hence we again apply \eqref{bdtran} to get
		\begin{equation}\label{5-4}
		\|{\widehat{\sigma}}(v_{s}\cdot\nabla)\widehat{y}\|_{L^{2}(Q_{\infty})}\leqslant e^{\beta T_{1}}\|\sigma_{0}\|_{L^{\infty}(\Omega)}\|v_{s}\|_{W^{1,\infty}(\Omega)}\|\widehat{y}\|_{{V}^{2,1}(Q_{\infty})}.
		\end{equation}
		and
		\begin{equation}\label{5-5}
		\|{\widehat{\sigma}}(\widehat{y}\cdot\nabla)v_{s}\|_{L^{2}(Q_{\infty})}\leqslant e^{\beta T_{1}}\|\sigma_{0}\|_{L^{\infty}(\Omega)}\|v_{s}\|_{W^{1,\infty}(\Omega)}\|\widehat{y}\|_{{V}^{2,1}(Q_{\infty})}.
		\end{equation}
		Now we estimate $(\widehat{y}\cdot\nabla)\widehat{y}$ in $L^{2}(Q_{\infty}).$ We know that $V^{2,1}(Q_{\infty})$ is continuously embedded in the space $L^{\infty}(0,\infty;H^{1}(\Omega)).$ Hence $\widehat{y}\in L^{\infty}(0,\infty;H^{1}(\Omega)),$ $\nabla \widehat{y}\in L^{2}(0,\infty;H^{1}(\Omega))$ and the following holds 
		\begin{equation}\label{5-6}
		\begin{split}
		\|(\widehat{y}\cdot\nabla)\widehat{y}\|_{L^{2}(Q_{\infty})} & \leqslant K\|\widehat{y}\|_{L^{\infty}(0,\infty;H^{1}(\Omega))}\|\nabla\widehat{y}\|_{L^{2}(0,\infty;H^{1}(\Omega))}\leqslant K \|\widehat{y}\|^{2}_{{V}^{2,1}(Q_{\infty})}.
		\end{split}
		\end{equation}
		Similarly
		\begin{equation}\label{5-7}
		\begin{array}{l}
		\|{\widehat{\sigma}}(\widehat{y}\cdot\nabla)\widehat{y}\|_{L^{2}(Q_{\infty})}\leqslant Ke^{\beta T_{1}}\|\sigma_{0}\|_{L^{\infty}(\Omega)}\|\widehat{y}\|^{2}_{{V}^{2,1}(Q_{\infty})}
		\end{array}
		\end{equation}
		and
		\begin{equation}\label{5-8}
		\begin{array}{l}
		\|\beta{\widehat{\sigma}}\widehat{y}\|_{L^{2}(Q_{\infty})}\leqslant |\beta|e^{\beta T_{1}}\|\sigma_{0}\|_{L^{\infty}(\Omega)}\|\widehat{y}\|_{{V}^{2,1}(Q_{\infty})}.
		\end{array}
		\end{equation}
		Now observe that   
		\begin{equation}\label{esmu0mu}
		\begin{array}{l}
		e^{\beta T_{1}}\|\sigma_{0}\|_{L^{\infty}(\Omega)}\|\widehat{y}\|_{V^{2,1}(Q_{\infty})}
		\leqslant \frac{1}{2}(	e^{\beta T_{1}}\|\sigma_{0}\|_{L^{\infty}(\Omega)}+\|\widehat{y}\|_{V^{2,1}(Q_{\infty})}^{2}).
		\end{array}
		\end{equation}
		Hence we use estimates \eqref{5-4}-\eqref{5-8} and \eqref{esmu0mu} to prove Lemma \ref{l5-2} and the estimate \eqref{5-9}.
	\end{proof}
	\begin{lem}\label{eohy}
		There exist constants $K_{7}>\max\{1,K_{5},K_{6}\}>0,$ $K_{8}>\max\{K_{5},K_{6}\}>0$ such that for all $(\widehat{y},{\widehat{w}_{c}})\in D_{\mu}$ (defined in \eqref{Dmu}), for all $(\sigma_{0},y_{0})$ with $\sigma_{0}$ satisfying \eqref{1-4}, $e^{\beta T_{1}}\|\sigma_{0}\|_{L^{\infty}(\Omega)}<1,$ and for $\widehat{\sigma}$ uniquely solving \eqref{5-2}$_{1}$-\eqref{5-2}$_{3},$  
		$({y},w_{c}
		)=\chi(\widehat{y},\widehat{w}_{c}),$ solving \eqref{5-2}$_{4}$-\eqref{5-2}$_{10}$ is well defined and satisfies the following inequality 
		\begin{equation}\label{5-15}
		\begin{array}{l}
		\|({y},w_{c})\|_{{V}^{2,1}(Q_{\infty})\times H^{1}(0,\infty;\mathbb{R}^{N_{c}})}\leqslant
		K_{7}{\max}\,\{e^{\beta T_{1}}\|\sigma_{0}\|_{L^{\infty}(\Omega)},\|{y}_{0}\|_{{V}^{1}_{0}(\Omega)}\}+K_{8}\|\widehat{y}\|_{{V}^{2,1}(Q_{\infty})}^{2}.
		\end{array}
		\end{equation}	
		\end{lem}
	\begin{proof}	
	Corollary \ref{c3-8} shows that $(
		y,w_{c}
		)$ satisfy the following estimate
		\begin{equation}\label{3-24re}
		\begin{array}{l}
		\|( y,w_{c} )\|_{V^{2,1}(Q_{\infty})\times H^{1}(0,\infty;\mathbb{R}^{N_{c}})}\leqslant K_{1}(\| {y}_{0} \|_{{V}^{1}_{0}(\Omega)} +\| \mathcal{F}(\widehat{y},{\widehat{\sigma}})\|_{L^{2}(Q_{\infty})}).
		\end{array}
		\end{equation}
		Now using \eqref{5-9} in \eqref{3-24re}, we get the desired result.
	\end{proof}	
	From now on we will consider the initial data $\sigma_{0}\in L^{\infty}(\Omega)$ and $y_{0}\in V^{1}_{0}(\Omega)$ such that they satisfy
	\begin{equation}\label{smu0}
	\left\{ \begin{array}{l}
	\displaystyle
	\sigma_{0}\,\,\mbox{satisfies}\,\,\eqref{1-4},\\
	\displaystyle
	\mbox{max}\,\{e^{\beta T_{1}}\|\sigma_{0}\|_{L^{\infty}(\Omega)},\|{y}_{0}\|_{{V}^{1}_{0}(\Omega)}\}<\min\,\left\{\frac{L(1-L)}{8K_{2}K_{7}},\frac{K_{3}}{2K_{7}},\frac{1}{4K_{7}K_{8}},1\right\},
	\end{array}\right.
	\end{equation}
	where $K_{2},$ $K_{7}$ and $K_{8}$ are the constants appearing respectively in \eqref{dofcon} and \eqref{5-15}.
	\begin{lem}\label{l5-3}
	For all $(\sigma_{0},y_{0})$ satisfying \eqref{smu0}, setting
	\begin{equation}\label{setmu}
		\begin{split}
		\displaystyle
		\mu={2K_{7}\mbox{max}\,\{e^{\beta T_{1}}\|\sigma_{0}\|_{L^{\infty}(\Omega)},\|{y}_{0}\|_{{V}^{1}_{0}(\Omega)}\}},
		\end{split}
	\end{equation}
 	where $K_7$ is the constant in \eqref{5-15}, the map $\chi$ (defined in \eqref{chi}) maps $D_{\mu}$ (defined in \eqref{Dmu}) into itself.
	\end{lem}
	\begin{proof}
		In view of \eqref{smu0}$_{2}$ and \eqref{setmu}, one observes in particular that
		\begin{equation}\label{mu0}
		\begin{array}{l}
		\displaystyle
		0<\mu<\min\,\left\{\frac{L(1-L)}{4K_{2}},{K_{3}},\frac{1}{2K_{8}}\right\}.
		\end{array}
		\end{equation}
		Now we will verify that with the choice \eqref{setmu} of $\mu,$ the map $\chi$ maps $D_{\mu}$ into itself. Let $(\widehat{y},{\widehat{w}_{c}})\in D_{\mu},$ for all $(\sigma_{0},y_{0})$ obeying \eqref{smu0} and for $\widehat{\sigma}$ uniquely solving \eqref{5-2}$_{1}$-\eqref{5-2}$_{3},$  
		$({y},w_{c}
		)=\chi(\widehat{y},\widehat{w}_{c}),$ solves \eqref{5-2}$_{4}$-\eqref{5-2}$_{10}.$ We claim that $(y,w_{c})\in D_{\mu}.$\\
		 First of all in view of \eqref{5-15}, \eqref{smu0}$_{2},$ \eqref{setmu} and \eqref{mu0} we observe that 
		 \begin{equation}\nonumber
		 \begin{split}
		 \displaystyle
			\|({y},w_{c})\|_{{V}^{2,1}(Q_{\infty})\times H^{1}(0,\infty;\mathbb{R}^{N_{c}})}\leqslant K_{7}\mbox{max}\,\{e^{\beta T_{1}}\|\sigma_{0}\|_{L^{\infty}(\Omega)},\|{y}_{0}\|_{{V}^{1}_{0}(\Omega)}\}+K_{8}\mu^{2}\leqslant \mu.
		 \end{split}
		 \end{equation}
	Since $({y},w_{c}
	)=\chi(\widehat{y},\widehat{w}_{c})$ solves \eqref{5-2}$_{4}$-\eqref{5-2}$_{10},$ the function $y$ on the boundary is given by 	$\sum\limits_{j=1}^{N_{c}}{w}_{j}(t){{g}_{j}}(x)$. This verifies \eqref{hatybd}.\\	
	Finally 
	\begin{equation}\label{L1L}
	\begin{split}
	\|{y}_{0} \|_{{V}^{1}_{0}(\Omega)}+\| \mathcal{F}(\widehat{y},\widehat{\sigma})\|_{L^{2}(Q_{\infty})}&\leqslant (1+K_{5})\mbox{max}\,\{e^{\beta T_{1}}\|\sigma_{0}\|_{L^{\infty}(\Omega)},\|{y}_{0}\|_{{V}^{1}_{0}(\Omega)}\}
	+K_{6}\mu^{2}\\
	& \leqslant\frac{3}{8}\frac{L(1-L)}{K_{2}}\leqslant\frac{L(1-L)}{2K_{2}},
	\end{split}
	\end{equation}	
	where in \eqref{L1L}$_{1}$ we have used \eqref{smu0}$_{2},$ \eqref{5-9} and in \eqref{L1L}$_{2}$ we have used \eqref{smu0}$_{2}$, \eqref{mu0} and the fact that $K_{7}>\max\{1,K_{5},K_{6}\}>0,$ $K_{8}>\max\{K_{5},K_{6}\}>0$ (which follows from the statement of Lemma \ref{eohy}). Now using Corollary \ref{p3.0.2} one verifies \eqref{concor**} for $y\mid_{\Sigma_{\infty}}.$\\ 
	Hence we have verified that $(y,w_{c})\in D_{\mu}$ and the proof of Lemma \ref{l5-3} is finished.
	\end{proof}	
	At this point we fix $\mu$ as in Lemma \ref{l5-3}. 
	\subsection{Compactness and continuity}\label{coco}	
	To start with, let us define the weighted space
	\begin{equation}\nonumber
	\begin{array}{l}
	\displaystyle
	L^{2}(0,\infty,(1+t)^{-1}dt;{ L^{2}}(\Omega)\times\mathbb{R}^{N_{c}})\\
	\displaystyle=\left\{\overline{z}=\begin{pmatrix}
	z(x,t)\\w_{c}(t)
	\end{pmatrix}\in L^{2}(Q_{\infty})\times L^{2}((0,\infty);\mathbb{R}^{N_{c}})\suchthat \int\limits_{0}^{\infty}{(1+t)^{-2}}{\|\overline{z}\|_{{L}^{2}(\Omega)\times\mathbb{R}^{N_{c}}}^{2}}dt<\infty\right\}.
	\end{array}
	\end{equation}
	We endow the set $D_{\mu},$ defined in \eqref{Dmu}, with the norm induced from $L^{2}(0,\infty,(1+t)^{-1}{dt};{L^{2}}(\Omega)\times\mathbb{R}^{N_{c}}).$ 
	\begin{lem}\label{l5-4}
		The set $D_{\mu}$ is compact in $L^{2}(0,\infty,{(1+t)^{-1}}dt;{L}^{2}(\Omega)\times\mathbb{R}^{N_{c}}).$
	\end{lem}
		\begin{proof}
			We divide the proof in two steps.\\	
			Step 1.\,\,We claim that $D_{\mu}$ is closed in the space $L^{2}(0,\infty,{(1+t)^{-1}}dt;{L}^{2}(\Omega)\times\mathbb{R}^{N_{c}}).$ Consider a sequence $\{\overline{y}_{n}\}_{n}$ $\left(\mbox{where}\,\,\overline{y}_{n}=(
			y_{n},{w_{n,c}}
			) \right)$ in $D_{\mu}$ such that $\{\overline{y}_{n}\}_{n}$ converges to some $\overline{y}$ $\left(\mbox{where}\,\,\overline{y}=(
			y,{w_{c}}
			) \right)$ in the space
			$L^{2}(0,\infty,{(1+t)^{-1}}dt;{L}^{2}(\Omega)\times\mathbb{R}^{N_{c}}).$ We will check that $\overline{y}\in D_{\mu}.$ Since for all $n,$ $\overline{y}_{n}\in D_{\mu},$ the definition of $D_{\mu}$ (see \eqref{Dmu}) yields
			\begin{equation}\label{bndyn}
			\begin{array}{l}
			\|\overline{y}_{n}\|_{V^{2,1}(Q_{\infty})\times H^{1}(0,\infty;\mathbb{R}^{N_{c}})}\leqslant \mu. 
			\end{array}
			\end{equation}
			Using the lower semi-continuity of the norms one obtains
			\begin{equation}\label{barymu}
			\begin{array}{l}
			\|\overline{y}\|_{V^{2,1}(Q_{\infty})\times H^{1}(0,\infty;\mathbb{R}^{N_{c}})}\leqslant\mu.
			\end{array}
			\end{equation}
			Now we will verify that 
			\begin{equation}\label{tracey}
			\begin{array}{l}
			y\mid_{\Sigma_{\infty}}=\sum\limits_{j=1}^{N_{c}}w_{j}(t){{g}_{j}}(x)\quad\mbox{for all}\quad (x,t)\in\Sigma_{\infty},
			\end{array}
			\end{equation}
			where ${w_{c}}=(w_{1},...,w_{N_{c}}).$
			From \eqref{bndyn} one has the following weak convergence
			$$y_{n}\rightharpoonup y\quad\mbox{in}\quad L^{2}(0,\infty;H^{2}(\Omega)),\quad\mbox{and}\quad w_{n,c}\rightharpoonup w_{c}\quad\mbox{in}\quad H^{1}(0,\infty;\mathbb{R}^{N_{c}}).$$
			 As the trace operator is linear and bounded from $H^{2}(\Omega)$ onto $H^{3/2}(\Gamma),$ $y_{n}\mid_{\Sigma_{\infty}}$ converges weakly to $y\mid_{\Sigma_{\infty}}$ in $L^{2}(0,\infty;H^{3/2}(\Gamma)).$ 
			On the other hand as $\overline{y}_{n}\in D_{\mu},$ for each $n$ $${y}_{n}\mid_{\Sigma_{\infty}}=\sum\limits_{j=1}^{N_{c}}w_{n,j}(t){{g}_{j}}(x),$$
			where $w_{n,c}=(w_{n,1},...,w_{n,N_{c}}).$ 
			Now since ${w_{n,c}}$ converges weakly to ${w_{c}}$ in $H^{1}(0,\infty;\mathbb{R}^{N_{c}})$ we have the following convergence in the sense of distribution 
			$${y}_{n}\mid_{\Sigma_{\infty}} \xrightarrow[n\rightarrow \infty]{} \sum\limits_{j=1}^{N_{c}}w_{j}(t){{g}_{j}}(x)\quad\mbox{in}\quad \mathscr{D}'(\Sigma_{\infty}).$$ 
			Since the distributional limit and the weak limit (in the space $L^{2}(0,\infty;H^{3/2}(\Gamma))$) of $y_{n}\mid_{\Sigma_{\infty}}$ coincides, one at once obtains the expression \eqref{tracey} of $y\mid_{\Sigma_{\infty}}.$  Also using the continuous embedding $H^{1}(0,\infty)\hookrightarrow L^{\infty}(0,\infty)$ one observes that 
			$$y_{n}\mid_{\Sigma_{\infty}}\stackrel{\ast}{\rightharpoonup}y\mid_{\Sigma_{\infty}}\quad\mbox{in}\quad L^{\infty}(\Sigma_{\infty}).$$
			Hence one has the following by lower semi continuity of norm with respect to the above weak type convergence
			$$\|y\mid_{\Sigma_{\infty}}\|_{L^{\infty}(\Sigma_{\infty})}\leqslant  \frac{L(1-L)}{2}. $$
			Hence $y\mid_{\Sigma_{\infty}}$ satisfies \eqref{concor**}. This finishes the proof of $\bar{y}\in D_{\mu}.$\\[2.mm]
			Step 2.\,\,Now to prove Lemma \ref{l5-4}, it is enough to show that ${V}^{2,1}(Q_{\infty})\times H^{1}(0,\infty;\mathbb{R}^{N_{c}})$ is compactly embedded in $L^{2}(0,\infty,(1+t)^{-1}{dt},{ L^{2}}(\Omega)\times\mathbb{R}^{N_{c}}).$ 
			Let $\{\overline{z}_{n}\}_{n}$ be a sequence in ${V}^{2,1}(Q_{\infty})\times H^{1}(0,\infty;\mathbb{R}^{N_{c}})$ such that
			$$\| \overline{z}_{n}\|_{{V}^{2,1}(Q_{\infty})\times H^{1}((0,\infty);\mathbb{R}^{N_{c}})}\leqslant 1.$$
			This implies that for any $T>0$
			\begin{equation}\label{5-16}
			\begin{array}{l}
			\displaystyle
			\int\limits_{T}^{\infty}{(1+t)^{-2}}{\|\overline{z}_{n}\|_{{L}^{2}(\Omega)\times \mathbb{R}^{N_{c}}}^{2} }dt\leqslant \frac{1}{(1+T)^{2}},
			\end{array}
			\end{equation}
			for all $n\in\mathbb{N}.$
			Let $\epsilon>0.$ Choose $T_{\epsilon}>0$ such that
			$$\frac{1}{(1+T_{\epsilon})^{2}}\leqslant \epsilon.$$
			So using \eqref{5-16} we have 
			\begin{equation}\label{5-17}
			\begin{array}{l}
			\|\overline{z}_{n}-\overline{z}_{m}\|^{2}_{L^{2}(T_{\epsilon},\infty,{(1+t)^{-1}}dt;{L}^{2}(\Omega)\times\mathbb{R}^{N_{c}})}\leqslant 4\epsilon,
			\end{array}
			\end{equation}
			for all $m,n\in \mathbb{N}.$
			\newline
			We know from Rellich's compactness theorem and Aubin-Lions lemma (\cite{aubin}) that the embedding of ${V}^{2,1}(Q_{T_{\epsilon}})\times H^{1}(0,T_{\epsilon};\mathbb{R}^{N_{c}})$ into $L^{2}(0,T_{\epsilon},{L}^{2}(\Omega)\times\mathbb{R}^{N_{c}})$ is compact. Hence up to a subsequence (denoted by the same notation) $\{\overline{z}_{n}\}_{n}$ is Cauchy in $L^{2}(0,T_{\epsilon},{L}^{2}(\Omega)\times\mathbb{R}^{N_{c}}).$
			\newline
			So it follows that there exists $N_{0}\in\mathbb{N}$ such that for all natural numbers $m,n\geqslant N_{0},$
			\begin{equation}\label{5-18}
			\begin{array}{l}
			\|\overline{z}_{n}-\overline{z}_{m}\|^{2}_{L^{2}(0,T_{\epsilon},{(1+t)^{-1}}dt,{ L^{2}(\Omega)}\times\mathbb{R}^{N_{c}})}\leqslant\epsilon.
			\end{array}
			\end{equation}
			Now combining \eqref{5-17}, \eqref{5-18} and a diagonal extraction argument, we can construct a subsequence $\{\overline{z}_{n}\}_{n}$ which is a Cauchy sequence in the Banach space $L^{2}(0,\infty,{(1+t)^{-1}}dt;{L}^{2}(\Omega)\times\mathbb{R}^{N_{c}})$.
			\newline
			The proof is complete.
		\end{proof}

	\begin{lem}\label{l5-5}
		If a sequence $\{\overline{z}_{n}\}$ in $D_{\mu}$ converges weakly to some $\overline{z}$ in ${V}^{2,1}(Q_{\infty})\times H^{1}(0,\infty;\mathbb{R}^{N_{c}})$ then up to a subsequence
		\begin{equation}\label{5-19}
			e^{-\beta t}\overline{z}_{n} \xrightarrow[n\rightarrow \infty]{}e^{-\beta t}\overline{z}\quad\mbox{strongly in}\quad L^{2}(0,\infty;{L}^{2}(\Omega)\times\mathbb{R}^{N_{c}}).
		\end{equation}
	\end{lem}
	\begin{proof}
	The proof follows from the arguments used in proving Lemma \ref{l5-4} and is left to the reader.
	\end{proof}
	\begin{lem}\label{l5-6}
		The map $\chi$ is continuous in $D_{\mu},$ endowed with the norm $L^{2}(0,\infty,{(1+t)^{-1}}dt;{L}^{2}(\Omega))$.
	\end{lem}
		\begin{proof}
			Let $\{\overline{\widehat{y}}_{n}\}_{n}$ $\left( \mbox{where}\,\, \overline{\widehat{y}}_{n}=(
			{\widehat{y}}_{n},{\widehat{w}_{n,c}}
			)\right)$ be a sequence in $D_{\mu}$ and assume that this sequence  $\{\overline{\widehat{y}}_{n}\}_{n}$ strongly converges to $\overline{\widehat{y}}$ $\left( \mbox{where}\,\, \overline{\widehat{y}}=(
			{\widehat{y}},\widehat{w}_{c}
			)\right)$ in the norm $L^{2}(0,\infty,{(1+t)^{-1}}dt;{L}^{2}(\Omega)\times\mathbb{R}^{N_{c}}).$
			\newline
			As for all $n\in\mathbb{N},$ $\|\overline{\widehat{y}}_{n}\|_{V^{2,1}(Q_{\infty})\times H^{1}(0,\infty;\mathbb{R}^{{N_{c}}})}\leqslant\mu,$ up to a subsequence we have the following weak convergence
			\begin{equation}\label{conyn}
			\begin{array}{l}
			\{\overline{\widehat{y}}_{n}\}_{n}\rightharpoonup \overline{\widehat{y}}\quad\mbox{in}\quad V^{2,1}(Q_{\infty})\times H^{1}(0,\infty;\mathbb{R}^{N_{c}})\quad\mbox{as}\quad n\rightarrow\infty.
			\end{array}
			\end{equation}
			Now corresponding to the vector field ${\widehat{y}_{n}},$ let us denote by ${\widehat{\sigma}_{n}}$ the solutions to \eqref{5-2}$_{1}$-\eqref{5-2}$_{3}$. Similarly $\widehat{\sigma}$ is the solution to \eqref{5-2}$_{1}$-\eqref{5-2}$_{3}$ which corresponds to the vector field ${\widehat{y}}.$ As ${\widehat{y}_{n}}$ converges strongly to ${\widehat{y}}$ in the norm  $L^{2}(0,\infty,{(1+t)^{-1}}dt,{L}^{2}(\Omega)),$ for any $T>0,$ ${\widehat{y}_{n}}$ converges to ${\widehat{y}}$ in particular in the norm $L^{1}(Q_{T}).$ Besides, the strong $L^1(\Sigma_T)$ convergence of ${\widehat{y}_{n}} \cdot \vec{n}$ towards ${\widehat{y}} \cdot \vec{n}$ is obvious in view of the identities \eqref{hatybd} and the strong convergence of $\widehat{w}_n$ to $\widehat{w}$ in $L^1(0,T)$, which immediately follows from the weak convergence of $\widehat{w}_n$ to $\widehat{w}$ in $H^1(0,\infty)$.  Hence from Lemma \ref{l4-4}, we obtain that ${\widehat{\sigma}_{n}}$ strongly converges to $\widehat{\sigma}$ in $C^{0}([0,T],L^{q}(\Omega))$ for all $1\leqslant q <+\infty.$ Due to the suitable choice of $\mu$ in Lemma \ref{l5-3}, we can conclude from Theorem \ref{t4-1} (in particular from \eqref{fiessig}) that each of ${\widehat{\sigma}_{n}}$ and $\widehat{\sigma}$ vanishes for $t\geqslant T_{1}.$ So 
			\begin{equation}\label{5-20}
			\begin{array}{l}
			\widehat{\sigma}_{n} \xrightarrow[n\rightarrow \infty]{}\widehat{\sigma}\quad\mbox{strongly in}\quad L^{\infty}(0,\infty;L^{q}(\Omega))\,\,\forall\,\,1\leqslant q<+\infty,
			\vspace{1.mm}\\ 
			\forall n\in\mathbb{N} ,\,{\widehat{\sigma}_{n}} (t)=\widehat{\sigma}(t) =0\quad \mbox{for all}\quad t\geqslant T_{1}.
			\end{array}
			\end{equation}
			Also from \eqref{bdtran} and \eqref{smu0} we know that the $L^{\infty}(Q_{\infty})$ norm of the sequence ${\widehat{\sigma}_{n}}$ is uniformly bounded.
			\newline
			We will now check that $\mathcal{F}(\widehat{y}_{n},{\widehat{\sigma}_{n}})$ converges weakly in $L^{2}(Q_{\infty})$ to $\mathcal{F}(\widehat{y},\widehat{\sigma}).$
			As $({\widehat{y}_{n}},{\widehat{w}_{n,c}})\in D_{\mu},$ from the estimate \eqref{5-9} we obtain a uniform bound for $\| \mathcal{F}({\widehat{y}_{n}},{\widehat{\sigma}_{n}})\|_{L^{2}(Q_{\infty})}.$ So there exists a subsequence of $\mathcal{F}(\widehat{y}_{n},{\widehat{\sigma}_{n}})$ which weakly converges in $L^{2}(0,\infty;{L}^{2}(\Omega)).$ This is therefore enough to show that the sequence $\mathcal{F}({\widehat{y}_{n}},{\widehat{\sigma}_{n}})$ converges to $\mathcal{F}({\widehat{y}},\widehat{\sigma})$ weakly in $\mathscr{D}'(Q_{\infty})$ ($i.e.$ in the sense of distribution).
			\newline
			Let us first check the weak convergence of the term $-e^{-\beta t}{\widehat{\sigma}_{n}}\frac{\partial{\widehat{y}_{n}}}{\partial t}.$ 
			From \eqref{5-20} we know that ${\widehat{\sigma}_{n}}$ strongly converges to $\widehat{\sigma}$ in $L^{2}(Q_{\infty})$ and each of ${\widehat{\sigma}_{n}}$ and $\widehat{\sigma}$ vanishes for all $t\geqslant T_{1}$ (see \eqref{5-20}). Also from \eqref{conyn} we have that $\frac{\partial \widehat{y}_{n}}{\partial t}$ converges weakly to $\frac{\partial \widehat{y}_{n}}{\partial t}$ in $L^{2}(Q_{\infty}).$ Hence their product ${\widehat{\sigma}_{n}}\frac{\partial \widehat{y}_{n}}{\partial t}$ converges weakly to $\widehat{\sigma}\frac{\partial \widehat{y}}{\partial t}$ in $L^{1}(Q_{\infty}).$ So it is now easy to verify that $e^{-\beta t}{\widehat{\sigma}_{n}}\frac{\partial \widehat{y}_{n}}{\partial t}$ converges to $e^{-\beta t}\widehat{\sigma}\frac{\partial \widehat{y}}{\partial t}$ weakly in $L^{1}(Q_{\infty}).$\\
			Now we consider $e^{-2\beta t}({\widehat{y}_{n}}\cdot\nabla){\widehat{y}_{n}}.$ As $\widehat{y}_{n}$ is bounded and weakly convergent to $\widehat{y}$ in ${V}^{2,1}(Q_{\infty}),$ using Lemma \ref{l5-5}, we have
			\begin{equation}\label{5-22}
			\begin{array}{l}
			e^{-2\beta t}{\widehat{y}_{n}} \xrightarrow[n\rightarrow \infty]{} e^{-2\beta t}\widehat{y}\quad \mbox{strongly in}\quad L^{2}(Q_{\infty}), 
			\end{array}
			\end{equation}  
			and
			\begin{equation}\label{5-23}
			\begin{array}{l}
			\nabla{\widehat{y}_{n}} \rightharpoonup\nabla\widehat{y}\quad\mbox{in}\quad L^{2}(Q_{\infty})\quad\mbox{as}\quad\,n\rightarrow\infty.
			\end{array}
			\end{equation} 			
			Therefore $e^{-2\beta t}(\widehat{y}_{n}\cdot\nabla)\widehat{y}$ converges to $e^{-2\beta t}(\widehat{y}\cdot\nabla)\widehat{y}$ weakly in $L^{1}(Q_{\infty}).$
			\newline
			Since $\widehat{y}_{n}$ converges weakly to $y$ in $V^{2,1}(Q_{\infty}),$ one has the following
			$$
			\nabla\widehat{y}_{n}\rightharpoonup \nabla\widehat{y}\quad\mbox{in}\quad L^{\infty}(0,\infty;L^{2}(\Omega))\cap L^{2}(0,\infty;H^{1}(\Omega))\quad\mbox{as}\quad n\rightarrow\infty.
			$$
			We use the interpolation result \cite[Theorem II.5.5]{boy} to obtain the following in particular
			\begin{equation}\label{congy}
			\begin{array}{l}
			\nabla\widehat{y}_{n}\rightharpoonup \nabla\widehat{y}\quad\mbox{in}\quad L^{3}(0,\infty;L^{3}(\Omega))\quad\mbox{as}\quad n\rightarrow\infty.
			\end{array}
			\end{equation}
			Using \eqref{5-20}, \eqref{5-22} and \eqref{congy} one has the following weak convergence
			$$
			e^{-2\beta t}{\widehat{\sigma}_{n}}(\widehat{y}_{n}\cdot\nabla)\widehat{y}_{n}  \rightharpoonup e^{-2\beta t}\widehat{\sigma}(\widehat{y}\cdot\nabla)\widehat{y}\quad\mbox{in}\quad L^{1}(Q_{\infty})\quad\mbox{as}\quad n\rightarrow\infty.
			$$
			The convergences of the remaining terms $e^{-\beta t}{\widehat{\sigma}_{n}}(v_{s}\cdot\nabla)\widehat{y}_{n},$ $e^{-\beta t}{\widehat{\sigma}_{n}}(\widehat{y}_{n}\cdot\nabla)v_{s}$ and $\beta e^{-\beta t}{\widehat{\sigma}_{n}}{\widehat{y}_{n}}$ can be analyzed similarly using the convergences \eqref{conyn} and \eqref{5-20}$_{1}.$ We thus conclude that $\mathcal{F}(\widehat{y}_{n},{\widehat{\sigma}_{n}})$ converges weakly to $\mathcal{F}(\widehat{y},\widehat{\sigma})$ in the space  $\mathscr{D}'(Q_{\infty}).$ Hence this is also the $L^{2}(Q_{\infty})$ weak limit.\\
			From Corollary \ref{c3-8}, we know that for the closed loop system \eqref{114r}, the map
			\begin{equation}\nonumber
			\begin{array}{l}
			\begin{matrix}
			L^{2}(0,\infty;{L}^{2}(\Omega))\times {V}^{1}_{0}(\Omega) & \mapsto &  {V}^{2,1}(Q_{\infty})\times H^{1}(0,\infty;\mathbb{R}^{N_{c}})\\
			({f},{y}_{0}) & \mapsto & \overline{y}
			\end{matrix}
			\end{array}
			\end{equation}
			is linear and bounded. Hence we obtain that $\overline{y}_{n}=\chi(\overline{{\widehat{y}}}_{n})$ ($\overline{y}_{n}=(
			y_{n},
			w_{n,c}
			)$) weakly converges to $\overline{y}=\chi(\overline{{\widehat{y}}})$ ($\overline{y}=(
			y,
			w_{c}
			)$) in $(D_{\mu},\|.\|_{V^{2,1}(Q_{\infty})\times H^{1}(0,\infty;\mathbb{R}^{N_{c}})}).$ Finally as $D_{\mu}$ is compact in $L^{2}(0,\infty;{(1+t)^{-1}}dt,\\{L}^{2}(\Omega)\times\mathbb{R}^{N_{c}})$ (see Lemma \ref{l5-4}), $\overline{y}_{n}$ strongly converges to $\overline{y}$ in $L^{2}(0,\infty;{(1+t)^{-1}}dt,{L}^{2}(\Omega)\times\mathbb{R}^{N_{c}}).$
			The proof of Lemma \ref{l5-6} is complete.
		\end{proof}	

	\subsection{Conclusion}\label{conc}
	Let $\mu$ is as in Lemma \ref{l5-3}. Then\\
	$(i)$\, For an initial datum $(\sigma_{0},y_{0})$ satisfying \eqref{smu0}, the map $\chi$ defined in \eqref{chi} maps $D_{\mu}$ defined in \eqref{Dmu} into itself.\\
	$(ii)$\, The non-empty convex set $D_{\mu}$ is compact in the topology of $L^{2}(0,\infty,{(1+t)^{-1}}dt;{L}^{2}(\Omega)\times\mathbb{R}^{N_{c}})$ (see Lemma \ref{l5-4}).\\
	$(iii)$\, The map $\chi$ is continuous on $D_{\mu},$ endowed with the norm $L^{2}(0,\infty,{(1+t)^{-1}}dt;{L}^{2}(\Omega))$ (Lemma \ref{l5-6}).\\  
	One observes that all the assumptions of Schauder fixed point theorem are satisfied by the map $\chi$ on $D_{\mu},$ endowed with the norm $L^{2}(0,\infty,{(1+t)^{-1}}dt;{L}^{2}(\Omega)\times\mathbb{R}^{N_{c}}).$ Therefore, Schauder fixed point theorem yields a fixed point $(
	{y}_{f},
	w_{{f},c}
	)$ of the map $\chi$ in $D_{\mu}.$ Hence the trajectory $(
	\sigma_{{f}},
	y_{f},w_{f,c}
	)$ solves the non linear problem \eqref{5-2*}. Moreover, as a consequence of Theorem \ref{t4-1} the following holds 
	\begin{equation}\label{stabreq}
	\begin{array}{l}
	\sigma_{{f}}(.,t)=0\quad\mbox{in}\quad\Omega\quad\mbox{for}\quad t\geqslant T_{1}.\\
	\end{array}
	\end{equation}
	Using \eqref{setmu} in \eqref{5-15} and \eqref{mu0}, one further obtains
	\begin{equation}\label{bndywc}
	\begin{array}{l}
	\|({y}_{f},w_{f,c})\|_{{V}^{2,1}(Q_{\infty})\times H^{1}(0,\infty;\mathbb{R}^{N_{c}})}\leqslant
	C\mbox{max}\,\{e^{\beta T_{1}}\|\sigma_{0}\|_{L^{\infty}(\Omega)},\|{y}_{0}\|_{{V}^{1}_{0}(\Omega)}\},
	\end{array}
	\end{equation}
	for some positive constant $C.$ Once again using Theorem \ref{t4-1}, \eqref{bndywc} furnish the following continuous dependence on initial data
	\begin{equation}\label{finalest}
	\begin{array}{l}
	\|(\sigma_{{f}},y_{f})\|_{L^{\infty}(Q_{\infty})\times V^{2,1}(Q_{\infty})}\leqslant C\|(\sigma_{0},y_{0})\|_{L^{\infty}(\Omega)\times V^{1}_{0}(\Omega)},
	\end{array}
	\end{equation}
	for some positive constant $C.$ Now in view of the change of unknowns \eqref{chun}, we obtain the existence of a trajectory $(\rho,v)\in L^{\infty}(Q_{\infty})\times V^{2,1}(Q_{\infty})$ which solves \eqref{1-3} and satisfies the decay estimate \eqref{1-6}. The proof of Theorem \ref{main} is complete.
\end{proof}
\section{Further comments}\label{furcom} Our result considers that the control $u_{c}$ is supported on $\Gamma_{c},$ which is an open subset of the inflow part $\Gamma_{in}$ (see \eqref{dgc}) of the boundary. This is in fact natural to control the inflow boundary of the channel. At the same time we remark that our analysis applies if one wants to control the outflow boundary $\Gamma_{out}$ or the lateral boundary $\Gamma_{0}$ of the channel $\Omega.$ In what follows we briefly discuss these cases.\\
(i) \textbf{Controlling the outflow boundary.} In this case the control zone $\Gamma_{c}$ is an open subset of $\Gamma_{out}.$ After the change of unknowns \eqref{chun}, one can imitate the linearization procedure (as done while transforming \eqref{2-1} into \eqref{2-2}). In this linearized system the transport equation modeling the density \eqref{2-2}$_{1}$-\eqref{2-2}$_{3}$ will remain unchanged but the boundary conditions on the velocity equations \eqref{2-2}$_{4}$-\eqref{2-2}$_{8}$ should be replaced by $y=0$ on $(\Gamma_{0}\cup\Gamma_{in})\times(0,\infty)$ and $y=\sum\limits_{j=1}^{N_{c}}{w_{j}}(t){{g}_{j}}(x)$ on $\Gamma_{out}\times(0,\infty).$ Still the proof of the boundary controllability of the Oseen equations can be carried in a similar way as done in Section \ref{velocity} and in the same spirit of Corollary \ref{p3.0.2}, one can prove that if the initial condition $y_{0}$ and the non-homogeneous term $f$ are suitably small then the inflow and the outflow boundaries of the perturbed vector field $(v_{s}+e^{-\beta t}y)$ coincide with that of $v_{s}.$ Since the transport equation \eqref{2-2}$_{1}$-\eqref{2-2}$_{3}$ remains unchanged in this case, the analysis done in Section \ref{density} applies without any change. The fixed point argument done in Section \ref{final} to prove the stabilization of the coupled system \eqref{1-3} also applies without change.\\
(ii) \textbf{Controlling the lateral boundary.} In this case the control zone $\Gamma_{c}$ is an open subset of $\Gamma_{0}.$ In particular we assume that $\Gamma_{c}\subset\Gamma_{b}$ (where $\Gamma_{b}=(0,d)\times\{0\}\subset \Gamma_{0}$). Now the inflow and outflow boundaries of the velocity vector $(e^{-\beta t}y+v_{s})$ cannot be characterized by using the notations $\Gamma_{in}$ and $\Gamma_{out}$ (as defined in \eqref{1-2}), since $\Gamma_{c}$ can contain an inflow part and an outflow part and one can not prove a result similar to Corollary \ref{p3.0.2}. More precisely here we can use the following notations for time $t>0,$
\begin{equation}\label{infl}
\left\{ \begin{array}{l}
\Gamma^{*}_{in,y}(t)=\Gamma_{in}\cup \{x\in\Gamma_{c}\suchthat( v_s (x)+ e^{-\beta t} y(x,t)\cdot n(x))<0\}\subset \Gamma_{in}\cup\Gamma_{b},\\
\Gamma_{h}=(0,d)\times\{1\}.
\end{array}\right.
\end{equation} 
In a similar way as we have obtained \eqref{2-2} from \eqref{1-3}, one gets the following system
\begin{equation}\label{2-2*}
\left\{\begin{array}{ll}
\displaystyle
\frac{\partial \sigma}{\partial t}+(({v}_s+e^{-\beta t}{y}) \cdot \nabla)\sigma-\beta\sigma=0\quad &\mbox{in} \quad Q_{\infty},
\vspace{1.mm}\\
\sigma (x,t)=0 \quad& \mbox{on} \quad \bigcup_{t\in(0,\infty)}(\Gamma^{*}_{in,y}(t) \times\{t\}),
\vspace{1.mm}\\
\sigma (x,0)=\sigma_0\quad&\mbox{in}\quad\Omega,\\[1.mm]
\displaystyle \frac{\partial {y}}{\partial t}-\beta {y}-\nu \Delta {y}+ ({v}_s \cdot \nabla){y}+({y} \cdot \nabla){v}_s +\nabla  q={f}\quad &\mbox{in}\quad Q_{\infty},
\vspace{2.mm}\\
\mbox{div}{y}=0\quad& \mbox{in}\quad Q_{\infty},
\vspace{2.mm}\\
{y}=0\quad& \mbox {on} \quad (\Gamma_{in}\cup \Gamma_{h}\cup \Gamma_{out}) \times (0,\infty),\\[1.mm]
{y}=\sum\limits_{j=1}^{N_{c}}{w_{j}}(t){{g}_{j}}(x) \quad& \mbox{on} \quad \ \Gamma_{b} \times (0,\infty),
\vspace{1.mm}\\
{y}(x,0)={y}_0\quad&\mbox{in}\quad\Omega.
\end{array}\right.
\end{equation}
One can use arguments similar to the ones in Section \ref{velocity} in order to stabilize $y$ solving \eqref{2-2*}$_{4}$-\eqref{2-2*}$_{8}.$ The functions ${g_{j}}$ can be constructed with compact support in $\Gamma_{b}$ (imitating the construction \eqref{basisU0}), and we can recover the $C^{\infty}$ regularity of the boundary control and $V^{2,1}(Q_{\infty})$ regularity of $y.$ Hence the flow corresponding to the vector field $(e^{-\beta t}y+v_{s})$ is well defined in classical sense, consequently one can adapt the arguments used in Section \ref{density} to prove that $\sigma,$ the solution of \eqref{2-2*}$_{1}$-\eqref{2-2*}$_{3}$ belongs to $L^{\infty}(Q_{\infty})$ and vanishes after some finite time provided the initial condition $\sigma_{0}$ is supported away from the lateral boundaries and $y$ is small enough. The use of a fixed point argument to prove the stabilizability of the solution of \eqref{1-3} is again a straightforward adaptation of the arguments used in Section \ref{final}.   
\bibliographystyle{plain}
\bibliography{bibliography}

\end{document}